\documentclass[aop,preprint]{imsart}

\usepackage{amsthm,amsmath}
\RequirePackage[numbers]{natbib}
\RequirePackage[colorlinks,citecolor=blue,urlcolor=blue]{hyperref}

\arxiv{arXiv:1706.10251}

\startlocaldefs

\usepackage{latexsym,amssymb,amsfonts,mathrsfs}
\usepackage{enumitem}

\newtheorem{theorem}{Theorem}
\newtheorem{lemma}[theorem]{Lemma}
\newtheorem*{lemma*}{Lemma}
\newtheorem{proposition}[theorem]{Proposition}
\newtheorem{corollary}[theorem]{Corollary}

\newtheorem{remark}[theorem]{Remark}
\newtheorem{assumption}[theorem]{Assumption}
\newtheorem*{assumption*}{Assumption}

\newtheorem*{fact*}{Fact}

\usepackage[utf8]{inputenc}
\usepackage[english]{babel}

\newcommand{\B}{\mathcal{B}}

\newcommand{\Prob}{\mathbb{P}}

\newcommand{\ind}{\mathbf{1}}

\renewcommand{\Re}{\mathrm{Re} \,}

\newcommand{\J}{\mathrm{J}}

\newcommand\N{{\mathbb N}}

\newcommand\R{{\mathbb R}}

\newcommand\PP{{\mathbb P}}
\newcommand\E{{\mathbb E}}

\newcommand{\scl}{\mu_{\mathrm{sc}}}
\newcommand{\Ucal}{\mathcal{U}}
\newcommand{\Rcal}{\mathcal{R}}
\newcommand{\pv}{\mathrm{pv}}

\newcommand{\Hi}{\mathscr{H}}
\newcommand{\Ro}{\mathcal{R}}

\newcommand{\g}{\mathrm{g}}
\newcommand{\K}{\mathrm{K}}

\newcommand{\Co}{\mathscr{C}}
\newcommand{\U}{\mathcal{U}}
\newcommand{\m}{\mathbf{m}}
\renewcommand{\S}{\boldsymbol{\Sigma}}
\newcommand{\y}{\mathrm{y}}

\newcommand{\Vreg}{\kappa+3}
\newcommand{\freg}{\kappa+4}

\numberwithin{equation}{section}
\numberwithin{theorem}{section}
\newcommand\beq{\begin {equation}}
\newcommand\eeq{\end {equation}}
\newcommand\beqs{\begin {equation*}}
\newcommand\eeqs{\end {equation*}}

\endlocaldefs

\begin{document}

\begin{frontmatter}
\begin{aug}
\title{Quantitative normal approximation \\of linear statistics of $\beta$-ensembles}
\runtitle{CLT for $\beta$-ensembles}

\author{\fnms{Gaultier} \snm{Lambert}\corref{}\ead[label=e1]{gaultier.lambert@math.uzh.ch}\thanksref{t1,t3,m1}}
\author{\fnms{Michel} \snm{Ledoux}\ead[label=e2]{ledoux@math.univ-toulouse.fr}\thanksref{m2}}
\and
\author{\fnms{Christian} \snm{Webb}\ead[label=e3]{christian.webb@aalto.fi}\thanksref{t1,t2,m3}}
\thankstext{t1}{G.L. and C.W. wish to thank the University of Toulouse -- Paul-Sabatier, France, for its hospitality during a visit when part of this work was carried out.}
\thankstext{t3}{G.L. was supported by the University of Zurich Forschungskredit grant  FK-17-112.} 
\thankstext{t2}{C.W. was supported by the Academy of Finland grants 288318 and 308123.}

\runauthor{G. Lambert, M. Ledoux, and C. Webb}

\affiliation{Universit\"at Z\"urich\thanksmark{m1}, Universit\'e de Toulouse\thanksmark{m2} and Aalto University\thanksmark{m3}}

\address{Institut f\"ur Mathematik,\\
Universit\"at Z\"urich,\\
Winterthurerstrasse 190. \\
CH-8057 Z\"urich, Switzerland\\
\printead{e1}}

\address{Institut de Math\'ematiques de Toulouse,\\
 Universit\'e de Toulouse -- Paul-Sabatier,\\
  F-31062 Toulouse, France\\
  \& Institut Universitaire de France\\
\printead{e2}}

\address{Department of mathematics and systems analysis,\\
Aalto University, \\
P.O. Box 11000, 00076 Aalto, Finland\\
\printead{e3}}
\end{aug}

\begin{abstract}.
We present a new approach, inspired by Stein's method, to prove a central limit theorem (CLT) for linear statistics of $\beta$-ensembles in the one-cut regime. Compared with the previous proofs, our result requires less regularity on the potential and provides a rate of convergence in the quadratic Kantorovich or Wasserstein-2 distance. The rate depends both on the regularity of the potential and the test functions, and we prove that it is optimal in the case of the Gaussian Unitary Ensemble (GUE) for certain polynomial test functions. 

The method relies on a general normal approximation result of independent
interest which is valid for a large class of Gibbs-type distributions. In the context of $\beta$-ensembles,  this leads to a multi-dimensional CLT for a sequence of linear statistics which are approximate eigenfunctions of the infinitesimal generator of \emph{Dyson Brownian motion} once the various error terms are controlled using the rigidity results of Bourgade, Erd\H{o}s and Yau. 
\end{abstract}



\end{frontmatter}

\section{Introduction and main results}

In this article, we study linear statistics of $\beta$-ensembles. By a $\beta$-ensemble, we mean a probability measure on $\R^N$ of the form
\beq \label{eq:betaens}
\PP_{V,\beta}^N(d\lambda_1,\ldots,d\lambda_N)
 \, = \, \frac{1}{Z_{V,\beta}^N} \, e^{-\beta \mathcal{H}_{V}^N(\lambda_1,\ldots,\lambda_N)}\prod_{j=1}^N d\lambda_j
\eeq where
\begin{equation*}
\mathcal{H}_{V}^N(\lambda) 
    \, = \, \sum_{1\leq i<j\leq N}\log \frac {1}{|\lambda_i-\lambda_j|} +N\sum_{j=1}^N V(\lambda_j) .
\end{equation*} 
The parameter $\beta>0$ is interpreted as the inverse temperature, $V:\R\to \R$ is a suitable function known as  the confining potential,  and $Z_{V,\beta}^N$ is a normalization constant, known as the partition function. By a linear statistic, we mean a random variable of the form $\sum_{j=1}^N f(\lambda_j)$, where the configuration $(\lambda_j)_{j=1}^N$ is drawn from $\PP_{V,\beta}^N$ and $f:\R\to\R$ is a test function. Our goal is to study the asymptotic fluctuations of these linear statistics in the large $N$ limit.

A general reference about the asymptotic behavior of $\beta$-ensembles is \cite[Section 11]{PS11}. 
It is a well known fact that when $N$ is large, a linear statistic is close to its mean or equilibrium value which is given by $N \int f(x) \mu_V(dx)$ where $\mu_V$, the so-called equilibrium measure associated to $V$, is a probability measure with compact support on $\R$ and can be characterized as being the unique solution of a certain optimization problem that we recall below -- see the discussion around \eqref{eq:el1}. Moreover, under suitable conditions on $V$ and $f$, it is known that the fluctuations of $\sum_{j=1}^Nf(\lambda_j)$ around the equilibrium value are of order 1 and described by a certain Gaussian random variable. This CLT first appeared in the fundamental work of Johansson, \cite[Theorem 2.4]{J2} valid for suitable polynomial potentials $V$. The argument is inspired from statistical physics and based on the asymptotic expansion of the partition function $Z_{V,\beta}^N$ and relies on the analysis of the first {\it loop equation} by the
methods of the perturbation theory. The method was further developed and simplified in the subsequent
papers \cite[Theorem 1]{KS}, \cite[Proposition 5.2]{BG13} and \cite[Theorem 1]{Shcherbina13}, but still under the assumption that the potential $V$ is real-analytic in a neighborhood of the support of the equilibrium measure. We also point out the article \cite{BLS17} which appeared nearly simultaneously with this article, and where the assumption of real-analyticity is relaxed by using a priori estimates from \cite{LS17}. 
 In particular they obtain a CLT which is valid when the potential $V\in \Co^5(\R)$ for test functions $f\in \Co_c^3(\R)$ -- these assumptions being slightly weaker than Assumption \ref{as:1cut}, which covers the assumptions under which we prove our main result, Theorem~\ref{th:main}. Their results are also stronger in the sense that they obtain the CLT by proving convergence of the Laplace transform of the law of a linear statistic and they also provide sufficient conditions to obtain Gaussian fluctuations in the multi-cut setting even when the potential is not regular -- that is beyond the context of our Assumption \ref{as:1cut}. We shall come back to their method and compare it with ours in Section~\ref{sec:BLS}. 
 In the special case $\beta=2$, there are also other types of proofs which rely on the determinantal structure of the ensembles $\PP_{V,2}^N$ and are valid in greater generality, see e.g.~\cite{BD17} or \cite{L}. 
 \medskip

Our goal is to offer a new proof for this CLT with an argument inspired by Stein's method -- 
see Theorem~\ref{th:main}. In addition to novelty, the benefits of our approach are a rate of convergence (in the Kantorovich or Wasserstein distance $\mathrm{W}_2$ -- see the discussion around \eqref{eq:W2} for a definition) as well as requiring less regularity of the potential $V$ than the more classical proofs.
We also demonstrate that at least in some special cases (e.g. the GUE with certain polynomial linear statistics), the rate of convergence is optimal. Compared to the references above, the drawback of our approach is that we are not able to provide estimates for exponential moments of linear statistics -- or in other words, strong asymptotic estimates for the partition function of a $\beta$-ensemble, although we do get convergence of the first and second moment of a linear statistic of a nice enough function. This being said, it does not seem to be obvious how one would obtain our rate of convergence in the Kantorovich distance from e.g. the Laplace transform estimates from \cite{BLS17}.

In this section, we will first introduce some key concepts and then state our main results concerning the precise asymptotic behavior of $\sum_{j=1}^N f(\lambda_j)$. After this, we offer a brief comparison of our approach to other approaches to the CLT for linear statistics of $\beta$ ensembles, and then conclude this introduction with an outline of the rest of the article. Before proceeding, we describe some notation that we will use repeatedly throughout the article.

\medskip

{\bf Notational conventions and basic concepts.} For two sequences $(A_n)$ and $(B_n)$, we will write $A_n\ll B_n$ if there exists a constant $C>0$ independent of $n$ such that $A_n\leq C B_n$ for all $n$. When this constant depends on some further parameter, say $\epsilon$, we write $A_n\ll_\epsilon B_n$. If $A_n\ll B_n$ and $B_n\ll A_n$, we write $A_n\asymp B_n$ (and $\asymp_\epsilon$ if the constants depend on the parameter $\epsilon$). We will also use analogous notation for functions: e.g.~$f\ll g$ if there exists a constant $C$ independent of $x$ such that $f(x)\leq C g(x)$ for all $x$. In most of our estimates, these constants will depend on the parameters $\beta$ and $V$ in the definition  of $\PP_{V,\beta}^N$, but as these parameters remain fixed throughout the article, we shall not emphasize these dependencies and simply write $\ll$ instead of $\ll_{V,\beta}$ etc.

We write $\N=\lbrace 0,1, \ldots\rbrace$ and for $k\in \N$ as well as for any open interval
${\mathrm {I}\subseteq \R}$,
 we write $\Co^k(\mathrm{I})$ for the space of functions which are $k$-times continuously differentiable on $\mathrm {I}$. If $\mathrm {I}$ is a finite closed interval, we also write $\Co^k(\mathrm {I})$ to indicate that the $k^{\mathrm{th}}$-derivative of the functions have finite limits at the end points
of~$\mathrm {I}$. For any $k\in\N$, $\alpha\in(0,1)$, and interval $\mathrm {I} \subset \R$, we write
$\Co^{k,\alpha}(\mathrm {I})$ for the space of functions $f\in \Co^k(\mathrm {I})$ which satisfy 
\beq\label{eq:hnorm}
{\| f \|}_{\Co^{k,\alpha}(\mathrm {I})}
      \, = \,  \sup_{x,y\in \mathrm {I}, x\neq y}\frac{|f^{(k)}(x)-f^{(k)}(y)|}{|x-y|^\alpha} \, < \, \infty.
\eeq
We also use the notation $\Co^k_c(\mathrm {I})$ for the subspace of $\Co^k(\mathrm {I})$ whose elements have compact support in $\mathrm {I}$. For a function $f:\R\to \R$, we also find it convenient to write 
\beq\label{eq:inftynorm}
{\|f\|}_{\infty,\mathrm {I}} \, = \, \sup_{x\in \mathrm {I}}|f(x)|.
\eeq

We will impose several constraints on the potential $V$ (see Assumption~\ref{as:1cut}) -- one being that it is normalized so that the support of the equilibrium measure $\mu_V$ is $[-1,1]$ (the actual constraint here being that the support is a single interval -- normalizing this interval to be $[-1,1]$ can always be achieved by scaling and translating space). We thus find it convenient to introduce the following notation: let $\J=(-1,1)$ and
$\J_\epsilon=(-1-\epsilon,1+\epsilon)$ for any $\epsilon>0$.  We also introduce notation for the semicircle law and arcsine law on $\J$:
\beq\label{eq:sclvarrho}
\scl(dx) \, = \, \frac{2}{\pi}\sqrt{1-x^2} \, \ind_{\J}(x) \, dx \quad \mbox{and}
    \quad \varrho(dx) \, = \, \frac{\ind_{\J}(x)}{\pi\sqrt{1-x^2}} \, dx . 
\eeq

\noindent If $\lambda = (\lambda_1,\dots,\lambda_N)$ is a configuration distributed according to \eqref{eq:betaens}, we define
\beq\label{eq:nuv}
\nu_N(dx) \, = \, \sum_{j=1}^N \delta_{\lambda_j}(dx)-N\mu_V(dx).
\eeq
This measure corresponds to the centered empirical measure when $\beta=2$ and, in general, it describes the fluctuation of the configuration $\lambda$ around equilibrium -- we suppress the dependence on $V$ here. 

A further notion that is instrumental in the statement of our main results is the quadratic Kantorovich or Wasserstein-2 distance between two probability distributions. For $\mu$ and $\nu$ probability distributions on $\R^d$ with finite second moment (so $\int_{\R^d}|x|^2\mu(dx)<\infty$ and similarly for $\nu$), we write
\beq\label{eq:W2}
\mathrm{W}_2(\mu,\nu) \, = \, \inf_{\pi}\left(\int_{\R^d \times \R^d}|x-y|^2 \pi(dx,dy)\right)^{\frac 12}
\eeq
where the infimum is over all probability measures $\pi$ on
$\R^d \times \R^d$ with marginals $\mu$ and $\nu$, that is over all couplings $\pi$ so that $\pi(dx,\R)=\mu(dx)$ and $\pi(\R,dy)=\nu(dy)$.  Here, $| \cdot |$ is the Euclidean distance on $\R^d$.
 A basic fact about the quadratic Kantorovich distance is that convergence with respect to it implies convergence in distribution along with convergence of the second absolute moment. We refer to \cite[Chapter 6]{V09} for further information about Kantorovich distances. If $\mu$ is the distribution of a random variable $X$ and $\nu$ the distribution of a random variable $Y$, we will occasionally find it convenient to write $\mathrm{W}_2(X,Y)$ or $\mathrm{W}_2(X,\nu)$ instead of $\mathrm{W}_2(\mu,\nu)$.

Finally, for any $d \geq 1$
we will denote by $\gamma_d$ the standard Gaussian law on $\R^d$: 
$$
\gamma_d(dx) \, = \, \frac{1}{(2\pi)^{\frac d2}}\prod_{j=1}^d e^{- \frac 12 x_j^2 }dx_j.
$$

\subsection{Main results}

In addition to smoothness, we will need some further conditions on $V$. These are slightly indirect as they are statements about the equilibrium measure associated to $V$. 
We recall that if $V$ is say continuous and grows sufficiently  fast  at infinity, its equilibrium measure $\mu_V$ is defined as the unique minimizer of the energy functional 
$$
\mu \, \mapsto  \, \frac{1}{2}\int_{\R\times \R} \log \frac {1}{|x-y|} \, \mu(dx)\mu(dy)+\int_{\R} V(x)\mu(dx)
$$
over all probability measures on $\R$. It turns out that $\mu_V$ has compact support and is characterized by the following two conditions:
\begin{itemize}[leftmargin=0.5cm]
\item There exists a constant $\ell_V$, such that for $\mu_V$-almost every $x\in \mathrm{supp}(\mu_V)$,
\beq\label{eq:el1}
Q(x) \, := \,  V(x)-\int_\R \log |x-y| \, \mu_V(dy) \, =\, \ell_V . 
\eeq
\item $Q(x)\geq \ell_V$ for all $x\notin \mathrm{supp}(\mu_V)$.
\end{itemize} 
 We refer to \cite[Section 2.6]{AGZ10} for further details.

\medskip

It is well known that the assumption that the support of $\mu_V$ is a single interval is necessary to observe Gaussian fluctuations for all test functions. In the multi-cut regime -- namely when the support consists of several intervals --  there is a subspace of smooth functions where the fluctuations are Gaussian, but for a general smooth test function the fluctuations can oscillate with the dimension $N$ and there is no CLT in general. This is due to the possibility of eigenvalues tunnelling from one component of the support to another.  See e.g.~\cite{Pastur06} for a heuristic description of this phenomenon, or \cite[Theorem 2]{Shcherbina13} and \cite[Section 8.3]{BG2} for precise results as well as \cite{BLS17} for a description of when one observes Gaussian fluctuations. 
When the potential $V$ is sufficiently smooth, the behavior of the equilibrium measure is well-understood. In fact, by \cite[Theorem~11.2.4]{PS11}, if one normalizes $V$ so that
$\mathrm{supp}(\mu_V) =\overline{\J} = [-1,1]$ (one can check that this is equivalent to the conditions (11.2.17) in \cite[Theorem~11.2.4]{PS11}), we have 
\beq \label{eq:Sdensity}
\mu_V(dx) \, = \,  S(x)\scl(dx) 
\eeq
 where $S:\overline{\J}\to [0,\infty]$ is given by 
\beq\label{eq:Sdef}
S(x) \, = \, \frac{1}{2} \int_{\J}\frac{V'(x)-V'(y)}{x-y} \, \varrho(dy). 
\eeq
Note that this formula can be used to define $S$ outside of $\overline{\J}$.
Moreover, one can check (see e.g.~Lemma~\ref{lem:derivative} in the Appendix) that this expression shows that if $V\in \Co^{\kappa+3}(\R)$, then the extension given by \eqref{eq:Sdef} satisfies $S\in  \Co^{\kappa+1}(\R)$ for any $\kappa\in\N$.

Finally, we assume that $S(x)>0$ for all $x\in\bar\J$.  In the random matrix theory literature, this is known as the {\it off-criticality condition}. 
Combining these remarks, we now state our standing assumptions about~$V$.

\begin{assumption}\label{as:1cut}
Let $V\in \Co^{\Vreg}(\R)$ be such that
${\liminf_{x\to \pm\infty}\frac{V(x)}{\log |x|}>1}$ and $\inf_{x\in \R}V''(x)>-\infty$. 
Moreover, assume that the support of the measure $\mu_V$ is $\overline{\J}=[-1,1]$ and $S(x)>0$ for all $x\in\overline{\J}$. 
\end{assumption}

There is a rather large class of potentials which satisfy these assumptions. For instance any function $V$ that is sufficiently convex on $\R$.  A concrete example being $V(x)=x^2$, for which \eqref{eq:betaens} is the so-called Gaussian or Hermite $\beta$-ensemble. In particular, it is a well known fact that  $\Prob_{x^2,2}^N$ is the distribution of the eigenvalues of a $N\times N$ GUE random matrix. 

\medskip

As a final step before stating our main result, we introduce some notation concerning the limiting mean and variance of a linear statistic. Let ${f\in \Co^1_c(\R)}$ and define
\beq \begin {split} \label{mean}
\m(f) & \,  = \,  \frac{1}{2}\left(f(1)+f(-1)\right)-\int_{\J} f(x)\varrho(dx)\\
&\quad \, \, -\frac{1}{2}\int_{\J} \frac{S'(x)}{S(x)}\left[\int_{\J}\frac{f(x)-f(y)}{x-y} \, \varrho(dy)\right]\scl(dx)\\   
\end{split} \eeq
and
\beq\label{variance}
\S(f) \, = \,  \frac{1}{4}\iint_{\J\times\J}\left(\frac{f(x)-f(y)}{x-y}\right)^2(1-xy)\varrho(dx)\varrho(dy)
\eeq
where the measures $\scl$ and $\varrho$ are as in \eqref{eq:sclvarrho}. As we see below,  $N\int_{\J} fd\mu_V +(\frac{1}{2}-\frac{1}{\beta})\m(f)$ describes the mean of a linear statistic (a fact first noted in \cite{Shcherbina13}) and $\frac{2}{\beta}\S(f)$ its variance. There are different possible expressions for the variance $\S$. For example, one can check that for $f\in \Co^1_c(\R)$,
\beq\label{variance_2}
\S(f) \, = \,  \frac{1}{4}\sum_{k=1}^\infty k f_k^2 ,
\eeq
where $f_k$, $k \geq 1$, denote the Fourier-Chebyshev coefficients of the function $f$, see  \eqref{Fourier_2}. We also refer to \eqref{Sigma_1} for yet another formula for $\S(f)$. We wish to emphasize here the fact (well known since \cite{J2}) that while $\mu_V$ and $\m(f)$ depend on $V$, the variance $\S(f)$ is independent of $V$ -- or in other words, the fluctuations of the linear statistics are universal.

\medskip

We now turn to the statement of our main result.

\begin{theorem}\label{th:main}
Let $V:\R\to \R$ satisfy Assumption~\ref{as:1cut} with $\kappa\ge 5$ and let $f\in \Co_c^{{\freg}}(\R)$ be a given function which is non-constant on $\J$. Moreover, let $(\lambda_1,\ldots,\lambda_N)$ be distributed according to \eqref{eq:betaens} and 
\beq \begin {split} \label{eq:linstat}
\mathcal{X}_N(f) &  \, := \,  \sqrt{\frac{\beta}{2\S(f)}}  \left(\int_\R f(x)\nu_N(dx)
    -\Big (\frac{1}{2}-\frac{1}{\beta}\Big )\m(f)\right)\\
 & \, = \,  \sqrt{\frac{\beta}{2\S(f)}}  \bigg( \sum_{j=1}^N f(\lambda_j)-N\int_{\J} f(x)\mu_V(dx) - \Big (\frac{1}{2}-\frac{1}{\beta}\Big )\m(f) \bigg) . 
\end{split} \eeq 
Then the $\mathrm{W}_2$-distance between the law of the random variable $\mathcal{X}_N(f)$ and  the standard Gaussian measure on $\R$ satisfies for any $\epsilon>0$, 
$$
\mathrm{W}_2\big (\mathcal{X}_{N}(f), \gamma_1 \big)  \, \ll_f \, N^{-\theta+\epsilon} , 
$$
where {$\theta = \min\big\{ \frac{2\kappa-9}{2\kappa+11}  , \frac{2}{3} \big\}$}.  In particular, as $N\to\infty$, $\sum_{j=1}^N f(\lambda_j)-N\int_{\J} fd\mu_V$ converges in law to a Gaussian random variable with mean $(\frac{1}{2}-\frac{1}{\beta})\m(f)$ and variance $\frac{2}{\beta} \, \S(f)$.
\end{theorem}

\begin{remark}\label{rk:rate}
The function $\theta(\kappa)$ is non-decreasing in $\kappa$ and positive if  $\kappa\ge 5$. This means that the rate of convergence in Theorem~\ref{th:main} is improving with the regularity of the potential $V$ as well as the regularity of the test function $f$. Moreover, we point out that the condition $\theta \le \frac 23$ comes only from the fluctuations  of the eigenvalues near the edges of the spectrum. 
 For instance, we deduce from the proof of Theorem~\ref{th:main} $($see Section~\ref{sect:proof_main}$)$  that if $f\in \Co_c^{{\infty}}(\J)$, then we could take $\theta=1$ which,  according to
Theorem~\ref{th:GUE} below, leads to the optimal rate $($at least for such test functions $f)$ up to a factor of $N^\epsilon$ for an arbitrary $\epsilon>0$. 
\end{remark}

One interpretation of Theorem~\ref{th:main} is that, if centered and normalized to have a $\beta$-independent covariance, the empirical measure $\sum_{j=1}^N \delta_{\lambda_j}$, viewed as a random generalized function, converges say in the sense of finite dimensional distributions to a mean-zero Gaussian process $\mathcal{X}$ which is an element of a suitable space of generalized functions and has the covariance structure
\begin{equation*}
\E\big[ \mathcal{X}(f) \mathcal{X}(g) \big]  \, = \,  \S(f,g)  \, := \, 
\frac{1}{4}\sum_{k\ge 1} k f _k g_k .
\end{equation*}
In fact, one can understand $\mathcal{X}$ as being the (distributional) derivative of the random generalized function 
$$
\mathcal{Y}(x) \, = \, \sum_{k=0}^\infty \frac{1}{\sqrt{k+1}} \, X_k U_k(x)\frac{2}{\pi}\sqrt{1-x^2} \, ,
$$
where $X_k$ are i.i.d. standard Gaussians, $x\in \J = (-1,1)$, $U_k$ are Chebyshev polynomials of the second kind (see Appendix \ref{app:hilbert} for their definition), and the interpretation is that one should consider this as a random element of a suitable space of generalized functions. With some work (we omit the details), one can then check that the covariance kernel of the process $\mathcal{Y}$ is given by
$$
\E\left[\mathcal{Y}(x)\mathcal{Y}(y)\right]
  \,= \, -\frac{4}{\pi^2}\log \bigg ( \frac{|x-y|}{1-xy+\sqrt{1-x^2}\sqrt{1-y^2}} \bigg),
    \quad x,y\in \J .
$$
We also point out that one can check formally (say using Lemma~\ref{le:Ucheb}) that the Hilbert transform of the process $\mathcal{Y}$ is a centered Gaussian process whose covariance kernel is proportional to $-\log 2|x-y|$ on $\J\times \J$ -- which is the covariance kernel of the $2d$-Gaussian free field restricted to $\J$.

\medskip

Returning to the statement of Theorem~\ref{th:main}, we note that the assumption of $f$ having compact support is not essential. In view of the large deviation principle for the extreme values of the configuration $(\lambda_j)_{j=1}^N$ (see e.g.~\cite[Proposition~2.1]{BG13} or \cite[section~2.6.2]{AGZ10}), one could easily replace compact support by  polynomial growth rate at infinity and the assumption of $f$ being smooth on $\R$ by $f$ being smooth on $(-1-\epsilon,1+\epsilon)$ for some $\epsilon>0$ and some very mild regularity assumptions outside of this interval. We choose not to focus on these classical generalizations here and be satisfied with the assumption that $f$ is smooth and compactly supported.

\medskip

We also wish to point out a variant of Theorem~\ref{th:main} for polynomial  linear statistics of the GUE to emphasize that our approach can yield, at least in some cases, an optimal rate for normal approximation
in the $\mathrm {W}_2$-metric.

\begin{theorem}\label{th:GUE}
Let $f:\R\to \R$ be a non-constant polynomial $($independent of $N)$ and $\nu_N$ be as in \eqref{eq:nuv} where $(\lambda_1, \ldots,\lambda_N)$ is distributed according to the ensemble $\PP_{x^2,2}^N$. Then,
as $N\to\infty$,
$$
\mathrm{W}_2\bigg ( \frac {1} {\sqrt {\S(f)}} \int_\R f(x) \nu_N(dx) ,\gamma_1\bigg)
  \, \ll_f \,  \frac 1N \, .
$$
Moreover, if $f(x)=x^2$, the result is sharp: 
$$
\mathrm{W}_2\bigg (\frac {1} {\sqrt {\S(f)}} \int_\R x^2 \nu_N(dx) ,\gamma_1\bigg )
  \, \asymp \,  \frac 1N \, .
$$
\end{theorem}

Our method of proof of Theorem~\ref{th:main} (and Theorem~\ref{th:GUE}) is to first use a Stein's method type argument to prove a general normal approximation result
of independent interest -- Proposition~\ref{prop.main} -- giving a bound on the $\mathrm{W}_2$-distance between the law of an essentially arbitrary function of some collection of random variables and a standard Gaussian distribution. As this bound is true in such generality, it is naturally not a very useful one in most cases. What is of utmost importance to our approach is that if this function is chosen to be \textit{an approximate eigenfunction} of a certain differential operator associated to the distribution of the random variables (e.g. in the setting of the Gaussian
$\beta$-ensembles, this operator is the infinitesimal generator of Dyson Ornstein-Uhlenbeck process), then this bound on the $\mathrm{W}_2$-distance becomes quite precise. We then proceed to prove that in the setting of the $\beta$-ensembles, there are many such approximate eigenfunctions -- in fact enough that linear statistics of nice enough functions can be approximated in terms of linear combinations of these approximate eigenfunctions and a multi-dimensional CLT for the approximate eigenfunctions implies a CLT for linear statistics of smooth enough functions. 

Related reasoning has appeared elsewhere in the literature as well. For example in \cite{L12}, a normal approximation result was found for functions that are exact eigenfunctions of the relevant differential operator. Also we wish to point out that our normal approximation result can be viewed as a
$\mathrm{W}_2$-version
of the multivariate normal approximation in the Kantorovich metric $\mathrm{W}_1$ developed
by Meckes in \cite {Meckes} and relying on exchangeable pairs. For the application to linear statistics associated to classical compact groups and circular $\beta$-ensembles, studied in \cite{DS1,DS2,F,W16}, the pairs are created through circular Dyson Brownian motion. Moreover, while this is not explicitly emphasized in these works, the CLT is actually proven for certain approximate eigenfunctions of the infinitesimal generator of the circular Dyson Brownian motion.

These functions and the multi-dimensional CLT for the corresponding linear statistics are of such central importance to our proof of Theorem~\ref{th:main} that despite their properties being slightly technical we wish to formulate general results about them in this introduction.

\begin{theorem}\label{th:basis}
Let $V$ be as in Assumption~\ref{as:1cut} and
{$\eta >  \frac {4(\kappa+1)}{2\kappa-1}$}. For small enough  $\delta>0$ depending only on $V$,  there exists a sequence of functions $(\phi_n)_{n=1}^\infty$ such that for $\epsilon_n= \delta n^{-\eta}$, $\phi_n\in \Co^\kappa_c(\J_{2\epsilon_n})$ for all $n\geq 1$,  
and they satisfy
\beq\label{eq:eveq}
V'(x) \, \phi_n'(x)-\int_\J\frac{\phi_n'(x)-\phi_n'(y)}{x-y} \, \mu_V(dy)
   \, = \, \sigma_n \phi_n(x)
\eeq
on $\J_{\epsilon_n}$, where $\sigma_n=\frac{1}{2\S(\phi_n)}>0$.
Moreover, we have the estimate $\sigma_n\asymp n$ and the bounds for any $0\le k<\kappa$,
\beq\label{eq:phinderb}
{\|\phi_n^{(k)}\|}_{\infty,\R} \, \ll \, n^{k\eta} .
\eeq

Finally, restricted to $\J$, the functions $\phi_n$ form an orthonormal basis of $\Hi_{\mu_V}$  $($see the discussion around \eqref{eq:sobo} for its definition$)$ and for any $f$ in $\Co^{\kappa+4}(\R)$, we can expand $f(x)=\widehat{f}_0+\sum_{n=1}^\infty \widehat{f}_n \phi_n(x)$ for all $x\in\overline{\J}$, 
where the Fourier-$\phi$ coefficients, given by $\widehat{f}_0=\int_{\J}f d\varrho$ and $\widehat{f}_n=\langle f,\phi_n\rangle_{\mu_V}$ $($see \eqref{eq:ip} for a definition$)$, satisfy
$|\widehat{f}_n|\ll_f  n^{-\frac{\kappa+3}{2}}$.
\end{theorem}

The proof of this theorem begins by noting that the equation \eqref{eq:eveq} can be inverted on $\J$ yielding an eigenvalue equation for a suitable operator which turns out to be self-adjoint, compact, and positive on a suitable weighted Sobolev space that we call ~$\Hi_{\mu_V}$. This provides the existence of the functions $\phi_n$ on $\J$, which are then continued outside of $\J$ by \eqref{eq:eveq}, which becomes an ODE outside of $\J$.

The fact that the functions $\phi_n$ satisfy \eqref{eq:eveq} turns out to imply that functions of the form $\sum_{j=1}^N \phi_n(\lambda_j)$ are approximate eigenfunctions of the operator 
$$
\mathrm{L} \, = \, \mathrm{L}_{V,\beta}^N  
\, = \, \sum_{j=1}^N \frac{\partial^2}{\partial \lambda_j^2}-\beta N \sum_{j=1}^N V'(\lambda_j)\frac{\partial}{\partial \lambda_j}+\beta\sum_{i\neq j}\frac{1}{\lambda_j-\lambda_i}\frac{\partial}{\partial \lambda_j} \, .
$$
in the sense that the random variable
\begin{equation} \label{loop}
\mathrm{L}\big (\mathcal{X}_N(\phi_n)\big) +\beta N\sigma_n \mathcal{X}_N(\phi_n)
\end{equation}
turns out to be small compared to the eigenvalue $\beta N\sigma_n$. Here and elsewhere in this article, we use a slight abuse of notation and write e.g. 
$\mathrm{L}\big (\sum_{i=1}^N \phi(\lambda_i) \big)$ for the random variable which is obtained by first calculating $\mathrm{L}\big (\sum_{j=1}^N \phi(\lambda_i)\big)$ with deterministic $\lambda$, and then evaluating this function at a random $\lambda$ drawn from \eqref{eq:betaens}. This  approximate eigenfunction property is precisely what allows us making use of the normal approximation result Proposition~\ref{prop.main}. 
This leads to the following multi-dimensional CLT.

\begin{proposition}\label{prop:clt}
Using the notation of Theorem~\ref{th:basis} and \eqref{eq:linstat}, we have for any $\epsilon>0$, 
$$
\mathrm{W}_2\left(\big(\mathcal{X}_{N}(\phi_n) \big)_{n=1}^d, \gamma_d\right) 
  \, \ll \, d^{ \frac{16(\kappa+1)}{2\kappa-1}} N^{-1+\epsilon} . 
$$
\end{proposition}

Our proof of Theorem~\ref{th:main} then proceeds by approximating a general function $f$ by a linear combination of the $\phi_n$'s in $[-1,1]$, arguing that essentially what happens outside of this interval is irrelevant, and then using the CLT for the $\phi_n$.
Actually, the proof of Theorem~\ref{th:GUE} is quite a bit simpler due to the fact that the functions $\phi_n$ are explicit in this case: they are just suitably normalized Chebyshev polynomials of the first kind and much of the effort going into the proof of Theorem~\ref{th:basis} is not needed. 

\medskip

Besides the existence of these eigenfunctions, another fundamental property that is required to control the error term coming from Proposition~\ref{prop.main} is a kind of rigidity of the random configuration  $(\lambda_j)_{j=1}^N$. In particular, in the proof of Proposition~\ref{prop:clt} and Theorem~\ref{th:main}, we will make use 
of the following strong rigidity result of Bourgade, Erd\H{o}s and Yau.

\begin{theorem}[\cite{BEY14}, Theorem~2.4] \label{thm:rigidity}
Let $V$ be as in Assumption~\ref{as:1cut}, ${\rho_0=-1}$, and for any $j\in\lbrace 1,\ldots,N\rbrace$, 
define the classical locations $\rho_j\in[-1,1]$ by  
\beq\label{eq:classical}
\int_{\rho_0}^{\rho_j}\mu_V(dx) \, = \, \frac{j}{N} \, .
\eeq
Moreover, write $\widehat{j}=\min(j,N-j+1)$, assume that the eigenvalues are ordered
$\lambda_1\leq \lambda_2\leq \cdots\leq \lambda_N$, and consider the event
\beq \label{rigidity}
\B_\epsilon \, = \,  \big\{\forall j\in \lbrace 1,\ldots ,N\rbrace :
|\lambda_j -\rho_j| \le  \widehat{j}^{-\frac 13} \, N^{-\frac 23 +\epsilon}   \big\}.
\eeq
Then, for any $\epsilon>0$ there exist $c_\epsilon,N_\epsilon>0$ such that for all $N\ge N_\epsilon$, 
$$
\Prob^N_{V,\beta} \big (\R^N\setminus \B_\epsilon \big ) \, \le \,  e^{-N^{c_\epsilon}} .
$$
\end{theorem}

\subsection{Connection with the literature} \label{sec:BLS}
First we point out that the operator $\mathrm{L}$ also plays an important role in many other approaches to the CLT for linear statistics as the first loop equation can be written as
${\mathbb{E}_{V,\beta}^N\big[ \mathrm{L}(\mathcal{X}_N(f)) \big]=0}$.  In particular, if \eqref{loop} is asymptotically small compared to $\beta N\sigma_n$ in a strong enough sense, the first loop equation implies that $\mathbb{E}_{V,\beta}^N[\mathcal{X}_N(\phi_n)]\to 0$ as $N\to\infty$ so that the mean of the linear statistic $\sum_{j=1}^N f(\lambda_j)$ behaves like  
$N \int_\R f(x)\mu_V(dx)+ \big (\frac{1}{2}-\frac{1}{\beta}\big)\m(f) $ for large~$N$. 
Moreover, the eigenequation \eqref{eq:eveq} also playes a role in the analysis of the loop equations as well as for the transport map approach that we briefly present below. The starting point of the methods of the previous works \cite{J2, BG13, BG2,  Shcherbina13,  Shcherbina14} and \cite{BLS17} consists of expressing the Laplace transform of the law of a linear statistic as 
 \begin{equation} \label{ratio_1}
\mathbb{E}_{V,\beta}^N \left[ e^{s \sum_{j=1}^N f(\lambda_j)} \right] \, = \,  \frac{Z_{V_t,\beta}^N}{Z_{V_0,\beta}^N} , \qquad  t \,  = \,  - \frac{s}{\beta N},
 \end{equation}
where one defines the deformed potential  $V_t(x) = V(x) + tf(x)$ for any $t\in\R$.
In order to obtain the asymptotics this ratio of partition functions, the idea from transport theory  introduced in \cite{BFG13, Shcherbina14} (to establish local universality) and used in \cite{BLS17} (to obtain a CLT) consists of making a change of variables $\lambda_j \leftarrow \vartheta_t(\lambda_j)$ for all $j=1,\dots, N$ in the integral
\beqs \begin {split}
 Z_{V_t,\beta}^N  
   & \, = \, \int_{\R^N} e^{-\beta \mathcal{H}_{V_t }^N(\lambda_1,\ldots,\lambda_N)}\prod_{j=1}^N d\lambda_j \\
 &\, = \, \int_{\R^N} e^{-\beta \mathcal{H}_{V_t }^N\big(\vartheta_t(\lambda_1),\ldots,\vartheta_t(\lambda_N)\big) + \sum_{j=1}^N \log \vartheta_t'(\lambda_j)}\prod_{j=1}^N d\lambda_j .
 \end{split} \eeqs
 This implies that
 \begin{equation} \label{ratio_2}
  \frac{Z_{V_t,\beta}^N}{Z_{V_0,\beta}^N} \, = \, 
    \mathbb{E}_{V,\beta}^N\Big [ e^{-\beta \big( \mathcal{H}_{V_t }^N(\vartheta_t(\lambda_1),
        \ldots,  \vartheta_t(\lambda_N)) - \mathcal{H}_{V_0 }^N(\lambda_1,\ldots,\lambda_N)\big)
   + \sum_{j=1}^N \log \vartheta_t'(\lambda_j)}  \Big].
 \end{equation}
 Then,  it turns out that to obtain the CLT, it suffices to consider a simple diffeormorphism of the form $\vartheta_t(x) = x +t\psi(x)$ where $\psi$ is a solution of the equation
 \begin{equation} \label{eq:BLS}
\square^V_x(\psi)\, = \,- V'(x) \psi(x) +\int_\J\frac{\psi(x) - \psi(y)}{x-y} \, \mu_V(dy) 
   \, = \, f(x) + c_f  
 \end{equation}
 for some suitably chosen constant $c_f \in\R$. This equation is important because if we expand the exponent on the RHS of \eqref{ratio_2} up to order $t^2$, we can check that
 \beqs \begin{split} 
   -\beta \big [ \mathcal{H}_{V_t }^N &\big(\vartheta_t(\lambda_1),\ldots,\vartheta_t(\lambda_N)\big)- \mathcal{H}_{V_0 }^N(\lambda_1,  \ldots,\lambda_N) \big ] + \sum_{j=1}^N \log \vartheta_t'(\lambda_j)\\
 & \, = \,  - \beta tN^2 \int f(x) \mu_V(dx) 
     -\beta  tN \int  \big( f(x) -\square^V_x(\psi) \big) \nu_N(dx) \\
   & \hskip 6mm + \beta N^2 t^2\S(f) + \Big(\frac{\beta}{2} -1\Big) tN \m(f)
 + \epsilon_{N,t} + O_{N\to\infty}(Nt^2) \\
\end{split} \eeqs
where the error term is deterministic, $\epsilon_{N,t}$ is a small random quantity, and we have
\begin{equation}\label{ratio_4}
\m(f) \, = \,   -\int \psi'(x) \mu_V(dx)\, \, \, \,  \text{and}\, \,   \, \,  
    \S(f) \, =\,  -\frac{1}{2}\int f'(x) \psi(x) \mu_V(dx).
\end{equation}
These formulae are those of the asymptotic mean and variance of the random variable $\int f(x) \nu_N$ given in \cite[Theorem 1]{BLS17}.
Therefore, since $\beta N t =-s$,  combining these asymptotics with 
\eqref{ratio_1} and \eqref{ratio_2}, we obtain
  \begin{equation} \label{ratio_3}
\mathbb{E}_{V,\beta}^N \left[ e^{s \int f(x) \nu_N(dx)} \right] 
   \, = \,   e^{ s (\frac{1}{2}-\frac{1}{\beta})\m(f)+\frac{s^2}{\beta}\Sigma(f) + O_{N\to\infty}
   (N^{-1}) } \, \mathbb{E}_{V,\beta}^N \left[ e^{ \epsilon_{N,t} }\right] .
\end{equation}
To complete the proof of the CLT, most of the technical challenges consist of showing that the Laplace transform on the RHS of \eqref{ratio_3} converges to 1 as $N\to\infty$ -- in particular that the so-called {\it anisotropy term} contained in $\epsilon_{N,t}$ is small.
This step is performed by using the regularity of the function $\psi$ and some a priori estimates on the partition function $Z_{V_t,\beta}^N$ from \cite{LS17} that we will not detail here. 
The bottom line is that the CLT holds for all test functions $f$ for which the equation \eqref{eq:BLS} has a sufficiently smooth solution. This leads to sufficient conditions on the test function $f$, see (1.14) and (1.15) in \cite{BLS17}, valid in the multi-cut and certain critical situations under which the asymptotics
$$
\mathbb{E}_{V,\beta}^N \left[ e^{s \int f(x) \nu_N(dx)} \right] 
  \, = \,  e^{ s (\frac{1}{2}-\frac{1}{\beta})\m(f)+\frac{s^2}{\beta}\Sigma(f) 
  + o_{N\to\infty}(1)} 
$$
hold. 
Note that the operator $\square^V$ is related to the operator \eqref{Xi} in the following way
 $$ 
 \Xi^{\mu_V}(f) \, = \,  -\square^V(f')  \qquad\text{on } \, \J . 
 $$ 
Thus, in order to solve the eigenequation \eqref{eigeneq_2} which is fundamental to our proof of
Theorem~\ref{th:basis}, it suffices to prove that the operator $\mathcal{R}$ given by
$$
\mathcal{R}(\phi)' \, = \,   (-\square^V)^{-1}(\phi)
$$ 
is compact when acting on a suitable space $\mathscr{F}$ of functions $\phi:\J \to\R$. Then taking into account the conditions \cite[(1.14) and (1.15)]{BLS17} in the definition of $\mathscr{F}$, one should be able -- by adapting the arguments of Section~\ref{sec:basis} -- to generalize Theorem~\ref{th:basis} to the multi-cut and critical situations treated in \cite{BLS17}. Hence, it should be possible to generalize Theorem~\ref{th:main} and  to obtain a rate of convergence in the Kantorovich distance in these cases as well (with a rate of convergence depending on the regularity of the potential~$V$ and the equilibrium measure $\mu_V$).  Finally, let us comment that if $f\in \Hi$, (see \eqref{eq:sobo}) we may use the eigenbasis $(\phi_n)_{n=1}^\infty$
of Theorem~\ref{th:basis} to solve the equation \eqref{eq:BLS}. 
Namely, according to equation \eqref{eq:eveq}, we have
$$
\square^V(\phi_n') \, = \, - \sigma_n \phi_n
$$
and, if we expand $f =\widehat{f}_0+\sum_{n=1}^\infty \widehat{f}_n \phi_n$, then the function $\psi = -\sum_{n=1}^\infty \frac{\widehat{f}_n}{\sigma_n} \phi_n'$  solves \eqref{eq:BLS} with $c_f=-\widehat{f}_0$ (note that $\psi \in L^2(\mu_V)$). 
Then, we deduce from the  formulae \eqref{ratio_4} that
$$
\m(f) =  \int \psi'(x) \mu_V(dx) \, = \,  \sum_{n=1}^\infty  \frac{\widehat{f}_n}{\sigma_n} \int \phi_n''(x)\mu_V(dx) , 
$$
and since the functions  $(\phi_n')_{n=1}^\infty$ are orthonormal with respect to $L^2(\mu_V)$,  
$$
 \S(f) \, = \, -  \frac{1}{2}\int f'(x) \psi(x) \mu_V(dx) 
     \, = \,    \sum_{n=1}^\infty  \frac{\widehat{f}_n^2}{2\sigma_n} \, . 
$$
Using \eqref{mean_2} and Lemma~\ref{thm:Fourier_sum} below, we conclude that, if $f$ is sufficiently smooth,  the  formulae \eqref{ratio_4} from \cite[Theorem 1]{BLS17} are consistent with \eqref{mean} and \eqref{variance}.

\subsection{Outline of the article} We now describe the structure of the rest of the article.
In Section~\ref{sec:Stein}, we prove Proposition~\ref{prop.main} -- our general normal approximation result. After this, in Section~\ref{sec:GUEproof} we apply it to the GUE in order to prove Theorem~\ref{th:GUE}. We next move on to Section~\ref{sec:basis}, where we prove the existence and basic regularity properties of the functions $\phi_n$ -- namely Theorem~\ref{th:basis}. Armed with this information about the functions $\phi_n$, we prove Proposition \ref{prop:clt} in Section~\ref{sec:phiclt} and then in Section~\ref{sec:clt} we apply 
Proposition~\ref{prop:clt} to prove Theorem~\ref{th:main}. Finally in Appendix~\ref{app:hilbert}, we recall some basic properties of Chebyshev polynomials and the Hilbert transform which will play a critical role in our analysis. 

\medskip

{\emph{Acknowledgement}}: We wish to thank two anonymous reviewers for their careful reading and helpful comments about the article.

\section{A general normal approximation result} \label{sec:Stein}

In this section, we describe and prove some general normal approximation results for functions of random variables drawn from  probability distributions of the form $\mu(dx)=\frac{1}{Z} e^{-\mathcal{H}(x)}dx$ on $\R^N$, where $\mathcal{H}:\R^N\to \R$ is a nice enough function and $Z$ a normalization constant. These approximation results will of course be useful (in the sense that they say that this function of random variables is close to a Gaussian random variable in the $N\to\infty$ limit) only for very special functions of these random variables. Nevertheless, we will be able to apply these approximation results to a suitably large class of linear statistics of measures of this form with the choice of $\mathcal{H}=\beta\mathcal{H}_{V}^N$ from \eqref{eq:betaens}. While our main interest is linear statistics of $\beta$-ensembles, we choose to keep the discussion on a more general level as such normal approximation results might be of use in other settings as well. Keeping this in mind, we offer some further discussion about the approximation results, and even touching on some facts and approaches that may not be of use in our application to linear statistics of $\beta$-ensembles.

Before going into the statement and proof of the approximation results, we introduce some notation. An object that is of critical importance to our approach is a second order differential operator which is symmetric with respect to the inner product of $L^2(\mu)$. More precisely, we write for smooth $f:\R^N\to \R$, 
\beq\label{eq:gene}
\mathrm{L}f \, = \, \Delta f-\nabla \mathcal{H}\cdot \nabla f
  \, = \, \sum_{i=1}^N \partial_{ii} f-\sum_{i=1}^N \partial_i \mathcal{H}\partial_i f.
\eeq
 We point out for later reference that for $\beta$-ensembles, one has 
\beq\label{eq:betagene}
\mathrm{L}\,=\, \mathrm{L}_{V,\beta}^N \, = \, \sum_{j=1}^N \partial_{\lambda_j}^2-N\beta \sum_{j=1}^N V'(\lambda_j)\partial_{\lambda_j}+\beta \sum_{i\neq j}\frac{1}{\lambda_j-\lambda_i} \, \partial_{\lambda_j}
\eeq
which is just the infinitesimal generator of the diffusion with invariant measure $\mathbb{P}_{V,\beta}^N$  which we refer as \emph{Dyson Brownian motion}.

As mentioned above, we will make use of the fact that $\mathrm{L}$ is symmetric with respect to the inner product of $L^2(\mu)$. More precisely, if we assume sufficient regularity of $\mathcal{H}$ (say that it is smooth) then integrating by parts shows that for say smooth $f,g:\R^N\to \R$ with nice enough behavior at infinity, 
\beq\label{eq:ibp}
\int_{\R^N} f(-\mathrm{L} g)d\mu
  \, = \, \int_{\R^N} \nabla f\cdot \nabla g d\mu \, =: \, \int_{\R^N} \Gamma(f,g)d\mu.
\eeq
where we thus write $\Gamma (f,g) $ for $\nabla f \cdot \nabla g$. To ease notation, we also write
$\Gamma(f):=\Gamma(f,f)=|\nabla f|^2$. Note that in the setting of $\beta$-ensembles, $\mathcal{H}$ is not smooth, and when considering the integral $\int_{\R^N}f(-\mathrm{L})gd\mu$, one encounters singularities of the form $\frac{1}{\lambda_j-\lambda_k}\partial_j g(\lambda) \prod_{i<j}|\lambda_i-\lambda_j|^\beta$. As $\beta>0$, this is an integrable singularity, so integration by parts is justified and \eqref{eq:ibp} is still true. Our argument will implicitly impose several regularity assumptions on $\mathcal{H}$ and we will not be explicit about what precisely one should assume about $\mathcal{H}$. Nevertheless, when there are issues of this type, we will point out why there are no problems in the setting of linear statistics of $\beta$-ensembles.

Below in Section~\ref{sec:Steinst} we state our general normal approximation results and offer some further discussion about them in some particular cases. Then in Section~\ref{sec:Steinpf}, we prove the approximation results.

\subsection{Statement and discussion of the normal approximation results}\label{sec:Steinst} To simplify the statement of our approximation results, we fix some further notation. Let $X = (X_1, \ldots, X_N)$ be a random vector in $\R^N$ with distribution $\mu$ and let $F = (F_1, \ldots , F_d): \R^N \to \R^d$ be a smooth function. What we mean by a normal approximation result is that we wish to estimate the Kantorovich distance 
$\mathrm{W}_2$ (and in case $d=1$ also in the total variation distance) between the law of $F(X)$ and that of the standard Gaussian measure $\gamma_d$ on $\R^d$. Bounds on these distances will be stated in terms of the function $F $ and its differentials --  more precisely, in the notation introduced above, the bounds will 
be expressed in terms of (the vector-valued versions of) $\mathrm{L}F$ and $\Gamma(F)$.

For positive numbers $\kappa_k >0$, $ k = 1, \ldots, d$, denote by $K$ the diagonal
matrix $K = \textrm {diag} (\kappa_1, \ldots, \kappa_d)$.
Given $F = (F_1, \ldots , F_d)$, 
let $\Gamma (F) = {(\Gamma (F_k, F_\ell))}_{1 \leq k, \ell \leq d}$
and introduce the quantities $A$ and $B$ by
\beq \label{eq.a}
A \, = \,  \bigg ( \int_{\R^N} \big | F + K^{-1} \, \textrm {L} F \big |^2 d\mu \bigg)^{\frac 12}
\eeq
and
\beq \label{eq.b}
B   \, = \, \bigg ( \int_{\R^N} \big | \mathrm{Id} - K^{-1}\,  \Gamma(F) \big |^2 \, d\mu \bigg)^{\frac 12},
\eeq
the norms $| \cdot |$ being understood in the Euclidean
space $\R^d$ and in the space of $d \times d$ matrices
(Hilbert-Schmidt norm). The expressions $A$ and $B$ thus depend on $F$ and on 
$\kappa_k >0$, $k = 1, \ldots , d$. In this notation, our first approximation result reads as follows.

\begin {proposition} \label{prop.main}
Let $F : \R^N \to \R^d$  be of class $\Co^2$ and in $L^2(\mu)$,
and denote by $\mu \circ F^{-1}$ the law of $F$ under $\mu$
$($that is the law of $F(X)$ on $\R^d)$. Then, for any choice of $\kappa_k >0$, $k = 1, \ldots , d$,
$$
\mathrm {W}_2 \big (\mu \circ F^{-1}, \gamma_d \big ) \, \leq \, A + B,
$$
\noindent where again $\gamma_d$ is the law of a standard $d$-dimensional Gaussian.
\end {proposition}

\noindent We postpone the proof of Proposition~\ref{prop.main} (as well as the forthcoming 
Proposition~\ref{prop.maintv}) to Section~\ref{sec:Steinpf}.

Before discussing further normal approximation results, we point out that for Proposition~\ref{prop.main} to be of any use, one of course will want $F+K^{-1}\mathrm {L}F$ and $\mathrm{Id}-K^{-1}\Gamma(F)$
to be in $L^2(\mu)$ and small in some sense. This typically will not be true for arbitrary $F$ and $K$, but only for very special choices of $F$ and $K$. We return to the choice of $F$ and $K$ later on. 

We next mention that Proposition~\ref {prop.main} is already of interest in dimension one (i.e. when $d=1$) in which case we also get a bound in the total variation distance
$$
{\| \nu - \nu' \|}_{\rm TV} 
   \, = \, \sup_{ E \in {\mathcal{B}}(\R)} \big [ \nu(E) - \nu'(E) \big] 
   \, = \, \frac {1}{2} \sup \bigg [ \int_\R  \varphi \, d \nu  - \int_\R \varphi \, d\nu' \bigg]
$$
where the supremum is taken over all bounded measurable
$\varphi : \R \to \R$ with ${\| \varphi \|}_\infty \leq 1$. The normal approximation result is the following.

\begin {proposition} \label{prop.maintv}
Let $F : \R^N \to \R$  be of class $\Co^2$ and in $L^2(\mu)$,
and denote by $\mu \circ F^{-1}$ the law of $F$ under $\mu$
$($that is the law of $F(X)$ on $\R)$. Then, for any $\kappa >0$,
${ \| \mu \circ F^{-1}- \gamma_1  \|}_{\rm TV} \leq  2A+2B $, that is
\beq \begin {split} \label{eq:TV}
{\big \| \mu \circ F^{-1}- \gamma_1 \big \|}_{\rm TV}
     & \, \leq \,   2   \bigg ( \int_{\R^N} \Big [ F  +  \frac {1}{\kappa} \, {\rm L}F \Big ]^2
     		 d\mu \bigg)^{\frac 12} \\
     & \quad \, \,  +  2 \bigg( \int_{\R^N} \Big [ 1 -   \frac {1}{\kappa} \, \Gamma (F) \Big ]^2 
     		d\mu \bigg)^{\frac 12} . \\
\end {split} \eeq
Moreover, if $F = (F_1, \ldots , F_d)$
and $G = \sum_{k=1}^d \theta_k F_k$ where $\sum_{k=1}^d \theta_k^2 = 1$, then
\beq  \label{eq.tvsum}
{\big \| \mu \circ G^{-1} - \gamma_1 \big \|}_{\rm TV}  \, \leq \, 2 A + 2 B
\eeq
and 
\beq  \label{eq.w2sum}
\mathrm {W}_2 \big ( \mu \circ G^{-1}, \gamma_1 \big ) \, \leq \,  A +  B.
\eeq
\end {proposition}

\begin{remark}
To underline the difference between Proposition \ref{prop.main} and Proposition \ref{prop.maintv}, we mention here that the proof of Proposition \ref{prop.maintv} is a rather classical one-dimensional Stein's method argument, relying on Stein's lemma. In this setting, the natural metric in which one obtains approximation results is the total variation metric. The situation in Proposition \ref{prop.main} is slightly different. While there are generalizations of Stein's method to multivariate normal approximation, see e.g. \cite{Meckes}, these typically yield approximation results in the metric $\mathrm W_1$. Using these results, one could indeed prove a $($weaker$)$ version of Theorem \ref{th:main}. While philosophically very similar to classical Stein's method arguments, our proof of Proposition \ref{prop.main} relies instead on semigroup interpolation techniques, partly following
\cite {LNP15}, which allow upgrading convergence in the $\mathrm W_1$-metric to the $\mathrm W_2$-metric.
\end{remark}

To widen the spectrum of Propositions~\ref {prop.main} and~\ref {prop.maintv},
it is sometimes convenient to deal with a random vector $X$ given
as an image $X = U(Y)$ of another random vector $Y$ on $\R^m$
(typically Gaussian) where $U : \R^m \to \R^N$. Depending
on specific properties of the derivatives of $U$,
Proposition~\ref {prop.main} may be used to also control the
distance between the law of $F(X)$ and $\gamma_d$. 
This follows from the description of $A$ and $B$ for the new map $G = F \circ U$. 

In the same way, after a (linear) change of variables, the previous statements may be formulated
with the target distribution being that of the
Gaussian distribution $\gamma_{m, \Sigma}$ on $\R^d$
with mean $m$ and invertible covariance matrix $\Sigma = M \, ^{t}\!M$. For example, one has:

\begin {corollary} \label {cor.main}
Let $F : \R^N \to \R^d$  be of class $\Co^2$ and in $L^2(\mu)$,
and denote by $\mu \circ F^{-1}$ the law of $F$ under $\mu$
$($that is the law of $F(X)$ on $\R^d)$. Then
$$
\mathrm {W}_2 \big (\mu \circ F^{-1}, \gamma_{m, \Sigma} \big ) \, \leq \, A_m +  B_{m,\Sigma}
$$
where
$$
A_m \, = \,  \bigg ( \int_{\R^N} \big |F - m + K ^{-1} \, {\rm L}F \big |^2 d\mu \bigg)^{\frac 12} 
$$
and
$$
B_{m, \Sigma} \, = \,  \|M \|  
 \bigg ( \int_{\R^N} \big | {\rm Id} - (KM)^{-1} \,  \Gamma (F) \big |^2 \, d\mu \bigg)^{\frac 12}
$$
with $\| M \|$ being the operator norm of $M$.
\end {corollary}

\noindent This is a direct application of Proposition~\ref{prop.main} and we skip the proof.
Let us now turn to the choice of the coefficients $\kappa_k$. In applications, the
choice of the coefficients $\kappa_k$ might depend on the context to some degree, but we point out
that for $B$ from \eqref{eq.b} to be small, one should at least expect ${\rm Id}-K^{-1}\int \Gamma(F)d\mu$ to be small as well (say in the Hilbert-Schmidt norm). This would suggest that one natural choice, where the diagonal entries of this matrix vanish, would be that $\kappa_k= \int_{\R^N} \Gamma(F_k) d\mu$, $ k = 1, \ldots, d$ (these are strictly positive if $\nabla F_k$ is not $\mu$-almost surely zero). This is essentially the choice we make in our application of the normal approximation results to linear statistics of $\beta$-ensembles -- this choice would correspond to $\E \big [ \sum_{j=1}^N f'(\lambda_j)^2\big]$
while our choice will be $N\int f'(x)^2\mu_V(dx)$.

Let us consider some consequences of this choice of $K$. First of all, we point out that in this case a direct calculation shows that 
\beq \label {eq.a2}
B^2   \, = \, \sum_{k=1}^d  \frac {1}{\kappa_k^2} \, {\rm Var}_\mu\big ( \Gamma (F_k) \big) 
             + \sum_{k \not= \ell}\int_{\R^N}  
           \frac {1}{\kappa_k^2} \, \Gamma (F_k, F_\ell)^2 d\mu .
\eeq
Here ${\rm Var}_\mu (f) = {\rm Var} (f(X))$ is the variance of $f : \R^N \to \R$ with respect
to $\mu$, equivalently the variance of the random variable $f(X)$. Note that if $d=1$,
$B^2   =  \frac {1}{\kappa_1^2} \, {\rm Var}_\mu( \Gamma (F_1) )$. 

To simplify the expression of $A^2$, we recall some notation and facts from \cite {BGL14}. If $f,g$
are smooth functions on $\R^N$, set

\begin{align*}
\Gamma _2(f, g) \, &= \, {\rm Hess}(f) \cdot {\rm Hess}(g) + {\rm Hess}(\mathcal{H}) \nabla f \cdot \nabla g\\
  \, &= \, \sum_{i,j=1}^N \partial_{ij}f\partial_{ij}g
     +\sum_{i,j=1}^N \partial_{ij}\mathcal{H}\partial_j f\partial_i g.
\end{align*}
As for $\Gamma$, write below $\Gamma_2 (f) = \Gamma_2 (f,f)$.  By integration by parts (again assuming e.g. that $\mathcal{H}$ is smooth), for smooth functions
$f, g : \R^N \to \R$,
$$
\int_{\R^N} {\rm L} f  \, {\rm L}g \, d\mu \, = \, \int_{\R^N} \Gamma_2(f,g) d\mu.
$$
Therefore, in this notation, another simple calculation using our particular choice of $\kappa_k$ shows that 
\beq \label {eq.b2}
A^2 \, = \, \sum_{k=1}^d \int_{\R^N} 
        \Big [ \frac {1}{\kappa_k^2} \, \Gamma_2(F_k)  + F_k^2 - 2 \Big] d\mu.
\eeq

We point out here that it is not obvious that this type of argument is valid for $\beta$-ensembles. Indeed, a priori, $\mathrm{L} \, f\mathrm{L}g$ and $\partial_{ij}\mathcal{H}$ will have terms of the form $(\lambda_i-\lambda_j)^{-2}$ and if we assume just that $\beta>0$, this could result in a non-integrable singularity. Nevertheless, we note that if we are interested in linear statistics, namely we have $f(\lambda)=\sum_{j=1}^N u(\lambda_j)$ and $g(\lambda)=\sum_{j=1}^N v(\lambda_j)$ for some smooth bounded functions $u,v:\R\to \R$, then by symmetry, one has e.g.
\beqs \begin{split}
\mathrm{L}f(\lambda)
 & \, = \, \sum_{j=1}^N u''(\lambda_j)-N\beta \sum_{j=1}^N V'(\lambda_j)u'(\lambda_j)
      +\beta \sum_{i\neq j}\frac{1}{\lambda_j-\lambda_i} \, u'(\lambda_j)\\
& \, = \,  \sum_{j=1}^N u''(\lambda_j)-N\beta \sum_{j=1}^N V'(\lambda_j)u'(\lambda_j)
              +\frac{\beta}{2} \sum_{i\neq j}\frac{u'(\lambda_j)-u'(\lambda_i)}{\lambda_j-\lambda_i} \, ,
\end{split} \eeqs
which no longer has singularities. Similarly one has in this case
\beqs \begin {split}
\sum_{i,j=1}^N & \partial_{ij}(\beta  \mathcal{H}_V^N(\lambda))  \partial_i f(\lambda)\partial_j g(\lambda) \\
& \, = \, N\beta \sum_{i=1}^N V''(\lambda_i)u'(\lambda_i)v'(\lambda_i)
    +\beta\sum_{i\neq j}\frac{u'(\lambda_i)v'(\lambda_i)-u'(\lambda_i)v'(\lambda_j) }{(\lambda_i-\lambda_j)^2} \, 
\end{split} \eeqs
which has only singularities of type $\frac {1}{\lambda_i-\lambda_j}$ so as we are integrating against $\prod_{i<j}|\lambda_i-\lambda_j|^\beta$ with $\beta>0$, we have only integrable singularities. Thus integration by parts is again justified in the setting we are considering. 

Turning back to more general $\mathcal{H}$, we mention some further simplifications or bounds
one can make use of in some special cases. Let us still assume that
$\kappa_k = \int_{\R^N} \Gamma(F_k) d\mu$, $ k = 1, \ldots, d$.
In case the measure $\mu$ satisfies a Poincaré inequality in the sense that for any (smooth)
function $f : \R^N \to \R$,
\beq \label {eq.poincare}
{\rm Var}_\mu (f) \, \leq \, C \int_{\R^N} \Gamma (f) d\mu,
\eeq
(for example $C=1$ for $\mu = \gamma_N$, cf.~\cite [Chapter~4]{BGL14}), the quantity $B$,
rather $B^2$ from \eqref{eq.a2}, is advantageously upper-bounded by
\beq \label {eq.a'}
{B'}^2 \, = \, \sum_{k=1}^d \frac {C}{\kappa_k^2} \, \int_{\R^N} \Gamma \big ( \Gamma (F_k) \big) d\mu
             + \sum_{k \not= \ell}\int_{\R^N}  
           \frac {1}{\kappa_k^2} \, \Gamma (F_k, F_\ell)^2 d\mu .
\eeq
In particular, in the setting of Proposition~\ref {prop.maintv},
\beqs \begin {split}
{\big \| \mu \circ F^{-1}- \gamma_1 \big \|}_{\rm TV}
   & \, \leq \,   2 \bigg ( \int_{\R^N} \Big [ F  + 
         \frac {1}{\kappa} \, {\rm L} F \Big]^2 d\mu \bigg)^{\frac 12}  \\
     &\quad \, \,     + \frac {2}{\kappa}
        \bigg ( C \int_{\R^N} \Gamma \big (\Gamma (F) \big) d\mu \bigg)^{\frac 12}.    \\
\end {split} \eeqs

Before turning to the proofs, we mention that Propositions~\ref {prop.main} and~\ref {prop.maintv} are related to several earlier
and parallel investigations. They may first be viewed as a $\mathrm {W}_2$-version
of the multivariate normal approximation in the Kantorovich metric $\mathrm {W}_1$ developed
by Meckes in \cite {Meckes} and relying on exchangeable pairs. For the application to
linear statistics of random matrices, the pairs are created
through Dyson Brownian motion and also in this setting, the operator $\mathrm{L}$ appears naturally and plays an important role.

The quantities $A$ and $B$ arising in Propositions~\ref {prop.main} and~\ref {prop.maintv}
are also connected to earlier bounds in the literature. As a first instance, assume
that $F : \R^N \to \R$ is an eigenvector of ${\rm L}$ in the sense
that $- {\rm L}F = \kappa F$, normalized in ${L}^2(\mu)$ so that
$$
\int_{\R^N} \Gamma (F) d\mu \, = \, \int_{\R^N} F (-{\rm L}F) d\mu \, = \, \kappa.
$$
In this case, the inequality of Proposition~\ref {prop.maintv} amounts to
$$
{\| \mu \circ F^{-1} - \gamma_1 \|}_{\rm TV}
    \, \leq \, \frac {2}{\kappa} \, {\rm Var}_\mu \big ( \Gamma(F) \big)^{\frac 12} ,
$$
a result already put forward in \cite {L12}. As such,
Propositions~\ref {prop.main} and~\ref {prop.maintv} suggest a similar result
provided that $F$ is approximately an eigenvector in the sense that $A$ is small. This indeed is a central theme in our approach to proving the CLT for linear statistics of $\beta$-ensembles, and the motivation for our choice of the function $F$ in later sections.

Another source of comparison is the works \cite {C09} and \cite {NP12}.
To emphasize the comparison, let us deal
with the one-dimensional case $F : \R^N \to \R$ corresponding to Proposition~\ref {prop.maintv}.
In the present notation, provided that $\int_{\R^N} F^2 d\mu = 1$,
the methodology of \cite {C09,NP12} develops towards the inequality
$$
{\big \| \mu \circ F^{-1} - \gamma_1 \big \|}_{\rm TV} \, \leq \, {\rm Var}_\mu (T) , 
$$
where $T = \Gamma ((- {\rm L})^{-1} F, F)$ and 
$(- {\rm L})^{-1}$ is the formal inverse of the positive operator $- {\rm L}$ ($-\mathrm{L}$ being positive because of the integration by parts formula \eqref{eq:ibp}). Then,
provided the measure $\mu$ satisfies a Poincar\'e inequality, the preceding variance
is bounded from above by moments of differentials of $F$, as in $A$ and $B$ of
Propositions~\ref {prop.main} and~\ref {prop.maintv}.

An alternative point of view on the latter results may be expressed in terms of the Stein discrepancy
and the inequality
\beq \label {eq.discrepancy}
\mathrm {W}_2 \big ( \mu \circ F^{-1}, \gamma_d \big)
     \, \leq \, {\rm S}_2 \big ( \mu \circ F^{-1} \, | \, \gamma_d \big)
\eeq
emphasized in \cite {LNP15} where
$$
{\rm S}_2 \big ( \mu \circ F^{-1} \, | \, \gamma_d \big) \, = \,
   \bigg ( \int_{\R^d} | \tau_{\mu \circ F^{-1}} - \mathrm{Id}|^2 d\mu \circ F^{-1} \bigg)^{\frac 12}
$$
with $\tau_{\mu \circ F^{-1}}$ a so-called Stein kernel of the distribution $\mu \circ F^{-1}$.
In dimension $d=1$, this Stein kernel is characterized by
\begin{align*}
\int_{\R^N} F \varphi (F) d\mu \, = \, 
\int_{\R} x \varphi \, d\mu \circ F^{-1} \, &= \, \int_\R \tau_{\mu \circ F^{-1}} \varphi' d\mu \circ F^{-1}\\
   \, &= \, \int_{\R^N} \tau_{\mu \circ F^{-1}} (F) \varphi' (F) d\mu
\end{align*}
for every smooth $\varphi : \R \to \R$. In the previous notation, the kernel $\tau_{\mu \circ F^{-1}}$ 
may be described as the conditional expectation of $T = \Gamma ((- {\rm L})^{-1} F, F)$
given $F$, so that, again under the normalization $\int_{\R^N} F^2 d\mu = 1$,
$$
\mathrm {W}_2 \big ( \mu \circ F^{-1}, \gamma_1 \big)^2 
  \, \leq \, {\rm S}_2 \big ( \mu \circ F^{-1} \, | \, \gamma_1 \big)^2
   \, \leq \,   {\rm Var}_\mu (T) .
$$

With respect to the analysis of these prior contributions
\cite {C09,NP12,LNP15}, the approach developed in Propositions~\ref {prop.main}
and \ref {prop.maintv}
is additive rather than multiplicative, and concentrates directly on the generator $\rm L$
rather than its (possibly cumbersome) inverse. 
The formulation of Propositions~\ref {prop.main} and \ref {prop.maintv} allows us to recover, sometimes
at a cheaper price, several of the conclusions and illustrations developed in \cite {C09}.

\subsection{Proof of Proposition~\ref{prop.main} and Proposition~\ref{prop.maintv}}\label{sec:Steinpf}

The main argument of the proof relies on standard semigroup interpolation (cf.~\cite {BGL14})
together with steps from \cite {LNP15}.
Denote by $(P_t)_{t \geq 0}$ the Ornstein-Uhlenbeck semigroup
on $\R^d$, with invariant measure $\gamma_d$ the standard Gaussian
measure on $\R^d$ and associated infinitesimal generator
${\mathcal{L}}_d = \Delta - x \cdot \nabla$. The operator $P_t$ admits the classical
integral representation
\beq \label {eq.oupt}
P_t f(x) \, = \, \int_{\R^d} f \big ( e^{-t}x + \sqrt {1 - e^{-2t}} \, y \big) \gamma_d(dy),
   \quad t \geq 0, \, \, x \in \R^d.
\eeq
A basic property of this operator we shall make use of is that it is symmetric with respect to the
inner product of $L^2(\gamma_d)$ in the sense that for, say, bounded continuous functions
$f,g:\R^d\to \R$, $\int_{\R^d} f P_tg d\gamma_d  = \int_{\R^d}g P_t f d\gamma_d$
for each $t>0$ (cf.~\cite [Chapter~2, Section~2.7]{BGL14}).

\medskip

We start with the first approximation result.

\begin{proof}[Proof of Proposition~\ref {prop.main}]
Assume first that $F: \R^N \to \R^d$ is smooth and such that the law $\mu \circ F^{-1}$ of $F$ admits a smooth and positive
density $h$ with respect to $\gamma_d$. Denote then by
$$
 {\rm I}  (P_t h) \, = \,  \int_{\R^d} \frac {|\nabla P_t h|^2}{P_t h} \, d\gamma_d,
    \quad t \geq 0,
$$
the Fisher information
of the density $P_t h$ with respect to $\gamma_d$, which is assumed to be finite.
 It is also assumed throughout the proof
that $A$ and $B$ are finite otherwise there is nothing to show.
With $ v_t = \log P_t h$, $ t \geq 0$,
after integration by parts and symmetry of $P_t$ with respect
to $\gamma_d$ (cf.~the analysis in \cite {LNP15}),
$$
 {\rm I}(P_t h) \,  = \,   \int_{\R^d} \frac {|\nabla P_t h|^2}{ P_t h} \,  d\gamma_d
  \,  = \,   - \int_{\R^d} {\mathcal{L}}_d[v_t]  \, P_t h\, d\gamma_d
  \,  = \,   - \int_{\R^d} {\mathcal{L}}_d[P_t v_t]  \, d\mu \circ F^{-1}   . 
$$
Now, for any $ t >0$, and any $\kappa_k > 0$, $k = 1, \ldots, d$,
\beqs \begin {split}
  {\rm I}  (P_t h) \, &= \,  - \int_{\R^N} {\mathcal{L}}_d[P_t v_t] (F) d\mu\\
    &  \, = \,  - \int_{\R^N} \bigg [\sum_{k=1}^d \partial_{kk} P_t v_t  (F)
          - \sum_{k=1}^d F_k \partial _k P_t v_t (F)\bigg ] d\mu \\
   & \, = \, \int_{\R^N} \sum_{k=1}^d \Big [F_k + \frac{1}{\kappa_k} \, {\rm L} F_k  \Big]
      \partial_k P_t v_t (F) d\mu \\
      & \quad \, \,  +   \int_{\R^N} \sum_{k, \ell =1}^d 
           \Big [\frac {1}{\kappa_k} \Gamma( F_k, F_\ell) - \delta_{k\ell} \Big]
             \partial_{k\ell} P_tv_t (F) d\mu   , 
\end {split} \eeqs
where the last step follows by adding and subtracting $\frac{1}{\kappa_k}  {\rm L}[F_k] 
      \partial_k P_t v_t (F)  $  and integrating by parts with respect to $\rm L$, i.e.  formula \eqref{eq:ibp}. 

Next, by the Cauchy-Schwarz inequality,
$$
\int_{\R^N} \sum_{k=1}^d \Big [F_k + \frac{1}{\kappa_k} \, {\rm L} F_k  \Big]
      \partial_k P_t v_t (F) d\mu 
      \, \leq \,  A \bigg (\int_{\R^N} |\nabla P_t v_t (F) |^2 d\mu \bigg)^{\frac 12}
$$
and
\begin{align*}
\int_{\R^N} |\nabla P_t v_t (F) |^2 d\mu  
\, &= \, \int_{\R^d} |\nabla P_t v_t |^2 d\mu \circ F^{-1} \\
    \, &\leq \, e^{-2t}  \int_{\R^d} P_t \big( |\nabla  v_t |^2 \big) d\mu \circ F^{-1}\\
     \, &= \, e^{-2t}  \, {\rm I}  (P_t h) .
\end{align*}
Now, for every $k, \ell = 1, \ldots, d$,
$ \partial_{k\ell} P_tv_t = e^{-2t} P_t( \partial_{k\ell} v_t)$ and, by integration by parts
in the integral representation of $P_t$,
$$
\partial_{k\ell} P_tv_t (x) \, = \, \frac {e^{-2t}}{\sqrt {1 - e^{-2t}}}
     \int_{\R^d} y_k \partial_\ell v_t \big ( e^{-t} x + {\sqrt {1 - e^{-2t}}} \, y \big) \gamma_d(dy).
$$
Then, by another application of the Cauchy-Schwarz inequality, 
\begin {equation*} \begin {split}
&\int_{\R^N} \sum_{k, \ell =1}^d  \Big [\frac {1}{\kappa_k}   \Gamma( F_k, F_\ell) - \delta_{k\ell} \Big]
             \partial_{k\ell} P_tv_t (F) d\mu  \\
      &\, \leq \,  \frac {e^{-2t}B}{\sqrt {1 - e^{-2t}}} \, 
           \bigg ( \int_{\R^d} \sum_{k,\ell =1}^d \bigg [ \int_{\R^d}
                        y_k \partial_\ell v_t \big ( e^{-t} x + {\sqrt {1 - e^{-2t}}} \, y \big) 
                      \gamma_d(dy)  \bigg]^2 \mu \circ F^{-1}(dx) \bigg)^{\frac 12} \\
       &\, \leq \,  \frac {e^{-2t}}{\sqrt {1 - e^{-2t}}} \, B
           \bigg ( \int_{\R^d} \sum_{ \ell =1}^d  
               \int_{\R^d} \big [ \partial_\ell v_t   \big ( e^{-t} x + {\sqrt {1 - e^{-2t}}} \, y \big) \big]^2
                   \gamma_d (dy) \mu \circ F^{-1}(dx) \bigg)^{\frac 12} \\
     &\, = \,  \frac {e^{-2t}}{\sqrt {1 - e^{-2t}}} \, B
        \bigg ( \int_{\R^d} P_t \big( |\nabla  v_t |^2 \big) d\mu \circ F^{-1} \bigg)^{\frac 12} \\
       &\, = \,  \frac {e^{-2t}}{\sqrt {1 - e^{-2t}}} \, B  \sqrt{{\rm I}  (P_t h)} , 
\end {split} \end {equation*}
where we used at the second step that for a given function $g : \R^d \to \R$ in
$L^2(\gamma_d)$, 
$$
 \sum_{k=1}^d \bigg ( \int_{\R^d} y_k g (y) \gamma_d (dy) \bigg)^2
 \, \leq \,  \int_{\R^d} g^2 d\gamma_d,
$$
which follows from the remark that the operator 
$g\mapsto \sum_{k=1}^d x_k\int y_k g(y)\gamma_d(dy)$ is an orthogonal projection on $L^2(\gamma_d)$.

Altogether, it follows that for every $t >0$,
$$
{\rm I}  (P_t h) \, \leq \, \bigg ( e^{-t} A + \frac {e^{-2t}}{\sqrt {1 - e^{-2t}}} \, B
     \bigg) \sqrt { {\rm I}  (P_t h)  },
$$
hence
$$
\sqrt { {\rm I}  (P_t h)  } \, \leq \, e^{-t} A + \frac {e^{-2t}}{\sqrt {1 - e^{-2t}}} \, B .
$$

From \cite [Lemma~2]{OV00} (cf.~also \cite[Theorem~24.2(iv)]{V09}), it follows that
$$
\mathrm {W}_2 \big (\mu \circ F^{-1},\gamma \big )
 \, \leq \, \int_0^\infty \! \sqrt { {\rm I}  (P_t h)  } \, dt \, \leq \, A + B
$$
which is the announced result in this case.

The general case is obtained by a regularization procedure which we outline
in dimension $d=1$.  Let therefore $F: \R^N \to \R$ be of class $\Co^2$
and in $L^2(\mu)$. 
Fix $\varepsilon >0 $ and consider, on $\R^N \times \R$,
the vector $(X,Z)$ where $Z$ is a standard normal independent of $X$ and
$$
F_\varepsilon (x,z) \, = \, e^{-\varepsilon} F(x) + \sqrt {1 - e^{-2 \varepsilon}} \, z,
   \quad x \in \R^N, \, z \in \R.
$$
Denote by $\mu_{F_\varepsilon}$ the distribution of $F_\varepsilon (X,Z)$, or
image of $\mu \otimes \gamma_1$ under $F_\varepsilon$.
The probability measure $\mu_{F_\varepsilon}$
admits a smooth and positive density $h_\varepsilon$ with respect to $\gamma_1$ given by
$$
h_\varepsilon (x) \, = \, \int_\R p_\varepsilon (x, y) \mu \circ F^{-1}(dy), \quad x \in \R,
$$
where $p_t(x,y)$, $t >0$, $x,y \in \R$, is the Mehler kernel of the semigroup
representation \eqref {eq.oupt} (and coincides with $P_\varepsilon h$
whenever $\mu \circ F^{-1}$ admits a density $h$ with respect to $\gamma_1$). 
Furthermore, by the explicit representation of $p_\varepsilon (x,y)$
(cf.~e.g.~\cite [(2.7.4)]{BGL14}),
$$
h_{\varepsilon}'(x) \, = \, - \frac {e^{-2\varepsilon}}{1 - e^{-2\varepsilon}}
      \int_\R [ x - e^\varepsilon y] \, p_\varepsilon (x, y) \mu \circ F^{-1}(dy),
$$
and from the Cauchy-Schwarz inequality, we obtain for any $x\in\R$, 
$$
\frac {{h_{\varepsilon}'(x)}^2}{h_{\varepsilon}(x)}
   \, \leq \, \Big ( \frac {e^{-2\varepsilon}}{1 - e^{-2\varepsilon}}\Big)^2
       \int_\R [ x - e^\varepsilon y]^2 \, p_\varepsilon (x, y) \mu \circ F^{-1}(dy) . 
$$
 Since $F \in L^2(\mu)$, it follows that
$\int_\R \frac {{h_{\varepsilon}'}^2}{h_{\varepsilon}} \, d\gamma_1 < \infty$ for any $\varepsilon>0$.
Now, for every $t \geq 0$, $P_t h_\varepsilon = h_{t + \varepsilon}$, so that
the Fisher informations ${\rm I}(P_th_\varepsilon)$, $t \geq 0$, are well-defined and finite.
The proof we have presented thus far then applies to $F_\varepsilon$
 on the product space $\R^N \times \R$ with respect to the generator
$\mathrm {L} \oplus \mathcal {L}_1$. Next
$F_\varepsilon \to F$ in $L^2(\mu \times \gamma_1)$ from which
$$
{\rm W_2} \big (\mu \circ F_\varepsilon^{-1}, \mu \circ F^{-1} \big ) ^2  
    \, \leq \, \int_{\R^N \times \R} | F_\varepsilon - F |^2 d\mu \times \gamma_1 \, \to \, 0.
$$
Hence, by the triangle inequality,
${\rm W_2} (\mu \circ F_\varepsilon^{-1}, \gamma_1 ) \to {\rm W_2} (\mu \circ F^{-1}, \gamma_1)$.
On the other hand, the $\Gamma$-calculus of \cite [Chapter 3]{BGL14}
developed on $\R^N \times \R$ yields the
quantities $A$ and $B$ in the limit as $\varepsilon \to 0$. The proof of
Proposition~\ref {prop.main} is complete.

\end{proof} 

We next turn to our second approximation result which is similar
but stays at the first order on the basis of the standard Stein equation.

\begin{proof}[Proof of Proposition~\ref {prop.maintv}]
The classical Stein lemma states that
given a bounded function $\varphi : \R \to \R$, the equation
\beq \label {eq.stein}
  \psi' - x \psi  \, = \, \varphi - \int_\R \varphi \, d\gamma_1 
\eeq
may be solved with a function $\psi$, which along with its derivative, is bounded. 
More precisely,
$\psi$ may be chosen so that ${\|\psi \|}_\infty \leq \sqrt {2\pi} \, {\|\varphi \|}_\infty $
and ${\|\psi' \|}_\infty \leq 4 \,{\|\varphi \|}_\infty $.
Stein's lemma may then be used to provide the basic approximation bound  
\beq \label {eq.steintv}
{\|\lambda - \gamma_1 \|}_{\rm TV} \, \leq  \, \sup \bigg | \int_\R \psi '(x) d \lambda (x)
      - \int _\R x \psi (x) d\lambda (x) \bigg | ,
\eeq
where the supremum runs over all continuously differentiable functions ${\psi  : \R \to \R}$
such that ${\| \psi  \|}_\infty \leq \sqrt { \frac {\pi}{2} } $ and ${\| \psi ' \|}_\infty \leq 2$.

We thus investigate
$$
\int_\R \psi '(x) d \mu \circ F^{-1} (x) - \int _\R x \psi (x) d\mu \circ F^{-1}(x) 
   \, = \, \int_{\R^N} \psi'(F) d\mu - \int_{\R^N} F \psi (F) d\mu
$$
and proceed as in the proof of Proposition~\ref {prop.main}.
Namely, with $\kappa >0$, by the integration by parts formula \eqref {eq:ibp},
\begin{align*}
\int_{\R^N} \big [ \psi'  (F) -   F\psi (F) \big ] d\mu 
    \, &= \,   -  \int_{\R^N} \psi (F) \Big [ F+ \frac {1}{\kappa} \, {\rm L} F  \Big ]  d\mu\\
  &  \quad        +  \int_{\R^N} \psi'(F) \Big [ 1 - \frac {1}{\kappa} \,  \Gamma (F) \Big] d\mu . 
\end{align*}
As a consequence, 
\beqs \begin {split}
\int_{\R^N} \big [ \psi'  (F) - F \psi (F) \big ] d\mu  
   & \, \leq \,  {\| \psi \|}_\infty
                 \bigg ( \int_{\R^N} \Big [F +  \frac {1}{\kappa} \, {\rm L} F  \Big ] ^2 d\mu \bigg)^{\frac 12}  \\
   & \quad \, \, + {\| \psi' \|}_\infty
                 \bigg ( \int_{\R^N} \Big [1 -  \frac {1}{\kappa} \, \Gamma(F)  \Big ] ^2 d\mu \bigg)^{\frac 12} . \\             
\end {split} \eeqs
By definition of the total variation distance and Stein's lemma, Proposition~\ref {prop.maintv} follows.
It remains to briefly analyze \eqref {eq.tvsum} and \eqref{eq.w2sum}.
Arguing as in the proof \eqref{eq:TV}, we find that
\beqs \begin {split}
\int_{\R^N} \big [ \psi'  (F) -   F\psi (F) \big ] d\mu 
    & \, = \,   -  \int_{\R^N} \psi (F) 
       \sum_{k=1}^d \theta_k \Big [ F_k + \frac {1}{\kappa_k} \, {\rm L} F_K  \Big ] d\mu \\
     & \quad \; \;   +      \int_{\R^N} \psi'(F)  \sum_{k, \ell =1}^d \theta_k \theta_\ell \Big [ \delta_{k \ell}
         - \frac {1}{\kappa_k} \, \Gamma (F_k, F_\ell) \Big] d\mu.  \\
\end {split} \eeqs
\eqref{eq.tvsum} then follows from the Cauchy-Schwarz inequality and the definition of $A$ and $B$. We note that \eqref{eq.w2sum} also holds as a consequence of Proposition~\ref {prop.main} since, as is easily checked,
$$
\mathrm {W}_2 \big (\mu \circ G^{-1}, \gamma_1 \big)
     \, \leq \, \mathrm {W}_2 \big (\mu \circ F^{-1}, \gamma_d \big ).
$$
The proof of Proposition~\ref {prop.maintv} is complete.
\end{proof}

\section{The CLT for the GUE -- Proof of Theorem~\ref{th:GUE}}
\label{sec:GUEproof}

On the basis of the general normal approximation results put forward in Section~\ref{sec:Stein},
we address the proof Theorem~\ref{th:GUE}. In this section, we will write simply $\Prob$ and $\mathrm{L}$ for $\Prob_{V,\beta}^N$ and $\mathrm{L}_{V,\beta}^N$ with $\beta=2$ and $V(x)=x^2$, as well as $\E$ for the expectation with respect to $\Prob$. 

\begin{proof}[Proof of Theorem~\ref{th:GUE}]
We proceed in several steps, starting with establishing the fact that linear statistics of Chebyshev polynomials of the first kind are approximate eigenvectors (in a sense that will be made precise in the course of the proof) of  the operator $\mathrm{L}$. Then, we move on to controlling the error in this approximate eigenvector property in order to apply Proposition~\ref{prop.main} to get a joint CLT for linear statistics of Chebyshev polynomials. Finally we expand a general polynomial in terms of Chebyshev polynomials to finish the proof.

\smallskip

{\textbf{Step 1 -- approximate eigenvector property.}}
Recall that $T_k$  and $U_k$, $k \geq 0$, 
denote the degree $k$ Chebyshev polynomial of the first kind, respectively  of the second kind -- we refer to Appendix~\ref{app:hilbert} for further details about the definition and basic properties of Chebyshev polynomials. See also \cite {CD01} and \cite {LP13} for related discussions.

We begin by noting that a direct application of Lemma~\ref{le:Ucheb} shows that when $x\in\J$,
$k\geq 0$,
\beq\label{eq:chebyev}
\int_\J\frac{T_k'(x)-T_k'(y)}{x-y} \, \scl(dy) \, = \, 2x T_k'(x)-2k T_k(x).
\eeq
 As both sides of this equation are polynomials, this identity is actually valid for all $x\in \R$. We also point out here that this is precisely \eqref{eq:eveq} for $V(x)=x^2$ with $\sigma_k=2k$. In fact, this part of the proof is completely independent of $\beta$, but since it is much simpler to control the various error terms  when $\beta=2$, for simplicity, we choose to stick to $\beta=2$ throughout the proof. The general case is covered by Theorem~\ref{th:main}. 

Next we note that integrating \eqref{eq:chebycoef} with respect to $\scl(dx)$ and making use of \eqref{eq:chebyortho} along with the facts that $T_k'=kU_{k-1}$, $2x=U_1(x)$, and $(1-x^2)=-\frac{1}{2}(T_2(x)-T_0(x))$, we find, for every $k \geq 1$,
\beq \begin{split}\label{eq:doubleint}
& \iint_{\J\times \J} \frac{T_k'(x)-T_k'(y)}{x-y} \,  \scl(dx)\scl(dy) \\
& \, = \, k\int_\J U_1(x)U_{k-1}(x)\scl(dx)+2k \int_\J T_k(x) \big (T_2(x)-T_0(x) \big )\varrho(dx)\\
& \, = \, 2k\delta_{k,2}\\
& \, = \, -4k\int_\J T_k(x)\scl(dx).
\end{split} \eeq
Applying first \eqref{eq:chebyev} and then \eqref{eq:doubleint}, we see that 
\beq \begin{split}\label{eq:evpropcheby}
\mathrm{L} \bigg (\sum_{j=1}^N & T_k(\lambda_j) \bigg) 
 \, = \, \sum_{j=1}^N T_k''(\lambda_j)-4N\sum_{j=1}^N \lambda_jT_k'(\lambda_j)+2\sum_{i\neq j}\frac{T_k'(\lambda_j)}{\lambda_j-\lambda_i}\\
& \, = \, -4kN \sum_{j=1}^N T_k(\lambda_j)-2N\sum_{i=1}^N \int_\J\frac{T_k'(\lambda_i)-T_k'(y)}{\lambda_i-y} \, \scl(dy) \\
& \quad  \; \; +\sum_{i,j=1}^N \frac{T_k'(\lambda_i)-T_k'(\lambda_j)}{\lambda_i-\lambda_j}\\
 & \, = \, -4kN\bigg(\sum_{j=1}^N T_k(\lambda_j)-N\int_\J T_k(x)\scl(dx)\bigg) \\
 & \quad \; \;  +\int_{\R\times \R}\frac{T_k'(x)-T_k'(y)}{x-y} \, \nu_N(dx)\nu_N(dy),
\end{split} \eeq
where $\nu_N$ is given by formula \eqref{eq:nuv} with the equilibrium measure 
$\mu_V=\scl$. We also used the fact that for all $k\in\N$, by symmetry,  
$$
2\sum_{i\neq j}\frac{T_k'(\lambda_j)}{\lambda_j-\lambda_i}   \, = \,
\sum_{i , j = 1}^N\frac{T_k'(\lambda_j)-T_k'(\lambda_i)}{\lambda_j-\lambda_i} 
    - \sum_{i=1}^N T_k''(\lambda_i)
$$
with the interpretation that the diagonal terms in the double sum are $T_k''(\lambda_i)$.
This shows that the {\it re-centered} linear statistics $\int T_k(x) \nu_N(dx)$
are approximate eigenfunctions in the sense that
\beq \label{approx_eq}
\mathrm{L}\bigg (\int_\R T_k(x) \nu_N(dx)\bigg)
  \, = \, -4k N\int_\R T_k(x) \nu_N(dx) + \zeta_k(\lambda) ,
\eeq
where for any configuration $\lambda\in\R^N$, 
\beq\label{eq:zeta}
\zeta_k(\lambda) \, = \, \int_{\R\times \R}\frac{T_k'(x)-T_k'(y)}{x-y} \, \nu_N(dx)\nu_N(dy)
\eeq 
is interpreted as an error term. 
In fact, knowing the CLT for linear statistics, one can check that for any $k\in\N$, $\zeta_k$ converges in law to a finite sum of products of Gaussian random variables as $N\to\infty$, so that its fluctuations are negligible in comparison to $4kN$ (though we make no use of this). 
In particular, we will choose the $d\times d$ diagonal matrix $K=K_N$ appearing in the definition of $A$ and $B$ in \eqref{eq.a} and \eqref{eq.b} to have entries $K_{kk}=4kN$.

\smallskip

{\textbf{Step 2 -- a priori bound on linear statistics.}} To apply Proposition~\ref{prop.main} we will make use of the fact that $\frac{T_k'(x)-T_k'(y)}{x-y}$ is a polynomial in $x$ and $y$ so that we can express each $\zeta_k$ as a sum of products of (centered) polynomial linear statistics. Thus to obtain a bound on $A$ in
Proposition~\ref{prop.main}, it will be enough to have some bound on the second moment of such statistics. In the case of the GUE we are able to apply rather soft arguments for this (as opposed to the fact that we rely on rigidity estimates for general $V$ and $\beta$).

It is known that for any polynomial $f$, the error term in Wigner's classical limit theorem (namely in the statement that $\lim_{N\to\infty} \frac 1N \, \E \big [\sum_{j=1}^N f(\lambda_j) \big]
=\int f(x)\scl(dx)$) is of order $\frac {1}{N^2}$ 
(see e.g.~\cite{M91}). That is, we have
\beq \label {eq.ratepolynomial}
\left|\E \bigg [ \int_\R f(x) \nu_N(dx)\bigg ]\right|   \, \ll_f \, N^{-1} .
\eeq
In addition, it is known that for any $p>0$, 
\beq \label{eq.convexity}
\sup_{N \geq 1} \E 
    \bigg | \int_\R f(x) \nu_N(dx)\bigg|^p \, < \, \infty.
\eeq
This property seems part of the folklore (cf.~\cite {AGZ10,PS11}) but we provide here a proof
for completeness. 
For $\beta=2$ and $V(x)=x^2$, the Hamiltonian $\beta\mathcal{H}_V^N$ in \eqref{eq:betaens} is more convex than the
quadratic potential $N \sum_{j=1}^N 2 \lambda_j^2$. In particular, the GUE eigenvalue measure
$\Prob_{x^2,2}^N$ satisfies a logarithmic Sobolev inequality with constant $N^{-1}$ 
(\cite  [Theorem 4.4.18, p. 290, or Exercise 4.4.33, p. 302]{AGZ10} or 
\cite [Corollary~5.7.2]{BGL14}) and 
by the resulting moment bounds (cf.~\cite[Proposition~5.4.2]{BGL14}), for every smooth
$g : \R^N \to \R$,
\beq \label{moment_bound}
\E \big | g - \E g \big |^p 
   \, \leq \, \frac {C_p}{N^{p/2}} \, \E |\nabla g|^p 
\eeq
for some constant $C_p >0$ only depending on $p \geq 2$.
%
We thus apply the latter to
$ g = \int f(x)\nu_N(dx)$. At this stage, we detail the argument for the value $p=4$ (which will be used
below) but the proof is the same for any (integer) $p$ (at the price of repeating the step). 
First, we have
\beqs \begin {split}
 \E |\nabla g|^4  & \, = \, \E \bigg ( \sum_{i=1}^N f'(x_i)^2 \bigg)^2  \\
 & \, = \,  \E \bigg ( \int_\R f'(x)^2 \nu_N(dx) + N \int_\J f'(x)^2\scl(dx) \bigg)^2  \\
\end {split} \eeqs
and, by \eqref{moment_bound}, 
\beq  \begin {split} \label{variance_5}
\E \bigg |  \int_\R f'(x)^2 \nu_N(dx)  - \E  \bigg[ \int_\R  f'(x)^2 \nu_N(dx)  \bigg] \bigg|^2  
    \, \leq \, \frac {4C_2}{N} \, \E\bigg[\sum_{i=1}^N f''(x_i)^2 f'(x_i)^2\bigg] . \\
\end {split} \eeq
By Wigner's law, the RHS of \eqref{variance_5} is uniformly bounded in $N$ and by \eqref{eq.ratepolynomial}, if $f$ is a polynomial, we obtain by the triangle inequality that 
$$
\E \bigg |  \int_\R f'(x)^2 \nu_N(dx) \bigg |^2   \, \ll_f \, 1 ,
$$
which implies that $ \E |\nabla g|^4   \ll_f  N^2 $.
Then, using the estimate \eqref{moment_bound} once more, we obtain
$$
\E \big | g - \E g \big |^4 
   \, \leq \, \frac {C_4}{N^{2}} \, \E |\nabla g|^4  \, \ll_f \,  1 .
$$
By \eqref {eq.ratepolynomial}, we also have that $|\E g| \ll N^{-1}$ and the claim \eqref {eq.convexity} follows by the triangle inequality.

\smallskip

{\textbf{Step 3 -- controlling $B$.}}  \label{sect:GUE_2}
We now move on to applying Proposition~\ref{prop.main}. Noting that by \eqref{variance_2}, one has $\S(T_k)=\frac{k}{4}$, we define $F = (F_1, \ldots, F_d) :\R^N\to \R^d$ by
$$
F_k \, = \, \frac{2}{\sqrt{k}} \int T_k(x) \nu_N(dx) , 
$$
where $d$ is independent of $N$ (in contrast to the proof of Proposition~\ref{prop:clt}
in Section~\ref{sec:phiclt}). 
Then, according to \eqref{eq.b}, we have
\beq \begin {split} \label{B_est}
B^2 & \, = \, \E \big |\mathrm{Id}-K^{-1}\Gamma(F) \big |^2 \\
 &  \, = \, \sum_{i,k=1}^d \E\left(\delta_{ik}-\frac{\Gamma(F_i,F_k)}{4kN}\right)^2 \\
  & \, = \, \sum_{k=1}^d\E\left(1-\frac{\Gamma(F_k)}{4kN}\right)^2
       +\sum_{i\neq k}\frac{\E\Gamma(F_i,F_k)^2} {16k^2 N^2} \, . \\
\end {split} \eeq

\noindent By definition of $\Gamma$ from \eqref{eq:ibp}, we have
$$
\Gamma(F_k) \, = \,  \frac{4}{k}\sum_{j=1}^N T_k'(\lambda_j)^2   =  \frac{4}{k} \int_\R T_k'(x)^2 \nu_N(dx) +  \frac{4}{k} N\int_\J T_k'(x)^2 \scl(dx) . 
$$
Now, recalling that the Chebyshev polynomial of the second kind are orthonormal with respect to the semicircle law (see formulae \eqref{eq:chebyder} and \eqref{eq:chebyortho}), we see that
$$
\frac{\Gamma(F_k)}{4kN}
\, = \, \frac{1}{k^2N} \int T_k'(x)^2 \nu_N(dx) +1 
$$
so that using the estimate \eqref{eq.convexity}, we obtain
$$
\sum_{k=1}^d\E\left(1-\frac{\Gamma(F_k)}{4kN}\right)^2 \, \ll_d \, N^{-2}  .
$$
A similar argument shows that for $i\neq k$, 
\beqs
 \Gamma(F_i,F_k)^2
 \, = \, \frac{4}{ik} \sum_{j=1}^N  T_i'(\lambda_j) T_k'(\lambda_j)
 \, = \, \frac{4}{ik}  \int_\R  T_i'(x) T_k'(x) \nu_N(dx)
 \eeqs 
 and 
$$
\sum_{i\neq k}\frac{\E \, \Gamma(F_i,F_k)^2}{16k^2 N^2} \, \ll_d \, N^{-2} . 
$$
Combining the two previous estimates, by \eqref{B_est}, we conclude that ${B\ll_d  N^{-1}}$.

\smallskip

{\textbf{Step 4 -- controlling $A$.}} Let us begin by noting that according to \eqref{eq.a} and the approximate eigenequation \eqref{approx_eq}, we have
\beq\label{eq:zetabound}
A^2 \, = \, \E \big |F+K^{-1}\mathrm{L}F \big |^2
  \, = \, \sum_{k=1}^d \frac{\E|\zeta_k|^2}{4k^3N^2} \,  . 
\eeq
Since we are dealing with polynomials, there exists real coefficients
$a_{i,j}^{(k)}$ which are zero if $i+j > k-2$  so that for any $k\in\N$, 
$$
\frac{T_k'(x)-T_k'(y)}{x-y} \, = \, \sum_{i,j\ge 0}a_{i,j}^{(k)}x^i y^j . 
$$
This implies that the error term \eqref{eq:zeta} factorizes:
$$
\zeta_k \, = \, \sum_{i,j\ge 0}a_{i,j}^{(k)} \int_\R x^i \nu_N(dx) \int_\R y^j \nu_N(dy) .
$$
Thus, by \eqref{eq.convexity} and Cauchy-Schwarz, we obtain that for any $k\in\N$, 
$$
 \E |\zeta_k|^2 \, \ll_k \, 1 .
$$
Hence, by \eqref{eq:zetabound}, we conclude that $A\ll_d  N^{-1}$. Combining this estimate with the one coming from Step~3, we see that Proposition~\ref{prop.main} gives, for any fixed $d\geq 1$, a multi-dimensional CLT for the linear statistics associated to Chebyshev polynomials
\beq \label{GUE_CLT_1}
\mathrm {W}_2 \big (\mu \circ F^{-1}, \gamma_d \big )  \, \ll_d\, N^{-1}
\eeq
where $\mu \circ F^{-1}$ refers to the law of the vector
$ F = \big ( \frac{2}{\sqrt{k}} \int T_k(x) \nu_N(dx) \big)_{k=1}^d$ . 

\smallskip

{\textbf{Step 5 -- extending to general polynomials.}} Consider now an arbitrary degree $d$ polynomial $f$. We can always expand it in the basis of Chebyshev polynomials of the first kind: $f=\sum_{k=0}^d f_k T_k$. 
By definition, this implies that
\beq \label{GUE_CLT_2}
\int_\R f(x) \nu_N(dx) \, = \,  \frac{1}{2} \sum_{k=0}^d \sqrt{k} f_k  F_k .
\eeq
Moreover, by formula \eqref{variance_2}, one has $\S(f)=\frac{1}{4} \sum_{k=0}^d kf_k^2$.
Then, since $\frac{1}{2\sqrt{\S(f)}} \sum_{k=0}^d \sqrt{k} f_k  X_k  \sim \gamma_1$ if the vector $X\sim \gamma_d$,
 it follows from the definition of the Kantorovich distance and the representation \eqref{GUE_CLT_2} that 
$$
\mathrm{W}_2^2\left( \int_\R f(x) \nu_N(dx) , \sqrt{\S(f)}  \gamma_1\right)  
  \, \le \,   \S(f)\,  \mathrm {W}_2^2 \big (\mu \circ F^{-1}, \gamma_d \big ). 
$$
Then, the first statement of Theorem~\ref{th:GUE} follows directly, after simplifying
by $\S(f)>0$, from the  multi-dimensional CLT \eqref{GUE_CLT_1}.

\medskip

{\textbf{Step 6 -- optimality.}} We conclude the proof of Theorem~\ref{th:GUE} by verifying the optimality claim. Consider now $f(x) = x^2 $.
In this case indeed, the sum $\sum_{j=1}^N \lambda_j^2$
corresponds to the trace of the square of a GUE matrix and is thus equals $\frac {1}{4N}$ times a random variable whose law is the $\chi^2$-distribution with $\frac 12 \, N(N+1)$ degrees of freedom.
Now, by standard arguments, given $X\sim\gamma_d$, we have
$$
\bigg | \E  \sin \bigg ( \sum_{k=1}^d \frac {X_k^2}{\sqrt d} - {\sqrt d} \bigg) 
- \frac {c}{\sqrt d}  \bigg|  \, \ll \, \frac 1d 
$$
where $ c \not= 0$. Comparing this to the Kantorovich-Rubinstein representation of the $\mathrm{W}_1$-distance shows that the total variation and all Kantorovich distances $\mathrm{W}_p$ (with $p>1$) between
the distribution of $  \sum_{k=1}^d \frac {X_k^2}{\sqrt d} - {\sqrt d} $ and the limiting Gaussian
distribution cannot be better than $\frac {1}{\sqrt d}$ in the $d\to\infty$ limit.
The optimality claim then follows from this with a simple argument. 
\end{proof}

\section{Existence and properties of the functions \texorpdfstring{$\phi_n$}{phi} -- 
Proof of Theorem~\ref{th:basis}}\label{sec:basis}

A central object in our construction of the functions $\phi_n$ in Theorem~\ref{th:basis} is the following Sobolev space
\beq\label{eq:sobo}
\Hi = \left\lbrace g \in L^2(\varrho): g'\in L^2(\scl)  \, \text{ and } \,  \int_\J g(x) \varrho(dx)=0\right\rbrace 
\eeq
which is equipped with the inner-product
$$
\langle f, g \rangle = \int_\J f'(x) g'(x) \scl(dx)  .
$$
In the following, we will consider other inner-products defined on $\Hi$ which are given by 
\beq\label{eq:ip}
{\langle f, g \rangle}_{\mu_V} = \int_\J f'(x) g'(x) \mu_V(dx)  .
\eeq
Note that, since we assume that $d\mu_V = Sd\scl$ and $S>0$ on $\bar\J$,
$$
\| g \|  \, \asymp  \, {\|g \|}_{\mu_V} 
$$
for all $g\in\Hi$. Thus the norms of the Hilbert spaces $\Hi_{\mu_V} := \big( \Hi, \langle\cdot\rangle_{\mu_V}\big)$ are equivalent to that of $\Hi$. 
Recall also that  we suppose that  $V\in\Co^{\kappa+3}(\R)$ for some $\kappa\in\N$ and  that the Hilbert transform (see Appendix~\ref{app:hilbert} for further information and our conventions for the Hilbert transform) of the equilibrium measure satisfies
  \begin{align} \label{V_condition}
  \begin{cases}
  \mathcal{H}_x(\mu_V) \, = \, V'(x)  &\text{if }x\in\bar\J, \\ 
    \mathcal{H}_x(\mu_V) \, < \,  V'(x) &\text{if }x\in \J_\delta\setminus \bar\J .  
  \end{cases} 
  \end{align}

    
  The aim of this section is to investigate the solutions of the equation
\beq \label{eigeneq_1}
V'(x)\phi'(x)-\int_\J\frac{\phi'(x)-\phi'(t)}{x-t} \, \mu_V(dt)\, = \, \sigma \phi(x)
\eeq
where $\sigma>0$ is an unknown of the problem. In particular, an important part of
the analysis will be devoted to control the regularity of a solution $\phi$. 
Let us consider the  operator
\beq \label{Xi}
\Xi^{\mu_V}_x(\phi) \, = \, \pv \int_\J\frac{\phi'(t)}{x-t} \, \mu_V(dt) . 
\eeq
This operator is well-defined on $\Hi$ and, since the Hilbert transform is bounded on $L^2(\R)$ (again, see Appendix~\ref{app:hilbert}), we have $\Xi^{\mu_V}(\phi) \in L^2(\R)$. Moreover, using the variational condition \eqref{V_condition}, equation~\eqref{eigeneq_1} reduces to the following for all $x\in\J$,
\beq \label{eigeneq_2}
 \Xi^{\mu_V}_x(\phi)  \, = \, \sigma \phi(x) . 
 \eeq

Our main goal is to give a proof of Theorem~\ref{th:basis}. The proof is divided into several parts. In Section~\ref{sect:preparation}, we present the general setup and give basic facts that we will need in the remainder of the proof.
In Sections~\ref{sect:regularity_1} and \ref{sect:regularity_2}, we analyze the regularity of a solution of equation  \eqref{eigeneq_1} (or actually, we will find it convenient to invert the equation  -- see \eqref{eigeneq_3} -- and consider the regularity of the solutions to the inverted equation). 
In Section~\ref{sect:regularity_1}, by a bootstrap argument, we show that a solution $\phi$ of  equation  \eqref{eigeneq_2}  is of class $\Co^{\kappa+1,1}(\overline{\J})$. In particular, note that the $\kappa+1$~first derivatives of the eigenfunctions have finite values at the edge-points $\pm 1$. 
In Section~\ref{sect:regularity_2}, by viewing \eqref{eigeneq_1} as on ODE on $\R\setminus\bar\J$, we show how to extend the eigenfunctions outside of the cut in such a way that that the eigenfunctions $\phi$ satisfy the following conditions: $\phi\in \Co^{\kappa}(\R)$ and $\phi$ solves the equation  \eqref{eigeneq_1} in a small neighborhood $\J_\epsilon$. We also control uniformly the sup norms on $\R$ of the derivatives $\phi^{(k)}$ for all $k<\kappa$.  

As mentioned, we will actually study an inverted version of \eqref{eigeneq_1}. 
In sections~\ref{sect:regularity_1} and~\ref{sect:regularity_2}, we establish the following result for the inverted equation. 

\begin{proposition} \label{pr:regularity0}
Let $\kappa\geq 1$ and  $\eta > \frac{4(\kappa+1)}{2\kappa-1}$.
Suppose that the equation \eqref{eigeneq_3} below has a solution $\phi \in \Hi$ with $\sigma>0$. Then, we may extend this solution in such a way that $\phi\in \Co^{\kappa}(\R)$ with compact support and 
it satisfies the equation \eqref{eigeneq_1} on the interval $\J_\epsilon$ where $\epsilon= \sigma^{-\eta} \wedge \delta$ and for all $k<\kappa$, 
$$
{\| \phi^{(k)}\|}_{\infty,\R}  \, \ll \,  \sigma^{\eta k } .
$$
\end{proposition}

In Section~\ref{sect:eigenfunctions}, we relate the equation \eqref{eigeneq_2} to a Hilbert-Schmidt operator $\Ro^S$ acting on $\Hi_{\mu_V}$ in order to prove the existence of the eigenvalues $\sigma_n$ and eigenfunctions $\phi_n$. 
Then, in Section~\ref{sect:expansion}, we put together the proof of  Theorem~\ref{th:basis} and get an estimate for the Fourier coefficients $\widehat{f}_n$ of the decomposition of a smooth function $f$ on $\J$ in the eigenbasis $(\phi_n)_{n=1}^\infty$.

\subsection{Properties of the Sobolev spaces \texorpdfstring{$\Hi_{\mu_V}$}{}} \label{sect:preparation}
%

We now review some basic properties of the spaces $\Hi_{\mu_V}$ along with how suitable variants of the Hilbert transform act on these spaces. We will find it convenient to formulate many of the basic properties of the space $\Hi_{\mu_V}$ in terms of Chebyshev polynomials. For definitions and basic properties, we refer to Appendix~\ref{app:hilbert}. We begin by pointing out that elements of $\Hi_{\mu_V}$ are actually H\"older continuous on $\overline{J}$ and in particular, bounded on $\overline{J}$. This follows from the simple remark that for $g\in \Hi$ and $x,y\in \J$,
\begin{equation} \begin {split} \label{continuity_1}
\big |g(x)-g(y) \big | 
     & \, \le \,  2\| g \| \sqrt{\int_x^y \varrho(dt)}  \\
       & \,= \, 2\| g \|  \big |\arcsin(x) -\arcsin(y) \big |^{\frac 12} \\
       & \, \ll \,  \|g\||x-y|^{\frac 14} . \\
\end {split} \end{equation}

A basic fact about Chebyshev polynomials of the first kind, that we recall in Appendix \ref{app:hilbert}, is that after a simple normalization, they form an orthonormal basis for $L^2(\J,\varrho)$, so in particular, we can also write for $g\in \Hi$,  
\beq \label{Fourier_1}
g \, = \,   \sum_{k\ge 1} g_k T_k 
\eeq
where $g_k$, $k \geq 1$,  are the Fourier-Chebyshev's coefficients of $g\in \Hi$ (see \eqref{eq:chebyortho}). On the other hand, since the Chebyshev polynomials of the second kind form an orthonormal basis for $L^2(\scl)$, the normalized Chebyshev polynomials of the first kind, $k^{-1}T_k$ form an orthonormal basis of $\Hi$. Combining these remarks, we see that 
\beq \label{Fourier_2}
g_k \, = \,  \frac{1}{k} \, \langle g , k^{-1}T_k \rangle = 2  \int_\J g(x) T_k(x)  \varrho(dx)  . 
\eeq
Thus by Parseval's formula (applied to the inner product of $\Hi$), we also obtain 
\beq \label{Fourier_3}
\| g \|^2  \, = \, \sum_{k\ge 1} k^2 g_k^2  \,  \ll \,  1 . 
\eeq 
From this, one can check that for $g\in \Hi$, we may differentiate formula \eqref{Fourier_1} term by term:
\beq \label{Fourier_4}
g'  \, = \,  \sum_{k\ge 1} kg_k U_{k-1} 
\eeq
where the sum converges in $L^2(\scl)$.
We define {\it the finite Hilbert transform} $\U$, for any $\phi$ for which
$$
\int_\J|\phi(x)|^p(1-x^2)^{- \frac p2}dx \, < \, \infty
$$
for some $p>1$ by
\beq \label{U_1}
\mathcal{U}_x(\phi) \, = \,  \pv \int_\J \frac{\phi(y)}{y-x} \, \varrho(dy) . 
\eeq
Note that, by boundedness of the Hilbert transform and the fact that elements of $\Hi$ are H\"older continuous on $\J$ so they are bounded, we see that for any $p<2$,  $\mathcal{U}$ maps $\Hi$ into $L^p(\R)$. Moreover,  it follows from Lemma~\ref{le:Ucheb} that if $g\in\Hi$, then
\beq \label{U_3}
\U(g) \, = \,   \sum_{k\ge 1} g_k U_{k-1} 
\eeq
where the sum converges in $L^2(\scl)$ and uniformly  on compact subsets of $\J$. Indeed, it is well known that
$|U_{k-1}(x)| \le k \wedge(1-x^2)^{-\frac 12} $ so that by \eqref{Fourier_3}
and the Cauchy-Schwarz inequality, for $x\in \J$,
$$
| \U_x(g)| \, \ll \,  (1-|x|)^{-\frac 12} \sum_{k\ge 1}  |g_k|  \, \ll \,  (1-|x|)^{-\frac 12}\| g\| . 
$$
Note that from \eqref{U_1}, one can check that $\U$ is actually a bounded operator from $\Hi$ to $L^2(\scl)$ (and thus also a bounded operator from $\Hi_{\mu_V}$ to $L^2(\scl)$).

Finally let us note that for any $f, g\in\Hi$, by \eqref{Fourier_4} and \eqref{U_3}, we obtain
\beq  \label{U_4}
\int_\J  f'(x) \, \U_x(g)  \scl(dx)  \, = \,  \int_\J \U_x(f) g'(x)  \scl(dx) \, = \,   \sum_{k\ge 1} kg_k f_k  
\eeq
implying that according to \eqref{variance_2}, we can also write
\beq \label{Sigma_1}
\S(g) \, = \,  \frac{1}{4} \int_\J g'(x) \,  \U_x(g)  \scl(dx) . 
\eeq

\subsection{Regularity of the eigenfunctions in the cut \texorpdfstring{$\J$}{J}} \label{sect:regularity_1}
In this section, we begin the proof of Proposition~\ref{pr:regularity0}.
By Tricomi's inversion formula, Lemma~\ref{lem:Tricomi}, if $\phi \in\Hi$ is a solution of the equation
\beq \label{eigeneq_3}
\phi'(x)S(x) \, = \,  \frac{\sigma}{2} \, \U_x(\phi) 
 \eeq
for all $x\in\J$, since $d\mu_V = Sd\scl$ and $\Xi^{\mu_V}(\phi) = \mathcal{H}(\phi'  \mu_V)$, then $\phi$ also solves  the equation~\eqref{eigeneq_2}. Hence, for now, we may focus on the properties of the equation   \eqref{eigeneq_3}. 
In particular, by \eqref{continuity_1}, we already know that if $\phi \in \Hi$, then $\phi$ is $\frac 14$-H\"older continuous on $\bar{\J}$ and we may use \eqref{eigeneq_3} to improve on the regularity of the eigenfunction $\phi$  by a bootstrap procedure.

\begin{proposition} \label{thm:regularity_1}
Suppose that the equation \eqref{eigeneq_3} has a solution $\phi $
belonging to $\in \Co^{0,\alpha}(\overline{\J})$ for some $\alpha>0$. If $S\in\Co^{\kappa+1}$ and $S>0$ in a neighborhood of $\J$, then $\phi\in \Co^{\kappa+1,1}(\overline{\J})$. 
\end{proposition}

The proof of Proposition~\ref{thm:regularity_1} is given below and relies on certain regularity properties of the finite Hilbert transform $\U$; see  Lemmas~\ref{lem:U_bound} and~\ref{lem:U_continuity}. 
 If $f\in\Co^k$ in a neighborhood of $x\in\R$, we denote its Taylor polynomial (a polynomial in $t$) by
$$
\Lambda_k[f](x,t) \, = \,  \sum_{i=0}^k  f^{(i)}(x) \frac{(t-x)^i}{i!}  \, .
$$
Let $\phi$ be a solution of equation \eqref{eigeneq_3} and set for all $x,t \in \J$, 
\beq \label{F_function}
F_k(x,t) \, = \,  \frac{\phi(t)- \Lambda_k[\phi](x,t) }{(t-x)^{k+1}} 
\eeq
and
\beq  \label{U_5}
\U^k_x(\phi) \, = \,  \int_\J  F_k(x,t)  \varrho(dt) . 
\eeq
In particular, by \eqref{U_0} and Lemma~\ref{lem:derivative}, if $\phi\in \Co^{k,\alpha}(\overline{\J})$ with $\alpha>0$, 
 then  for almost all $x\in \J$,
\beq  \label{U_6}
\frac{1}{k!}\left(\frac{d}{dx}\right)^k \U_x(\phi) \, = \,   \U_x^k(\phi)  . 
\eeq

The following two technical results are our key ingredients in the proof of Proposition~\ref{thm:regularity_1}, and we will present their proofs once we have proven Proposition~\ref{thm:regularity_1}.

\begin{lemma} \label{lem:U_bound}
 Let, for all $x\in\J$,
\beq \label{K}
  \K_\alpha(x) = \begin{cases}
   (1-x)^{\alpha-\frac 12} +   (1+x)^{\alpha-\frac 12} &\text{if }\alpha\neq \frac 12 \, ,  \\
1+ \log \big (\frac{1}{1-x^2}\big)  &\text{if }\alpha= \frac 12 \, .
\end{cases} 
\eeq
For any $k\ge 0$, if  $f\in \Co^{k,\alpha}(\overline{\J})$ with $0<\alpha < 1$, then for all $x\in\J$, 
\beq\label{U_7}
\left| \U^k_x(f)  \right|  \, \ll_{\alpha,f} \,  \K_\alpha(x).
\eeq
\end{lemma}

\begin{lemma}  \label{lem:U_continuity}
For any $k\ge 0$, if  $f\in \Co^{k,1}(\overline{\J})$, then $\U^k(f) \in \Co^{0,\alpha}(\overline{\J})$  for any $\alpha\le \frac 12$. 
\end{lemma}

\begin{proof}[Proof of Proposition~\ref{thm:regularity_1}]
Without loss of generality, we may assume that $\sigma=2$ in \eqref{eigeneq_3}. 
Let $0 \le k \le \kappa$ and suppose that 
$\phi\in \Co^{k,\alpha}(\overline{\J})$ where $\alpha>0$  and that  for  almost all $x\in\J$,  
\beq \label{eigeneq_6}
 \U_x^k(\phi) - S(x) \phi^{(k+1)}(x)   = \sum_{j=1}^k \binom{k}{j}  S^{(j)}(x)\phi^{(k+1-j)}(x) .
\eeq
Note that by our assumptions, these conditions are satisfied when $k=0$ in which case the RHS equals zero. Since $\phi\in \Co^{k,\alpha}(\overline{\J})$, the RHS of \eqref{eigeneq_6} is uniformly bounded, so by Lemma~\ref{lem:U_bound} and since $S>0$ on $\bar\J$, we see that for all $x\in\J$,
$$
\big|  \phi^{(k+1)}(x)  \big|  \, \ll_\alpha \,   \K_\alpha(x) . 
$$
By the definition of $\K_\alpha$, this bound implies that  
$\phi^{(k)} \in \Co^{0,\alpha'}(\overline{\J})$ where
$\alpha' = (\alpha+ \frac 12) \wedge 1 $. By repeating this argument, we obtain that $\phi\in \Co^{k,1}(\overline{\J})$ and, by Lemma~\ref{lem:U_continuity}, this implies that $\phi ^{(k+1)} \in \Co^{0,\alpha}(\overline{\J})$. Then, \eqref{U_6} shows that we can differentiate formula \eqref{eigeneq_6}. In particular, this shows that $\phi^{(k+2)}$ exists and that for almost all  $x\in\J$, 
$$
S(x) \phi^{(k+2)}(x)  \, = \,  \U_x^{k+1}(\phi) - \sum_{j=1}^{k+1} \binom{k+1}{j} S^{(j)}(x)\phi^{(k+2-j)}(x) . 
$$
Thus, we obtain  \eqref{eigeneq_6} with $k$ replaced by  $k+1$.  Therefore, we can repeat this argument until $k=\kappa+1$, in which case we have seen that the solution  $\phi\in \Co^{\kappa+1,1}(\overline{\J})$. 
Proposition~\ref{thm:regularity_1} is established.
\end{proof}

The remainder of this subsection is devoted to the proof of Lemmas~\ref{lem:U_bound} and~\ref{lem:U_continuity} as well as bounding norms of derivatives of solutions of \eqref{eigeneq_3} in terms of $\sigma$.

\begin{proof}[Proof of Lemma~\ref{lem:U_bound}]
Let $\phi\in \Co^{k,\alpha}(\overline{\J})$. We claim that for any $x,y\in \J$,      
\begin{equation*}
\big| \phi(t)- \Lambda_k[\phi](x,t) \big| \, \ll_\phi \,  |t-x|^{k+\alpha} .
\end{equation*}
Indeed, this is just the definition of the class  $\Co^{0,\alpha}(\overline{\J})$ when $k=0$. 
Moreover, if $k\geq 1$ it can be checked that
$$
\phi(t)- \Lambda_k[\phi](x,t) \, = \,
     (t-x)  \int_0^1\Big [ \phi'\big(x+ (t-x)u\big)- \Lambda_{k-1}[\phi']\big(x, x+ (t-x)u\big) \Big] du 
$$
so that the estimate follows directly by induction. 
According to \eqref{F_function}--\eqref{U_5}, this implies that 
\beq \label{estimate_1}
\big| F_k(x,t) \big| \, \ll_\phi \,   |t-x|^{\alpha-1} 
\eeq
and, by splitting the integral,
\beq \label{estimate_2}
  \big| \U_x^k(f) \big|  \, \ll_\phi \,   \int_{-1}^x \frac{(x-t)^{\alpha-1}}{\sqrt{1-t^2}} \, dt 
+ \int_x^1 \frac{(t-x)^{\alpha-1}}{\sqrt{1-t^2}} \, dt .
\eeq
The RHS of \eqref{estimate_2} is finite  for all $x\in\J$ and,
by symmetry, we may assume that $x\ge0$. Then,  we have
\beq \begin {split} \label{estimate_3}
 \int_x^1 \frac{(t-x)^{\alpha-1}}{\sqrt{1-t^2}} \, dt 
 & \, \le \,  \int_x^1 \frac{(t-x)^{\alpha-1}}{\sqrt{1-t}} \, dt \\
 &  \, = \, (1-x)^{\alpha-\frac 12} \int_0^1 u^{\alpha-1}(1-u)^{-\frac 12}dt \\
 &  \, \ll_\alpha \,  \K_\alpha(x) . 
\end{split} \eeq
On the other hand, we have
$$
\int_{-1}^x \frac{(x-t)^{\alpha-1}}{\sqrt{1-t^2}} \, dt 
   \, \le \,  \int_{0}^1 \frac{t^{\alpha-1}}{\sqrt{1-t^2}} \, dt 
		+  \int_{0}^x \frac{(x-t)^{\alpha-1}}{\sqrt{1-t}} \, dt  . 
$$
The first integral does not depend on $x\ge0$, and by a change of variable
$$
\int_{-1}^x \frac{(x-t)^{\alpha-1}}{\sqrt{1-t^2}} \, dt 
  \, \le \,  C+ (1-x)^{\alpha-\frac 12} \int_0^{\frac{x}{1-x}} u^{\alpha-1}(1+u)^{-\frac 12}  du .
$$
It is easy to verify that
$$
 \int_0^{\frac{x}{1-x}} u^{\alpha-1}(1+u)^{-\frac 12} du 
=  \underset{x\to1}{O}\begin{pmatrix}
1  &\text{if } \alpha<\frac 12 \\
\log \big (\frac{1}{1-x}\big)   &\text{if } \alpha = \frac 12  \\
(1-x)^{\frac{1}{2}-\alpha}   &\text{if } \alpha > \frac 12  
\end{pmatrix}
$$
which shows that
\beq \label{estimate_4}
\int_{-1}^x \frac{(x-t)^{\alpha-1}}{\sqrt{1-t^2}}dt  \, \ll_\alpha \,  \K_\alpha(x) . 
\eeq
Combining the estimates \eqref{estimate_2}--\eqref{estimate_4}, we obtain \eqref{U_7}.
\end{proof}

To finish the proof of Proposition \ref{thm:regularity_1}, we conclude with the proof of Lemma~\ref{lem:U_continuity}. 

\begin{proof}[Proof of Lemma~\ref{lem:U_continuity}]
We claim that, if $\phi\in \Co^{k,1}(\overline{\J})$, then the function $x \mapsto F_{k-1}(x,t)$ is $\Co^{0,1}(\overline{\J})$. Indeed,
by Taylor's formula, we have
\beq \label{Taylor}
F_{k-1}(x,t) \,  = \, \int_0^1 \phi^{(k)}\big(t(1-u) + xu\big)  \frac{u^{k-1}}{(k-1)!} \, du ,
\eeq
 so that
\beq \begin{split} \label{estimate_5}
\big|  F_{k-1} &(t,y) - F_{k-1}(t,x) \big| \\
& \, \le \,  \int_0^1 \Big | \phi^{(k)}\big(t(1-u) + yu\big)  -  \phi^{(k)}\big(t(1-u) + xu\big) \Big|
     \frac{u^{k-1}}{(k-1)!}  \, du   \\
& \, \ll  \,  |y-x| 
 \end{split} \eeq
where the underlying constant is independent of the parameter $t\in\J$.
Since
$$
 F_k(x,t) \, = \, \frac{F_{k-1}(x,t) - \frac{1}{k!}\phi^{(k)}(x)}{t-x} \,,
$$
we have that  
\begin{align*}
&F_k(y,t) - F_k(x,t)\\
 \, &= \,  \left( F_{k-1}(y,t) - \frac{1}{k!}\phi^{(k)}(y)- F_{k-1}(x,t) + \frac{1}{k!}\phi^{(k)}(x) +(y-x) F_k(x,t)   \right)\frac{1}{t-y} \, . 
\end{align*}

Then, using the estimates \eqref{estimate_1} with $\alpha=1$ and \eqref{estimate_5}, we obtain that uniformly for all $x,t \in \bar{\J}$, 
$$
\big| F_k(t,y) - F_k(t,x)  \big|  \, \ll \,   1\wedge \left( \frac{|y-x|   }{|t-y|\vee |t-x|} \right) .  
$$
This implies that for all $-1<x<y <1$, 
\beq \begin{split} \label{estimate_6} 
 \big|  \U^k_x(\phi)  -   \U^k_y(\phi) \big|
 & \, \ll \,  (y-x) \left( \int_{t<x} \frac{\varrho(dt)}{y-t} 
      + \int_{y<t}   \frac{\varrho(dt)}{t-x} \right)  + \int_x^y \varrho(dt) \\
 & \, \ll \,  \sqrt{y-x} \, \big( 1+ \theta(x,y) \big)
\end{split} \eeq
where
\beq \begin{split} \label{estimate_7} 
\theta(x,y) 
  & \, = \,  \sqrt{y-x} \left( \int_{t<x} \frac{\varrho(dt)}{y-t}  + \int_{y<t}   \frac{\varrho(dt)}{t-x} \right) \\
 & \, = \,   \sqrt{y-x}  \, \frac{\sqrt{1-y^2} + \sqrt{1-x^2} }{\sqrt{1-y^2}\sqrt{1-x^2}}  
\log\bigg(  \frac{1-xy+ \sqrt{1-x^2} \sqrt{1-y^2}}{y-x} \bigg) .
\end{split} \eeq
Formula \eqref{estimate_7} follows by explicitly evaluating the integrals. This function extends by continuity to the domain $\mathcal{D} = \{-1<x \le y <1 \}$ with ${\theta(x,x) =0}$. 
By \eqref{estimate_6}, it just remains to prove that the function $\theta$ is bounded on $\overline{\mathcal{D}}$. First, it is easy to check that 
$$ 
\lim_{y\to 1}\theta(x,y)  \, = \, \sqrt{1+x} 
 \qquad\text{and}\qquad
 \lim_{x\to -1}   \theta(x,y) \, = \, \sqrt{1-y} . 
$$
To study the boundary values of $\theta$ around the point $(1,1)$, we make the change of variables 
$$
\widetilde{\theta}(\epsilon, \tau) 
  \, = \, \theta\big(\sqrt{1-\tau^2\epsilon^2}, \sqrt{1-\epsilon^2} \, \big)
$$
where $0<\epsilon<1$ and $1< \tau < \frac1\epsilon $.
Then $y = \sqrt{1-\epsilon^2} \underset{\epsilon\to0}{\simeq}1- \frac {\epsilon^2}{2}$ and  
 $x = \sqrt{1-\tau^2\epsilon^2} \underset{\epsilon\to0}{\simeq}1- \frac{\tau^2 \epsilon^2}{2}$
 so that 
$$
\lim_{\epsilon\to0} \widetilde{\theta}(\epsilon, \tau) 
 \, = \, \g(\tau) \,= \, \sqrt{\frac{\tau^2 - 1}{2}} \,
     \frac{\tau+1}{\tau} \,  \log\left(\frac{ \tau+1}{\tau -1}\right) . 
$$
The function  $\g(\tau)$ is continuous on the interval $(1,\infty)$ with boundary values 
$\lim_{\tau\to1} \g(\tau) = 0$  and 
$\lim_{\tau\to\infty} \g(\tau) = \sqrt{2}$.  In particular, this shows that 
$$
\limsup_{x\to 1, y\to1} \theta(x,y)  \, = \,  \max_{\tau \ge 1} \g(\tau) <\infty . 
$$
A similar analysis in a neighborhood of the point $(-1, -1)$ allows us to conclude that the function $\theta$ is bounded on $\overline{\mathcal{D}}$. Lemma~\ref{lem:U_continuity} is established.
\end{proof} 

We are now in a position to provide some quantitative bounds on solutions to \eqref{eigeneq_3}. Our implicit assumption is that $\sigma$ is not small -- an assumption that will be fulfilled in our applications of the results we prove.

\begin{proposition} \label{thm:bound_1}
Suppose that the equation  \eqref{eigeneq_3} has a solution ${\phi \in \Hi}$ such that $\|\phi\|_{\infty,\J} \ll 1$ $($this bound does not depend on the parameter $\sigma>0)$. Then, we have for any
$0<\alpha<\frac12$ and for all  $k\le \kappa+1$,
\beq  \label{estimate_8}
{\|\phi^{(k)} \|}_{\infty,\J}  \, \ll_\alpha \,  \sigma^{ \frac k\alpha} . 
\eeq
\end{proposition}

\begin{proof}  By \eqref{continuity_1}, we already know that $\phi\in \Co^{0,\frac14}(\overline{\J})$. Then, by Proposition~\ref{thm:regularity_1}, we obtain that ${\phi\in \Co^{\kappa+1,1}(\overline{\J})}$.
By our assumptions, we note that the bound \eqref{estimate_8} holds when $k=0$ and we proceed by induction on $k \le \kappa$ to prove the general case. 
If we differentiate $k$ times the equation \eqref{eigeneq_3}, we obtain 
  \beq \label{estimate_9}
 S(x) \phi^{(k+1)}(x)   \,= \,  \frac{\sigma}{2} \, \U_x^k(\phi)
     - \sum_{j=1}^k \binom{k}{j-1} S^{(k+1-j)}(x)\phi^{(j)}(x) .
\eeq
By Lemma~\ref{lem:U_continuity}, the RHS of \eqref{estimate_9} is continuous for all $x\in\bar{\J}$ and, since $S>0$ on $\bar{\J}$, this shows that
  \beq  \label{estimate_10}
{\|\phi^{(k+1)}\|}_{\infty,\J}  \, \ll \, \sigma \, {\| \U^k(\phi) \|}_{\infty,\J} + \sigma^{\frac k\alpha}  . 
  \eeq
Here, we assumed that the estimate \eqref{estimate_9} is valid for all $j \le k$ and
that the parameter $\sigma$ is large.
Recall that, according to \eqref{Taylor}, we have the bound
$ \sup_{x,t\in \J}\big| F_{k-1}(t,x)\big| \le \| \phi^{(k)}\|_{\infty, \J}$.
Since
$$F_k(x,t) \, = \, \frac{F_{k-1}(x,t) - \frac{1}{k!}\phi^{(k)}(x)}{t-x} \, ,
$$
this shows that for any $\alpha\in[0,1]$ and for $x, t\in \bar{\J}$, 
\beqs
 \big|  F_k(t,x) \big|
 \, \ll \,  \frac{\| \phi^{(k+1)}\|_{\infty, \J}^{1-\alpha}  \, \|\phi^{(k)}\|_{\infty, \J}^{\alpha}}{|x-t|^{\alpha}} 
 \, \ll \,  \sigma^k \, \frac{\| \phi^{(k+1)}\|_{\infty, \J}^{1-\alpha}}{|x-t|^{\alpha}} \, ,
\eeqs
where we used the estimate \eqref{estimate_8} once more. 
By \eqref{U_5}, this shows that for any $\alpha< \frac 12$,
$$
 {\|  \U^k(\phi) \|}_{\infty, \J}  \, \ll \,  \sigma^k {\| \phi^{(k+1)}\|}_{\infty, \J}^{1-\alpha} .
$$ 
Combining this bound with the estimate \eqref{estimate_10}, we conclude that
$$
{\|\phi^{(k+1)}\|}_{\infty,\J}  \ll \sigma^{k+1} {\| \phi^{(k+1)}\|}_{\infty, \J}^{1-\alpha}
    + \sigma^{\frac k\alpha} 
$$
Then either $\|\phi^{(k+1)}\|_{\infty,\J}  \ll \sigma^{\frac k\alpha}$ or 
$ \|\phi^{(k+1)}\|_{\infty,\J}^\alpha  \ll \sigma^{k+1}$,  
which completes the proof.
\end{proof}

\subsection{Extension of the eigenfunctions on  \texorpdfstring{$ \R \setminus \bar {\J}$}{the complement of the cut}}   \label{sect:regularity_2}

The aim of this subsection is to complete the proof of Proposition~\ref{pr:regularity0} by extending the
eigenfunction $\phi$ ouside of the cut in such a way that the resulting function, still denoted by $\phi$,
satisfies  the equation~\eqref{eigeneq_1} in a neighborhood $\J_{\epsilon}$ of the cut and 
that ${ \phi\in\Co^{\kappa}(\R)}$.   
We will focus on the extension of $\phi$ in a neighborhood of the edge-point $1$, the construction being completely analogous in a neighborhood of $-1$.
Suppose that $\phi\in \Co^{\kappa+1, 1}(\overline{\J})$ and let for all $x \in \R$, 
\beq \label{G_function}
\g(x) \, = \,  \pv \int \frac{\phi'(t)}{x-t} \, \mu_V(dt) . 
\eeq
Note that $\g$ is smooth on  $\R\setminus \bar\J$ and that it  does not depend on the values of $\phi$ outside of cut. Thus, we may interpret the equation \eqref{eigeneq_1} as an ODE  satisfied by
the function $\y= \phi_{|\R\setminus \bar\J}$ :
\beq\label{eigeneq_7}
Q' \y'  \, = \,  \sigma \y- \g,  
\eeq
where $Q$ is as in \eqref{eq:el1}. The condition \eqref{V_condition} guarantees that $Q'>0$ and the equation \eqref{eigeneq_7} is well-posed for all  $x\in\J_{\delta} \setminus \bar\J$.  Moreover, 
by Proposition~\ref{thm:s},  since $ \frac {1}{Q'}$ is integrable in a neighborhood of $1$, the equation \eqref{eigeneq_7} has a unique solution such that $\y(1) =\varkappa_0$ for any $\varkappa_0\in\R$.  
Choosing $\varkappa_0= \phi(1)$, in order to extend outside of the cut the solution $\phi$ of the equation \eqref{eigeneq_2} in a smooth way, we need to check  that $\y$ satisfies  for all $1 \le k\le \kappa$, 
\beq  \label{ODE_1}
\lim_{x\to 1^+} \y^{(k)}(x) \, = \,  \lim _{x\to1^-} \phi^{(k)}(x) . 
\eeq
The conditions \eqref{ODE_1} are checked using the following result and Proposition~\ref{thm:bound_2} below. Again, in our quantitative regularity bounds in terms of $\sigma$, we will be assuming that $\sigma$ is large.

\begin{proposition} \label{thm:regularity_2}
Let $\kappa\geq 1$ and suppose that { $\eta >  \frac {4(\kappa+1)}{2\kappa-1}$}. 
Let $G$ in $\Co^{\kappa+1}(1,\infty)$ be
such that $G(x) =o_{x\to1^+}(x-1)^{\kappa} $ and let $Y$ be the solution of the equation 
\beq \label{ODE_2}
Q' Y'  \, = \,  \sigma Y + G 
\eeq
on the interval $(1,\infty)$ with boundary condition $Y(1) =0$. Then, 
$Y$ belongs to $\Co^{\kappa+2}(1,1+\delta)$ and satisfies  for all $k\le\kappa$,
$$
Y^{(k)}(1) \, = \,  \lim_{x\to 1^+} Y^{(k)}(x)=0 .
$$ 
Moreover, letting $\epsilon = \sigma^{-\eta} \wedge \delta$,  if we assume that for any $\frac1\eta<\alpha<\frac 12$,
\beq \label{estimate_11}
{\| G^{(\kappa-1)} \|}_{\infty, [1, 1+\epsilon]}  \, \ll_\alpha \, \sigma^{\frac {\kappa+1}{\alpha}} ,
\eeq
 then  we have for all $k<\kappa$, 
\beq \label{estimate_12}
{\|Y^{(k)}\|}_{\infty, [1, 1+\epsilon]} \, \ll \,  \sigma^{\eta k} .
\eeq
\end{proposition}

\begin{proof}
If we let, for all $1 \le x<1+\delta$,
\begin{align} \label{ODE_4}
 \mathrm{h}(x) \, = \,  \int_1^x \frac{\sigma}{Q'(t)} dt 
\quad\text{and}\quad
 \Upsilon(x) \, = \,  \int_1^x \frac{G(t)}{Q'(t)} e^{- \mathrm{h}(x) } dt ,
 \end{align}
then the  solution of the equation  \eqref{ODE_2} for which $Y(1)=0$ is given by
\beq \label{ODE_3}
Y(x) \, = \,  \Upsilon(x) e^{\mathrm{h}(x)} . 
\eeq
Then, if $G\in \Co^{\kappa+1}(1,\infty)$, we immediately check that $Y\in \Co^{\kappa+2}(1,1+2\delta)$. 
Moreover, by Proposition~\ref{thm:s}, the condition $G(x) =O_{x\to1^+}(x-1)^{\kappa} $ implies that 
\beq \label{ODE_5}
\big|\Upsilon(x) \big| \, \ll \,  (x-1)^{\kappa+\frac{1}{2}} .
\eeq
Since the function $\mathrm{h}(x)$ is continuous on the interval $[1,1+\delta]$ with $\mathrm{h}(0)=0$, we conclude that the estimate \eqref{ODE_5} is also satisfied by the solution $Y$ which proves that it is of class $\Co^\kappa$ at 1 and that $Y^{(k)}(1)=0$ for all $k\le \kappa$. 

Now, we turn to the proof of \eqref{estimate_12}. 
By Taylor's theorem, since ${G^{(k)}(1)=0}$ for all $k\le \kappa$,
$$
 G^{(j)}(t) \, = \,  (t-1)^{\kappa-1-j}   \int_0^1 G^{(\kappa-1)}\big(1 -u(1-t)\big) 
    \frac{u^{\kappa-2-j}}{(\kappa-2-j)!} \, du ,
$$
and the estimate \eqref{estimate_11} implies that for all $j\le \kappa-1$ and for all $1<x<\epsilon$, 
$$
\big| G^{(j)}(t) \big|  \, \ll \,   \sigma^{\frac {\kappa+1}{\alpha}}  (t-1)^{\kappa-1-j}   .
$$
On the other hand, by Proposition~\ref{thm:s}, for all $j\le \kappa$, 
$$
\frac{d^j}{dx^j} \Big ( \frac {1}{Q'(x)} \Big) \, =  \, \underset{x\to1^+}{O}(x-1)^{-\frac{1}{2}-j} ,
$$
so that we deduce from \eqref{ODE_4} that for any $k <\kappa$ and for all $1\le x <1+\delta$, 
\beq \begin{split} \label{estimate_13}
\mathrm{h}^{(k)}(x) &\, \ll \,  \sigma (x-1)^{\frac{1}{2}-k}  \\
\big|\Upsilon^{(k)}(x) \big| &\, \ll \,   \sigma^{(\kappa+1)/\alpha} (x-1)^{\kappa-\frac{1}{2}-k} . 
\end{split}
\eeq
Combined with \eqref{ODE_3}, the estimates \eqref{estimate_13} imply that if $\epsilon = \sigma^{-\eta} \wedge \delta$, we have for all $k<\kappa$, 
\begin{equation*}
{\| Y^{(k)} \|}_{\infty, [1, 1+\epsilon]} \, \ll \,  \sigma^{(\kappa+1)/\alpha- \eta(\kappa -\frac{1}{2}-k)} .
\end{equation*}
Now, it is easy to check that { if $\eta >  \frac {4(\kappa+1)}{2\kappa-1}$}, we can choose $\frac1\eta<\alpha<\frac 12$ 
so that $\frac{\kappa+1}{\alpha}\le \eta (\kappa-\frac 12) $. Hence,  we obtain the
estimate \eqref{estimate_12}, completing the proof of the proposition.
\end{proof}

The second ingredient of our extension of the eigenfunctions is the following result.

\begin{proposition} \label{thm:bound_2}
Suppose that $\phi\in \Co^{\kappa+1,1}(\overline{\J})$ is a solution of the equation \eqref{eigeneq_2} on $\J$. For all $k\le \kappa+1$, define $\varkappa_k= \lim_{x\to1^-} \phi^{(k)}(x)$, and let
$\Omega(x) = \sum_{i=0}^{\kappa+1} \varkappa_i \frac{(x-1)^i}{i!}$. Then, the function $\g$ given by \eqref{G_function} satisfies
\beq \label{G_expansion}
\g(x) \, = \, - \, \Omega'(x)Q'(x) + \sigma \, \Omega(x)   + \underset{x\to1}{o}(x-1)^{\kappa} .
\eeq
Moreover,  if for any $0<\alpha < \frac 12$, we have
${\|\phi^{(k)} \|}_{\infty,\J}  \, \ll_\alpha \,  \sigma^{\frac k\alpha} $ for all $k\le \kappa+1$, then the function $G(x) = \g(x) + Q'(x)\Omega'(x) - \sigma \Omega(x)$ satisfies \eqref{estimate_11} $($and we can apply the previous proposition$)$. 
\end{proposition}

\begin{proof}
On the one-hand, according to \eqref{eigeneq_2} and \eqref{G_function}, 
we have ${\g(x) =\sigma \phi(x)}$ for all $x\in\J$. This shows that 
$\g\in \Co^{\kappa+1}(\bar\J)$ and
\beq \label{G_3}
\g(x) \, = \,  \sigma \Omega(x)   + \underset{x\to1^-}{O}(x-1)^{\kappa+1} .
\eeq
From the variational condition \eqref{V_condition}, $Q' = V'-\mathcal{H}(\mu_V) = 0$ on $\bar{\J}$ and we see that \eqref{G_expansion} holds when the
limit is taken from the inside of the cut.
On the other-hand, we have for all $x\in\J$,
\beq \label{G_1}
\g(x) =  V'(x)\phi'(x) -\Theta(x) 
\qquad\text{where}\quad
\Theta(x) \, = \, \int_\J \frac{\phi'(t)-\phi'(x)}{t-x} \, \mu_V(dt)   . 
\eeq
By Lemma~\ref{lem:derivative}, since $\phi\in \Co^{\kappa+1,1}(\overline{\J})$,  the function $\Theta \in\Co^{\kappa}(\bar\J)$ and for all $k\le \kappa$,
\beq \label{G_2}
\Theta^{(k)}(1) \, = \,  \lim_{x\to 1^{-}}\Theta^{(k)}(x) 
  \, = \,   k!\int \frac{\phi'(t) -\Lambda_k[\phi'](1,t)}{(t-1)^{k+1}} \, \mu_V(dt) . 
\eeq
Hence, by \eqref{G_3} and \eqref{G_1}, we obtain that 
\beq \label{G_4}
\Theta^{(k)}(1) \, = \,   \frac{d^k(V'\phi')}{dx^k}\Big|_{x=1^-}  - \sigma \varkappa_k  \, . 
\eeq

Define for all $x\ge 1$,
$$
\widetilde{\Theta}(x) \, = \,  \int \frac{ \phi'(t) -\Omega'(x)}{t-x} \,  \mu_V(dt) ,
$$
so that
\beq \label{G_5}
\g(x) \, = \,  \mathcal{H}_x(\mu_V)\Omega'(x)  - \widetilde{\Theta}(x) . 
 \eeq
It is also easy to check that the function
$\widetilde{\Theta}\in \Co^\infty(1,\infty)$ and that for all $k\ge 1$ and $x>1$, 
\beq \label{G_6}
\widetilde{\Theta}^{(k)}(x) 
\, = \,  k!\int \frac{\phi'(t) -\Lambda_k[\Omega'](x,t)}{(t-x)^{k+1}} \, \mu_V(dt) . 
\eeq
Since $\Lambda_k[\Omega'](1,t) = \Lambda_k[\phi'](1,t)$, by \eqref{G_2}, we obtain for all $k\le\kappa$, 
\beq  \label{G_7}
\lim_{x\to 1^+}\widetilde{\Theta}^{(k)}(x) \, = \,  \Theta^{(k)}(1) .
\eeq
According to \eqref{G_5}, this shows that for all $x\ge 1$,
\beq  \begin {split} \label{G_8}
G(x) & \, = \,  \g(x) + Q'(x)\Omega'(x) - \sigma \, \Omega(x)  \\
	& \,  = \,  -\widetilde{\Theta}(x) + V'(x)\Omega'(x)  - \sigma \Omega(x)  . \\
\end {split} \eeq
By  \eqref{G_4} and \eqref{G_7},  since 
$\Theta^{(k)}(1) =  \frac{d^k(V'\Omega')}{dx^k}\big|_{x=1}  - \sigma \varkappa_k$, we see that there is a cancellation on the RHS of \eqref{G_8}
 so that  $\lim_{x\to1^+} G^{(k)}(x) = 0$ for all $k\le \kappa$.  This implies that 
$ G(x)=  o_{x\to1^+}(x-1)^{\kappa} $ and we obtain the expansion \eqref{G_expansion}. 

Now, we turn to the proof of the estimate \eqref{estimate_11}. 
We let  $\epsilon = \sigma^{-\eta}$ where ${\eta> \frac 1\alpha}$ 
and the parameter $\sigma$ is assumed to be large. 
By hypothesis,  ${|\varkappa_k| \ll_\alpha \sigma^{\frac k\alpha}}$ and it is easy to verify that for all $k \le \kappa+1$, 
\beq \label{estimate_0}
{\|\Omega^{(k)}\|}_{\infty, [1, 1+\epsilon]} \, \ll \, \sigma^{\frac k\alpha} . 
\eeq
Thus, by \eqref{G_8},  it suffices to show that 
\beq \label{estimate_14}
{\| \widetilde{\Theta}^{(\kappa-1)} \|}_{\infty, [1, 1+\epsilon]}  
    \, \ll_\alpha \, \sigma^{\frac {\kappa+1}{\alpha}} .
\eeq
Since   $\Lambda_{\kappa-1}[\Omega'](1,t) = \Lambda_{\kappa-1}[\phi'](1,t)$, by 
\eqref{G_6}, we obtain for all $x\ge 1$, 
\beq  \begin {split} \label{estimate_15} 
& \big| \widetilde{\Theta}^{(\kappa-1)}(x)  \big|  \\
 &\, \ll \,  \sup_{t\in\J} \left\{  \left|  \frac{\phi'(t)- \Lambda_{\kappa-1}[\phi'](1,t) }{(1-t)^{\kappa}} \right|
+ \left| \frac{\Lambda_{\kappa-1}[\Omega'](x,t) -\Lambda_{\kappa-1}[\Omega'](1,t)}{(x-t)^{\kappa}} \right|   \right\} . 
\end {split} \eeq
On the one-hand, by Taylor's theorem, we have 
$$
\sup_{t\in\J} \left|  \frac{\phi'(t)- \Lambda_{\kappa-1}[\phi'](1,t) }{(1-t)^{\kappa}} \right|
  \, \le  \, {\|  \phi^{(\kappa+1)} \|}_{\infty, \J}
 \, \ll \, \sigma^{\frac {\kappa+1}{\alpha}} .
$$
On the other-hand, if  $f$ is a smooth function, for all $k\ge 0$,
$$
\frac{d\Lambda_k[f](x,t)}{dx} \, = \,  \frac{f^{(k+1)}(x)}{k!} \, (x-t)^{k+1} ,
$$
and by the mean-value theorem, there exists $t<\xi <x$ such that
  $$
\frac{\Lambda_k[f](x,t) -\Lambda_k[f](1,t)}{(t-x)^{k+1}} 
   \, = \,  \frac{f^{(k+1)}(\xi)}{k!}\left(\frac{\xi-t}{t-x}\right)^{k+1} . 
$$
Therefore, we have for all $x>1$, 
$$
\left| \frac{\Lambda_{\kappa-1}[\Omega'](x,t) -\Lambda_{\kappa-1}[\Omega'](1,t)}{(t-x)^{\kappa}} \right| 
  \,  \ll \,   {\| \Omega^{(\kappa+1)}\|}_{\infty} \, = \,  |\varkappa_{\kappa+1}|,
$$
as we defined, $\Omega(x)=\sum_{i=0}^{\kappa+1}\varkappa_i \frac{(x-1)^i}{i!}$. Since
$|\varkappa_{\kappa+1}|  \ll \sigma^{\frac {\kappa+1}{\alpha}}$, we deduce from the estimate \eqref{estimate_15} that
$\| \widetilde{\Theta}^{(\kappa-1)} \|_{\infty, [1,\infty)}  \ll_\alpha \sigma^{\frac{\kappa+1}{\alpha}}$ 
which obviously  implies \eqref{estimate_14} and completes the proof.
\end{proof}

We are now ready for the proof of Proposition~\ref{pr:regularity0}. \\

\noindent
{\it Proof of Proposition~\ref{pr:regularity0}.} 
First, if the equation \eqref{eigeneq_3} has a solution $\phi \in \Hi$ with $\sigma>0$,
by Propositions~\ref{thm:regularity_1} and~\ref{thm:bound_1}, this solution
$\phi\in \Co^{\kappa+1,1}(\bar\J)$ and satisfies for all  $k\le \kappa+1$,
\beq  \label{estimate_16}
{\|\phi^{(k)} \|}_{\infty,\J} \, \ll_\alpha \, \sigma^{\frac k\alpha} . 
\eeq
Then, let us write as in Proposition \ref{thm:bound_2}: $\varkappa_i= \lim_{x\to1^-} \phi^{(i)}(x)$ and 
$\Omega(x) = \sum_{i=0}^{\kappa+1} \varkappa_i \frac{(x-1)^i}{i!}$,
so that making the change of variable $\y(x) = Y(x)- \Omega(x)$ in \eqref{eigeneq_7},  we obtain
$$
Q' Y'  \, = \,  \sigma Y + G ,
$$
where as before, $G(x) = \g(x) + Q'(x)\Omega'(x) - \sigma \Omega(x)$. 
By Tricomi's formula, Lemma~\ref{lem:Tricomi}, the function $\phi$ also solves the
equation \eqref{eigeneq_2} and according to Proposition~\ref{thm:bound_2},
$G(x) =o_{x \to 1^+} (x-1)^{\kappa}$
and we deduce from Proposition~\ref{thm:regularity_2} that for all $k\le \kappa$,
$$ 
 \lim_{x\to 1^+} \y^{(k)}(1) \, = \,  \lim_{x\to 1} \Omega^{(k)}(x) \, = \, \varkappa_k .
 $$ 
This shows that the conditions \eqref{ODE_1} are satisfied, and 
that we may extend the solution $\phi$ outside of~$\J$ by setting $\phi(x)= \y(x)$ for all $x\in\J_\epsilon \setminus \J$ in such a way that it satisfies \eqref{eigeneq_1} on $\J_\epsilon$ for some $\epsilon\le\delta$  and $\phi\in\Co^{\kappa}(\R)$. 
Moreover, the estimate  \eqref{estimate_16} and Proposition~\ref{thm:bound_2} imply that 
the function~$G$ satisfies \eqref{estimate_11} so that
$\|Y^{(k)}\|_{\infty, [1, 1+\epsilon]} \ll \sigma^{\eta k}$ for all $k<\kappa$,
where  $\epsilon = \sigma^{-\eta} \wedge \delta$ and { $\eta >  \frac {4(\kappa+1)}{2\kappa-1}$}. 
Then, by choosing the parameter $\alpha> \frac1\eta$ in the estimates \eqref{estimate_0} and \eqref{estimate_16}, we  obtain that $\|\phi^{(k)}\|_{\infty, [-1, 1+\epsilon]} \ll \sigma^{\eta k}$. 
Moreover, an analogous analysis in the neighborhood of the other edge-point $-1$ allows us
to conclude that for all $k<\kappa$, 
\beq \label{estimate_17}
{\|\phi^{(k)}\|}_{\infty, \J_\epsilon} \, \ll \, \sigma^{\eta k}.
\eeq

Now, let $\chi: \R \to [0,1]$ be a smooth even function with compact support 
in $ \J = (-1,1)$ such that $\chi(0)=1$ and  $\chi^{k}(0)=0$ for all $1<k\le \kappa$, and recall that $t \mapsto\Lambda_\kappa[\phi](x,t)$ denotes the Taylor polynomial of the function $\phi$ of degree $\kappa$ at $x$. Then, if we set for all $t \in \R\setminus\overline{\J_\epsilon}$, 
$$
\phi(t) \, = \,  \chi\left(\frac{t-1-\epsilon}{\epsilon}\right)\Lambda_\kappa[\phi](1+\epsilon,t) + \chi\left(\frac{t+1+\epsilon}{\epsilon}\right)\Lambda_\kappa[\phi](-1-\epsilon,t) ,
$$
then we indeed have $\phi\in\Co^{\kappa}(\R)$ and it is  easy to check that  the estimate \eqref{estimate_17} remains true for the norms~$\|\phi^{(k)}\|_{\infty,\R}$ and for all $k<\kappa$. 
Finally, note that this extension procedure guarantees that the function $\phi$ has support in $\J_{2\epsilon}$. 
Proposition~\ref{pr:regularity0} is proved. \qed

  \subsection{Existence of the eigenfunctions} \label{sect:eigenfunctions}
In order to prove that there exists a sequence of eigenfunctions $\phi_n \in \Hi$, 
the equation \eqref{eigeneq_3} suggests to consider the operator $\Ro^{S}(\phi) = \psi$
  where $\psi$ is the (weak) solution in $\Hi$ of  the equation
\beq \label{R_1}
\psi' S  \, = \, \mathcal{U}(\phi) .
\eeq
Note that if $\phi \in \Hi$, since $\mathcal{U}(\phi)$ is continuous on $\J$, there exists a solution and the condition ${\int_\J \psi(x) \varrho(dx) =0}$ guarantees that it is unique.

\begin{theorem} \label{thm:R_operator}
The operator $\Ro^{S}$ is compact, positive, and self-adjoint on $ \Hi_{\mu_V} $.
If we denote by $(\frac {2}{\sigma_n})_{n=1}^\infty$ the eigenvalues of $\Ro^{S}$ in non-increasing order, then  $\sigma_n \asymp n$ and the corresponding $($normalized$)$ eigenfunctions $\phi_n$ satisfy
$\| \phi_n\|_{\infty,\J} \ll 1$   for all $n\geq 1$.
\end{theorem}

\begin{proof}
Since the finite Hilbert transform $\mathcal{U}: \Hi \mapsto L^2(\scl)$ is bounded (as we remarked in Section \ref{sect:preparation}) and $S>0$ on $\bar\J$, by \eqref{R_1} we obtain
\beq  \label{R_2}
{\big \| \Ro^S(\phi) \big \|}_{\mu_V}^2 \, = \, 
    \int_{\J} \U_x(\phi)^2  \, \frac{2\sqrt{1-x^2}}{\pi S(x)} \, dx \, \ll \,
      {\big \| \U(\phi) \big \|}_{L^2(\scl)}^2 , 
\eeq
so that $\Ro^{S} : \Hi_{\mu_V} \mapsto \Hi_{\mu_V}$ is a bounded operator. Moreover, by \eqref{U_4},
for any $\phi, g\in \Hi$,
$$
{\langle \Ro^{S}(\phi) , g \rangle}_{\mu_V} 
 \, = \,  \int_\J \mathcal{U}_x(\phi) g'(x) d\scl(x) 
  \, = \,  \sum_{k\ge 1} k \phi_k  g_k  .
$$
This shows that  $\Ro^{S}$ is self-adjoint and positive-definite with 
\beq \label{Sigma_2}
{ \langle \Ro^{S}(\phi) , \phi \rangle} _{\mu_V}  \, = \,  4 \, \S(\phi) . 
\eeq
To prove compactness, we rely on the results of~\cite[Section~VI.5]{Reed_Simon_1}. Introduce  the operator
$\Ro^{S,N} \!: \phi\mapsto \psi_N$
where  $\psi_N$ solves the equation $\psi_N' S = \mathcal{U}^N(\phi)$ and $\mathcal{U}^N(\phi)  = \sum_{k=1}^N \phi_k U_{k-1}$. Plainly, the operators $\Ro^{S}_N$ have finite rank for all
$N\geq 1$ and for any $\phi\in \Hi$, we have
$$ \big \|  (\Ro^{S} -\Ro^{S,N} )(\phi)  \big \|_{\mu_V} \, \asymp \, 
  \big \|  (\mathcal{U} -\mathcal{U}_N )(\phi) \big \|_{L^2(\scl )} \, = \,
\sqrt{\sum_{k>N} |\phi_k|^2} \,  \ll \,  \frac{\|\phi\|}{\sqrt{N}}  \, .$$
The first step is similar to \eqref{R_2}, for the second step we used that by \eqref{U_3}, 
$(\mathcal{U} -\mathcal{U}_N )(\phi) = \sum_{k>N} \phi_k U_{k-1}$ and at last, we used the Cauchy-Schwarz inequality and \eqref{Fourier_3}. 
This proves that  $\Ro^{S} = \lim_{N\to\infty}\Ro^{S,N}$ in operator-norm  so that  $\Ro^{S}$ is compact. 
Therefore, by the spectral theorem, there exists an orthonormal basis $(\phi_n)_{n\ge 1}$ of $\Hi_{\mu_V}$ so that 
\beq  \label{eigeneq_5}
\Ro^{S}(\phi_n) \, = \,  \frac{2}{\sigma_n} \, \phi_n
\eeq
and $\sigma_n>0$ for all $n\geq 1$. 
Moreover, each eigenvalue has finite multiplicity, and we may order them so that the sequence $\sigma_n$ is non-decreasing  and  ${\sigma_n\to\infty}$ as $n\to\infty$.  In fact,  by the Min-Max theorem, we have 
$$
  \frac{2}{\sigma_n}  \, = \,  \max_{\begin{subarray}{c} \mathcal{S} \subset \Hi \\ \operatorname{dim}  \mathcal{S}=n \end{subarray}} \min_{\phi\in \mathcal{S}} \frac{ \langle \Ro^{S} \phi , \phi \rangle_{\mu_V} }{\| \phi\|_{\mu_V}^2}  \, , 
$$
 where the maximum ranges over all $n$-dimensional subspaces of the Hilbert space $\Hi$.
By \eqref{Sigma_2}, since the quantity $\S(\phi)$ does not depend on the equilibrium
measure $\mu_V$ and $\| \phi\|_{\mu_V} \asymp \| \phi\|$,  this shows that 
\beq \label{min_max}
  \frac{2}{\sigma_n}  \, \asymp \, \max_{\begin{subarray}{c} \mathcal{S} \subset \Hi \\ \operatorname{dim}  \mathcal{S}=n \end{subarray}} \min_{\phi\in \mathcal{S}} \frac{ \langle \Ro^{1} \phi , \phi \rangle}{\| \phi\|^2}  
    \,  . 
\eeq
In the Gaussian case ($S=1$), the eigenfunctions of $\Ro^1$ are the  Chebyshev polynomials
$\frac 1n T_n$ and the eigenvalues satisfy $\sigma_n =2n$, therefore, we deduce from \eqref{min_max} that, in the general case, we also have 
$\sigma_n \asymp n$. 

Finally, to prove that the eigenfunctions are bounded, we may expand $\phi_n$ in a Fourier-Chebyshev series: $\phi_n = \sum_{k=1}^\infty \phi_{n,k} T_k$  on $\J$.
Since $\| T_k \|_{\infty,\J}  =1$ for all $k\ge 1$, by  \eqref{Fourier_3}
and the Cauchy-Schwarz inequality, we obtain
$$
{\| \phi_n\|}_{\infty,\J}  \, \le \,   \sum_{k\ge 1} |\phi_{n,k}|  \, \ll \,  \| \phi_n \|   .
  $$  
Finally, since $\|\phi_n \| \asymp  \| \phi_n \|_{\mu_V} =1 $ for all $n\in\N$, 
 we conclude that ${\| \phi_n\|_{\infty,\bar\J} \ll 1}$. 
 \end{proof}

We split the proof of Theorem~\ref{th:basis} into two parts -- we now focus on existence and bounds on the functions $\phi_n$ and then in the next section discuss the Fourier expansion. \\

\noindent
{\it Proof of Theorem~\ref{th:basis} -- existence and bounds on eigenfunctions.}
By Theorem~\ref{thm:R_operator}, we have $\Ro^{S}(\phi_n) = \frac{2}{\sigma_n} \phi_n$  and if we differentiate this equation and use \eqref{R_1}, we obtain  that 
$S(x) \phi_n' (x)= \frac{\sigma_n}{2} \, \mathcal{U}_x(\phi_n)$ for all $x\in\J$. 
Then, by Lemma~\ref{lem:Tricomi}, this implies that for all $n\geq 1$, 
\beq \label{eigeneq_4}
 \Xi^{\mu_V}(\phi_n)  \, = \, \sigma_n \phi_n .
\eeq
Hence, $\phi_n\in\Hi$ and  $\sigma_n\asymp n$ solve the equation \eqref{eigeneq_2}.
By Proposition~\ref{pr:regularity0}, this means that for any $n\geq 1$, 
we can extend the eigenfunction (keeping the notation $\phi_n$)  in such a way that $\phi_n\in \Co^{\kappa}(\R)$ with compact support and 
it satisfies the equation \eqref{eigeneq_1} on the interval $\J_{\epsilon_n}$ where $\epsilon_n= \sigma_n^{-\eta} \wedge \delta$ and that for all $k<\kappa$, 
$$
{\| \phi_n^{(k)}\|}_{\infty,\R} \, \ll \,  \sigma_n^{\eta k } .
$$
Finally, combining formulae \eqref{Sigma_2} and \eqref{eigeneq_5}, since $\|\phi_n\|_{\mu_V}=1$, we obtain the identity $2\S(\phi_n)= \frac {1}{\sigma_n}$ for all $n\geq 1$.
The part of the theorem concerning existence and bounds on the functions is now proven.
\qed

  \subsection{Fourier expansion} \label{sect:expansion}
The proof of Theorem \ref{th:basis} will be complete once we prove the following result:
    
\begin{proposition}\label{thm:Fourier_coeff}
If $f\in \Co^{\kappa+4}(\bar\J)$, we can expand $f =\int_\J f(x)\varrho(dx)+  \sum_{n=1}^\infty \widehat{f}_n \phi_n$ where the Fourier coefficients $ \widehat{f}_n = \langle f , \phi_n \rangle_{\mu_V}$ satisfy
$$
|\widehat{f}_n|  \, \ll \,  \sigma_n^{- \frac {\kappa+3}{2}}
$$
for all $n\geq 1$. 
\end{proposition}

To do this, we prove an integration by parts result.

\begin{lemma}[Integration by parts] \label{thm:ibp}
The operator $\Xi^{\mu_V}$ is symmetric in the sense that for any function $g ,f \in \Co^2(\R)$, we have  
\beq\label{ibp}
     {\langle \Xi^{\mu_V}(f) , g \rangle}_{\mu_V}  
       \, = \,   {\langle  f,\Xi^{\mu_V}(g) \rangle}_{\mu_V}   . 
 \eeq
\end{lemma}
  
\begin{proof} We let $\tilde{g}(x) = g'(x) \frac{d\mu_V}{dx}$ for all $x\in\R$ and similarly for $\tilde{f}$. By assumption $\tilde{g}$ is continuous with support in $\bar\J$, it is differentiable  and  $\tilde{g}' \in L^p(\R)$ for any $p<2$. 
  Moreover,  since $\Xi^{\mu_V}(f) =\mathcal{H}(\tilde{f}) $, using the differentiation properties of the Hilbert transform, we have $ \frac{d}{dx}\mathcal{H}(\tilde{f})=\mathcal{H}(\tilde{f}')$. This proves that under our assumptions,  the function $\mathcal{H}(\tilde{f})$ is continuous on $\R$. Then, an integration by parts shows that 
$$
  {\langle \Xi^{\mu_V}(f) , g \rangle} _{\mu_V}  
   \, = \,  \int_\J \frac{d\mathcal{H}_x(\tilde{f})}{dx} \, \tilde{g}(x) dx 
   \,  = \,  - \int_\J  \mathcal{H}_x(\tilde{f}) \, \tilde{g}'(x) dx   .
$$
Note that there is no boundary term because we have just seen that the function $\tilde{g}\mathcal{H}(\tilde{f})$ is continuous with support in $\bar\J$. Using the anti-selfadjointness and differentiation properties of the Hilbert transform, this implies that
  $$
  {\langle \Xi^{\mu_V}(f) , g \rangle}_{\mu_V}  \, = \,   \int_\J  \tilde{f}(x) \mathcal{H}_x(\tilde{g}') dx 
  \, =  \,   \int_\J \tilde{f}(x) \frac{d\mathcal{H}_x(\tilde{g})}{dx} \, dx  
  $$  
 Finally, using that $\mathcal{H}(\tilde{g}) = \Xi^{\mu_V}(g)$ and $ \tilde{f}(x)dx = f'(x) \mu_V(dx)$, 
 we obtain \eqref{ibp} establishing Lemma~\ref {thm:ibp}.
 \end{proof}

We now turn to the missing ingredient in the proof of Theorem \ref{th:basis}

\begin{proof}[Proof of Proposition \ref{thm:Fourier_coeff}]
First of all, we claim  that if $\phi\in \Co^{k+2}(\bar\J)$ for some $k\le \kappa+2$, then the function 
$\Xi^{\mu_V}(\phi)\in \Co^{k}(\bar\J)$. 
Indeed, by \eqref{Xi} and the variational condition \eqref{V_condition}, we see that for all $x\in\J$,
  $$
  \Xi^{\mu_V}_x(\phi)  \, = \,  -  \int_\J \frac{\phi'(t)-\phi'(x)}{t-x} \, \mu_V(dt)- V'(x) \phi'(x) . 
  $$
Using Lemma~\ref{lem:derivative}, if $\phi' \in \Co^{k+1}(\bar\J)$ and the potential $V \in \Co^{\kappa+3}(\R)$, this implies that the function $  \Xi^{\mu_V}_x(\phi) \in   \Co^{k}(\bar\J)$. 
Let $f^{(0)}=f\in\Co^{\kappa+4}(\bar\J)$ and define $f^{(k+1)}=  \Xi^{\mu_V}(f^{(k)})$ for all $0\le k \le K$ (here $f^{(k)}$, is not to be confused with the $k$th derivative of $f$). 
The previous observation shows that
${f^{(1)}\in \Co^{\kappa+2}(\bar\J), \dots,  f^{(K)} \in\Co^{\kappa - 2(K-2)}(\bar\J) }$. 
Thus, choosing $K= \frac{\kappa+3}{2}$, we obtain that $f^{(K)}\in \Co^{1}(\bar\J)$. 
By definition, we have
\beqs \begin {split}
 \widehat{f}_n  \, = \,   { \langle f, \phi_n \rangle}_{\mu_V}  
 & \, = \,  \sigma_n^{-1}  { \langle  f,\Xi^{\mu_V}(\phi_n) \rangle}_{\mu_V}  \\
 & \, = \,  \sigma_n^{-1}  { \langle  f^{(1)} , \phi_n \rangle}_{\mu_V}  \\
 &  \, = \,  \sigma_n^{-K} {\langle  f^{(K)} , \phi_n \rangle}_{\mu_V} .  
\end {split} \eeqs
At first, we used the eigenequation \eqref{eigeneq_4}. Then, we used Lemma~\ref{thm:ibp} observing that the functions $f^{(k)} \in \Co^{2}(\bar\J)$ for all $0\le k<K$. The last step follows by induction. 
Hence,  since $\| \phi_n\|_{\mu_V} = 1$ for all $n\geq 1$, we obtain 
$$
| \widehat{f}_n |  \, \le \,  \sigma_n^{-K} {\| f^{(K)}\|}_{\mu_V}  \, \ll \,  \sigma_n^{-K}
$$
since the function  $ f^{(K)\prime}$ is uniformly bounded on $\bar\J$. 
Proposition~\ref{thm:Fourier_coeff} is established.
\end{proof}

\section{The CLT for the functions \texorpdfstring{$\phi_n$}{phi} -- Proof of Proposition~\ref{prop:clt}}\label{sec:phiclt}

In this section,  we establish 
a multi-dimensional CLT for the linear statistics of the test  functions $\phi_n$ constructed in the previous section, for a general one-cut regular potential $V$ and $\beta>0$. This will closely parallel the argument for the GUE from
Section~\ref{sec:GUEproof}, but we do encounter some technical difficulties. 
We point out here that we will find it convenient to assume that the eigenvalues are ordered, 
$\lambda_1\leq \cdots \leq \lambda_N$, allowing us to use the rigidity result of Theorem~\ref{thm:rigidity} to control the behavior of the random measure $\nu_N$ (see~Section~\ref{sect:rigidity}).

\subsection{Step 1 -- approximate eigenvector property}\label{sec:approxev} In this section we
establish that linear statistics of $\phi_n$ are approximate eigenfunctions of the generator  $\mathrm{L}$ as was the case for the Chebyshev polynomials when the potential $V$ is quadratic.

\begin{proposition}\label{le:approxev}
Let $(\phi_n)_{n=1}^\infty$ and $(\sigma_n)_{n=1}^\infty$ be as in Theorem~\ref{th:basis}.
If $\m$ is given by \eqref{mean} and $\mathrm{L}$  by \eqref{eq:betagene}, then for any $n\geq 1$
\beq  \begin {split} \label{eq:Vevprop}
& \mathrm{L}\bigg ( \sum_{j=1}^N \phi_n(\lambda_j)\bigg) \\
 &  \, = \, -\beta \sigma_n N\bigg (\sum_{j=1}^N \phi_n(\lambda_j)
    - N\int_\J \phi_nd\mu_V-  \Big (\frac{1}{\beta}-\frac{1}{2}\Big ) \m(\phi_n)\bigg)+\zeta_n(\lambda) \\
\end {split} \eeq
where $\zeta_n:\R^N\to \R$ satisfies the following two conditions: for any 
${\eta>  \frac {4(\kappa+1)}{2\kappa-1}}$,
 $|\zeta_n(\lambda)| \ll  N^2 n^{2\eta}  $  for all $\lambda\in\R^N$ and, 
if $\lambda_1,\ldots,\lambda_N\in[-1- \epsilon_n,1+ \epsilon_n]$ $($where $\epsilon_n$ is as in Theorem~\ref{th:basis}$)$, then 
\beq \begin {split} \label{eq:zetakV}
\zeta_n(\lambda) 
& \, = \, \Big (1-\frac{\beta}{2}\Big ) \int_\R \phi_n''(x)\nu_N(dx) \\
& \quad \, \, +\frac{\beta}{2} \int_{\R\times\R} \frac{\phi_n'(x)-\phi_n'(t)}{x-t} \, \nu_N(dx)\nu_N(dt). \\
\end {split} \eeq
\end{proposition}

To prove Proposition~\ref{le:approxev}, we will need the following lemma.

\begin{lemma} \label{le:phimuv}
Let $(\phi_n)_{n=1}^\infty$ and $(\sigma_n)_{n=1}^\infty$ be as in Theorem~\ref{th:basis}. 
We have for all $n\geq 1$, 
$$
\int_{\J\times \J} \frac{\phi_n'(t)-\phi_n'(x)}{t-x} \, \mu_V(dt)\mu_V(dx)
  \, = \, -2\sigma_n\int_\J \phi_n(x)\mu_V(dx) . 
$$
\end{lemma}

\begin{proof} First of all, let us observe that using the variational condition \eqref{V_condition}, for any function $f\in\Co^1_c(\R)$,   
$$
\int_\J f(x)V'(x)\mu_V(dx)
  \, = \, \int_\J f(x) \mathcal{H}_x(\mu_V) \mu_V(dx) 
  \, = \, - \int_\J \mathcal{H}_x(f\mu_V)  \mu_V(dx) 
$$
where we used the anti-self adjointness of the Hilbert transform, \eqref{eq:antiself},
and the fact that $\mu_V$ is absolutely continuous with a bounded density function. In fact, one has
$$
 \mathcal{H}_x(f\mu_V) \, = \, - \int_\J \frac{f(t) -f(x)}{t-x} \, \mu_V(dt) + f(x)V'(x)
$$
and, by symmetry, this implies that 
\beq \label{eq:fmuv}
\int_\J f(x)V'(x)\mu_V(dx) \, = \,\frac 12 \int_{\J\times\J} \frac{f(t) -f(x)}{t-x} \, \mu_V(dt)\, \mu_V(dx) . 
\eeq
On the other-hand, integrating the equation \eqref{eq:eveq} with respect to $\mu_V$, we see that 
\beqs \begin {split}
 \int_\J V'(x)\phi_n'(x)  \, \mu_V(dx) - \int_{\J\times\J} 
  \frac{\phi_n'(x)-\phi_n'(y)}{x-y} \, & \mu_V(dx) \, \mu_V(dy) \\
  & \, = \, \sigma_n \int_\J \phi_n(u)\mu_V(du), \\
\end {split} \eeqs
which, together with \eqref{eq:fmuv} completes the proof of this lemma. 
\end{proof}

\begin{proof}[Proof of Proposition~\ref{le:approxev}]
Let us first prove the uniform bound for the error term $\zeta_n$ on $\R^N$.  
Since the density $S = \frac{d\mu_V}{d\scl}$ is $\Co^1$ and positive
 on $ \bar \J = [-1,1]$, we have (in the notation of \eqref{eq:inftynorm})
$$
\big| \m(\phi_n) \big|  \, \ll \,  {\| \phi_n\|}_{\infty,\J} + {\| \phi_n'\|}_{\infty,\J}  .
$$
One has
\beq \begin{split} \label{eq:genephi}
 &\mathrm{L}  \bigg ( \sum_{j=1}^N\phi_n(\lambda_j) \bigg)  \, = \, \sum_{j=1}^N \phi_n''(\lambda_j)-\beta N
     \sum_{j=1}^N V'(\lambda_j)\phi_n'(\lambda_j)+\beta \sum_{i\neq j}\frac{1}{\lambda_j-\lambda_i} \, \phi_n'(\lambda_j)\\
& \, = \, \left(1-\frac{\beta}{2}\right)\sum_{j=1}^N \phi_n''(\lambda_j)-\beta N\sum_{j=1}^N V'(\lambda_j)\phi_n'(\lambda_j)+\frac{\beta}{2}\sum_{i,j=1}^N \frac{\phi_n'(\lambda_j)-\phi_n'(\lambda_i)}{\lambda_j-\lambda_i} \, .
\end{split} \eeq
Hence, we obtain
$$
\bigg | \mathrm{L} \bigg (\sum_{j=1}^N \phi_n(\lambda_j) \bigg)  \bigg| \,  \ll \,
N^2 \Big( { \| \phi_n''\|}_{\infty,\R}  +  {\| V'\|}_{\infty,\J_{2\delta}} {\| \phi_n'\|}_{\infty,\R}  \Big)$$
where the last bound follows from the fact that, by construction, the functions $\phi_n$
have compact support in the interval $\J_{2\delta}$. 
Combining these two estimates, using the upper-bound  \eqref{eq:phinderb} and the fact
that $ \sigma_n\asymp  n$ from Theorem~\ref{th:basis}, we obtain
\beqs \begin {split}
\big| \zeta_n(\lambda) \big|  
  & \, = \, \bigg|
\mathrm{L} \bigg (\sum_{j=1}^N \phi_n(\lambda_j) \bigg) + \beta \sigma_n N\bigg (\sum_{j=1}^N \phi_n(\lambda_j)
    - N\int_\J \phi_n(\lambda)\mu_V(d\lambda)-\m(\phi_n)\bigg) \bigg| \\
& \, \ll \,  N^2\sigma_n^{2\eta},
\end{split} \eeqs
which yields $\| \zeta_n\|_{L^\infty(\R^N)} \ll N^2n^{2\eta}$. 

In order to prove \eqref{eq:zetakV}, observe that 
\beqs \begin{split}
\sum_{i,j=1}^N \frac{\phi_n'(\lambda_j)-\phi_n'(\lambda_i)}{\lambda_j-\lambda_i}
& \, = \,  2N   \iint_{\R\times \J} \frac{\phi_n'(x)-\phi_n'(t)}{x-t} \, \nu_N(dx)\mu_V(dt)  \\
&\quad \, \,    +\iint_{\R\times \R} \frac{\phi_n'(x)-\phi_n'(t)}{x-t} \, \nu_N(dx) \nu_N(dt)  \\
& \quad \, \, + N^2\iint_{\J\times \J} \frac{\phi_n'(x)-\phi_n'(t)}{x-t} \, \mu_V(dx) \mu_V(dt). 
\end{split} \eeqs
Thus, if $\lambda_1,\ldots,\lambda_N\in \J_{\epsilon_n}$, using the equation~\eqref{eq:eveq}, we obtain 
\begin{align*}
N \sum_{j=1}^N & V'(\lambda_j)  \phi_n'(\lambda_j) 
 - \frac{1}{2} \sum_{i,j=1}^N \frac{\phi_n'(\lambda_j)-\phi_n'(\lambda_i)}{\lambda_j-\lambda_i}  \\
 & \, = \,  N \sigma_n \sum_{j=1}^N \phi_n(\lambda_j)
-  \frac{1}{2} \iint_{\R\times \R} \frac{\phi_n'(x)-\phi_n'(t)}{x-t} \, \nu_N(dx) \nu_N(dt)\\
&\quad + \frac{N^2}{2}\iint_{\J\times \J} \frac{\phi_n'(x)-\phi_n'(t)}{x-t} \, \mu_V(dx) \mu_V(dt) \\
 & \, = \,  N \sigma_n \int_\R \phi_n(x) \nu_N(dx) -  \frac{1}{2} \iint_{\R\times \R} \frac{\phi_n'(x)-\phi_n'(t)}{x-t} 
       \, \nu_N(dx) \nu_N(dt) ,
\end{align*}
where we used Lemma~\ref{le:phimuv} for the last step. According to \eqref{eq:genephi}
and \eqref{eq:zetakV}, it follows that
\begin{equation*} \begin {split}
\mathrm{L} & \bigg (  \sum_{j=1}^N   \phi_n(\lambda_j)\bigg) \\
  & \, = \, -\beta \sigma_n N\bigg (\int_\R \phi_n(x) \nu_N(dx) 
    + \Big(\frac{1}{2}-\frac{1}{\beta}\Big) \frac{1}{\sigma_n} \int_\J \phi_n''(x) \mu_V(dx)  \bigg)+\zeta_n(\lambda) . \\
\end {split} \end{equation*}

Therefore, to complete the proof, it just remains to check that according to \eqref{mean}, we have
for all $n \geq 1$
\beq \label{mean_2}
\m(\phi_n) \, = \,  \frac{1}{\sigma_n} \int_\J \phi_n''(x) \mu_V(dx) .
\eeq
Since the density $S = \frac{d\mu_V}{d\scl}$ is $\Co^1$ and positive
on $ \bar \J = [-1,1]$, an integration by parts shows that 
$$
 \int_\J \phi_n''(x) \mu_V(dx)    \, = \, - \frac{2}{\pi}\int_{\J} \phi_n'(x) S(x)\left(\frac{S'(x)}{S(x)}\sqrt{1-x^2}-\frac{x}{\sqrt{1-x^2}}\right)dx . 
$$
Since for any $n \geq 1$ 
$\phi_n$ is a solution of the equation \eqref{eigeneq_3}, this implies that   
\begin{align*}
 \frac{1}{\sigma_n} \int_\J \phi_n''(x) \mu_V(dx)    \, = \, - \frac{1}{\pi}\int_{\J} \U_x(\phi_n)\left(\frac{S'(x)}{S(x)}\sqrt{1-x^2}-\frac{x}{\sqrt{1-x^2}}\right)dx . 
\end{align*}
By Lemma~\ref{lemma:mean}, we may rewrite this formula as 
\begin{align*}
 \frac{1}{\sigma_n} \int_\J \phi_n''(x) \mu_V(dx)    \, = \,  \frac{1}{2} \left(
\phi_n(1)+\phi_n(-1) -
  \int_\J \U_x(\phi_n)\frac{S'(x)}{S(x)} \, \scl(dx) \right) ,
\end{align*}
which, according to \eqref{U_0}, finishes the proof of \eqref{mean_2}.
The proof of Proposition~\ref {le:approxev} is thus complete. \end{proof}

\subsection{Step 2 -- a priori bound on linear statistics}  \label{sect:rigidity}

In this section we use the rigidity estimates from Theorem~\ref{thm:rigidity} to establish a result corresponding to \eqref{eq.convexity}. The following lemma is a standard consequence of rigidity, but we record it for completeness.

\begin{lemma}\label{le:linstatbound}
Let $\epsilon>0$ be fixed and $\mathcal{B}_\epsilon$  be the event defined in Theorem~\ref{thm:rigidity}. We have for any configuration $\lambda\in\mathcal{B}_\epsilon$ and  for any Lipschitz continuous function $f:\R\to \R$,
$$
\bigg |\sum_{j=1}^N f(\lambda_j) - N\int_\J f(x) \mu_V(dx)\bigg| \, \ll \, {\|f\|}_{\Co^{0,1}(\R)} N^\epsilon . 
$$
In particular, if $f$ is a bounded Lipschitz function, then for any $p>0$, 
$$
\E\bigg |\sum_{j=1}^N f(\lambda_j) - N\int_\J  f(x) \mu_V(dx)\bigg |^p 
  \, \ll_p \,  \|f\|^p_{\Co^{0,1}(\R)} N^{p\epsilon}+{\|f\|}_{\infty,\R}^p \, e^{-N^{c_\epsilon}}.
$$
\end{lemma}

\begin{proof}
Recalling the definition of the classical locations $\rho_j$ from Theorem~\ref{thm:rigidity}, we may write 
$$
\sum_{j=1}^N f(\lambda_j) - N\int_\J  f(x) \mu_V(dx)
  \, = \, N  \sum_{j=1}^N \int_{\rho_{j-1}}^{\rho_j} (f(\lambda_j) - f(x))\mu_V(dx)
$$
and we see that if $\lambda = (\lambda_1, \ldots, \lambda_N) \in\mathcal{B}_\epsilon$, 
\beq \begin{split} \label{rigidity_1}
\bigg |\sum_{j=1}^N f(\lambda_j) & - N  \int_\J  f(x) \mu_V(dx)\bigg | \\
& \, \leq \, {\|f\|}_{\Co^{0,1}(\R)} \sum_{j=1}^N 
   \max \big (|\lambda_j - \rho_j|,|\lambda_j -\rho_{j-1}| \big ) \\
& \, \ll \,  {\|f\|}_{\Co^{0,1}(\R)}\sum_{j=1}^N \Big\{ \widehat{j}^{- \frac13}N^{- \frac23+\epsilon}
  +(\rho_j - \rho_{j-1})\Big\}.
\end{split} \eeq
Obviously, $\sum_{j=1}^N \widehat{j}^{-\frac 13} \ll N^{\frac 23}$
and $\sum_{j=1}^N (\rho_j - \rho_{j-1})$ ${= \rho_N - \rho_0 = 2}$ 
due to the normalization of the support of the equilibrium measure, so that we have proved the first estimate. The second estimate follows directly from the facts  that $f$ is bounded and, by  Theorem~\ref{thm:rigidity},
that the probability of the complement of $\mathcal{B}_\epsilon$ is exponentially small in some power of $N$.
Lemma~\ref {le:linstatbound} is established.
\end{proof}

\begin{lemma}\label{le:ratio}
Let $\epsilon>0$ be fixed and $f\in\Co^4(\R)$. We have for any $\lambda\in\mathcal{B}_\epsilon$,
$$
\left|\int_{\R\times \R} \frac{f'(x)-f'(y)}{x-y} \, \nu_N(dx)\nu_N(dy)\right|
  \, \ll \,   {\|f^{(4)}\|}_{\infty,\R} \, N^{2\epsilon} .
$$
\end{lemma}

\begin{proof}
We begin by pointing out the following fact that is easily checked by repeatedly performing the integrals: for any $f\in \Co^4(\R)$ and $s,t,u,v\in \R$,
\beq \begin{split}\label{eq:interp}
 & \hskip -4mm \frac{f'(s)-f'(t)}{s-t}+\frac{f'(u)-f'(v)}{u-v}-\frac{f'(s)-f'(v)}{s-v}-\frac{f'(u)-f'(t)}{u-t}\\
 & \, = \, (s-u)(t-v)\int_{[0,1]^3}f^{(4)} \big (ac(s-u)+(1-a)b(t-v)+au+(1-a)v \big )\\
 &\qquad\qquad\qquad \qquad  \qquad  \times a(1-a)dadbdc ,
\end{split} \eeq
where the LHS is interpreted as $f''$ when some of the parameters coincide.  
Noting that 
\beqs \begin{split}
& \int_{\R\times \R}\frac{f'(x)-f'(y)}{x-y} \, \nu_N(dx)\nu_N(dy)\\
& \, = \, \sum_{i,j=1}^N N^2\int_{\rho_{i-1}}^{\rho_i}\int_{\rho_{j-1}}^{\rho_j} \Bigg[\frac{f'(\lambda_i)-f'(\lambda_j)}{\lambda_i-\lambda_j}+\frac{f'(x)-f'(y)}{x-y}-\frac{f'(\lambda_i)-f'(x)}{\lambda_i-x}\\
&\hskip 42mm -\frac{f'(\lambda_j)-f'(y)}{\lambda_j-y}\Bigg]\mu_V(dx)\mu_V(dy),
\end{split} \eeqs
formula \eqref{eq:interp} implies that 
\beqs \begin{split}
\bigg | \int_{\R\times \R} &  \frac{f'(x)-f'(y)}{x-y}  \, \nu_N(dx)  \nu_N(dy)\bigg |   \\ 
& \, \leq \, {\|f^{(4)}\|}_{\infty,\R} \, N^2\sum_{i,j=1}^N 
  \int_{\rho_{i-1}}^{\rho_i}|\lambda_i-x|\mu_V(dx)
  \int_{\rho_{j-1}}^{\rho_j}|\lambda_j-y|\mu_V(dy)\\
& \, \leq  \, {\|f^{(4)}\|}_{\infty,\R}\bigg(\sum_{i=1}^N
    \max \big (|\lambda_i-\rho_i|,|\lambda_i-\rho_{i-1}| \big )\bigg)^2.
\end{split} \eeqs
Like in the proof of Lemma~\ref{le:linstatbound}, conditionally on $\mathcal{B}_\epsilon$, the previous sum is asymptotically of order at most $N^\epsilon$ which completes the proof. 
\end{proof}

\subsection{Step 3 -- controlling the error terms \texorpdfstring{$A$}{A} and \texorpdfstring{$B$}{B}} 
As for the GUE, we now move on to estimate the ``$A$-term"  and ``$B$-term" by using
the results of Section~\ref{sect:rigidity} and Theorem~\ref{th:basis} in order to apply
Proposition~\ref{prop.main} and deduce the multi-dimensional CLT in the next subsection. 

\begin{lemma}\label{le:Abound}
 Let $(\phi_n)_{n=1}^\infty$ and $(\sigma_n)_{n=1}^\infty$ be given as in Theorem~\ref{th:basis}.
Let  ${F:\R^N\to \R^d}$ be defined by
\beq  \begin {split} \label{eq:FVdef}
& F(  \lambda)  \, = \,  F_n(\lambda)  \\
& \quad \hskip 3mm   \, = \, \sqrt{\beta \sigma_n} \bigg (\sum_{j=1}^N \phi_n(\lambda_j)
    - N\int_\J \phi_n(x) \mu_V(dx)-   \Big(\frac{1}{\beta}-\frac{1}{2}\Big) \m(\phi_n)\bigg) , \\
\end {split} \eeq
and let the  matrix $K=K^{N,d}_{V,\beta}= \beta N \,\mathrm{diag}(\sigma_1, \dots, \sigma_d) $. Then, we have for any {$\eta >  \frac {4(\kappa+1)}{2\kappa-1}$}, any $\epsilon>0$, and for any 
$d\le N^{\frac {1}{2\eta}}$,  
$$
A^2 \, = \, \E\left\|F+ K^{-1}\mathrm{L}(F)\right\|^2
  \, \ll \,   { d^{8\eta}N^{-2+4\epsilon}} . 
$$
\end{lemma}

\begin{proof}
According to \eqref{eq:FVdef}, we may rewrite \eqref{eq:Vevprop} as
$$
\mathrm{L} F  =- KF + \sqrt{\beta \sigma_n} \, \zeta
$$
where we extend the definition of $\zeta:\R^N\to \R^d$.  Then, we have
$$
A^2   \, = \,  \frac{1}{\beta N^2} \E \bigg[ \sum_{n=1}^d \frac{\zeta_n(\lambda)^2}{\sigma_n} \bigg] . 
$$
Since $\|\zeta_n\|_{L^\infty(\R^N)} \ll N^2n^{2\eta}$ for all  $n \geq 1$  by  Theorem~\ref{thm:rigidity}, this implies that 
$$
\bigg|  A^2  -  \frac{1}{\beta N^2} \sum_{n=1}^d  \frac{\E \, [  \zeta_n(\lambda)^2 \ind_{\mathcal{B}_\epsilon} ]}{\sigma_n}
\bigg|   \, \ll \,  N^2 d^{4\eta} e^{-N^c} . 
$$
Recall that $\epsilon_d \asymp \sigma_d^{-\eta}$ when $d$ is large, so that for any $\epsilon>0$, we have $\mathcal{B}_\epsilon \subset [-1- \epsilon_d,1+ \epsilon_d]^N$ when the parameter $N$ is sufficiently large. Then, by \eqref{eq:zetakV}, 
\begin{align*}
\E \big[  \zeta_n(\lambda)^2 \ind_{\mathcal{B}_\epsilon} \big] 
 \, &\ll \,  \E \bigg[ \ind_{\mathcal{B}_\epsilon}  \bigg( \int_\R \phi_n''(x)\nu_N(dx) \bigg)^2 \bigg]\\
 &\quad \, \, + \E \bigg[ \ind_{\mathcal{B}_\epsilon}  \bigg( \int_{\R\times \R} \frac{\phi_n'(x)-\phi_n'(t)}{x-t} \, \nu_N(dx)\nu_N(dt) \bigg)^2 \bigg] . 
\end{align*}
By Lemma~\ref{le:linstatbound}, the first term is bounded by $\| \phi_n^{(3)}\|_{\infty,\R}^2N^{2\epsilon}$, while by Lemma~\ref{le:ratio}, the second term is bounded by $\| \phi_n^{(4)}\|_{\infty,\R}^2 N^{4\epsilon}$. Now, using the estimates of Theorem~\ref{th:basis}, we conclude that 
$$
A^2   \, \ll \,  N^{-2+4\epsilon} \sum_{n=1}^d \sigma_n^{8\eta-1}
  \, \asymp  \,  N^{-2+4\epsilon} \sigma_d^{8\eta} 
$$
which completes the proof.
\end{proof}

 \begin{lemma}\label{le:Bbound}
 Let $B$ be given by formula \eqref{eq.b} where $F:\R^N\to \R^d$ and the matrix $K$ are
 as in Lemma~\ref{le:Abound}.
Then, we have for any {$\eta >  \frac {4(\kappa+1)}{2\kappa-1}$} and any $\epsilon>0$, 
\beq\label{eq:B_estimate}
B^2 \, = \,
 \E\left\|\mathrm{Id}-K^{-1}\Gamma(F)\right\|^2
  \, \ll \, d^{6\eta+2} N^{-2+2\epsilon}  , 
\eeq
as long as the RHS converges to $0$ as $N\to\infty$. 
$($In particular, this is the case when $d\le N^{\frac {1}{4\eta}})$. 
\end{lemma}

\begin{proof} By definition, we have
\beqs \begin {split}
\Gamma(F_i, F_j) 
  & \, = \, \beta  \sqrt{\sigma_i\sigma_j} \sum_{\ell =1}^N \phi_i'(\lambda_\ell )\phi_j'(\lambda_\ell ) \\
  & \, = \,  \beta   \sqrt{\sigma_i\sigma_j} \left(  \int_\R  \phi_i'(x)\phi_j'(x) \nu_N(dx) +  N\delta_{i, j}\right)
\end {split} \eeqs
since, by Theorem~\ref{th:basis}, $(\phi_n)_{n=1}^\infty$ is an orthonormal basis of the Sobolev space $\Hi_{\mu_V}$. This implies that, for the Hilbert-Schmidt norm, 
\beqs \begin{split}
\left\|\mathrm{Id}-K^{-1}\Gamma(F)\right\|^2
& \, = \,   \frac{1}{N^2} \sum_{i,j =1}^d \frac{\sigma_j}{\sigma_i} \left( \int_\R  \phi_i'(x)\phi_j'(x) \nu_N(dx)\right)^2 \\
& \, \le \, \frac{2}{N^2} \sum_{1 \le i \le j \le d} \frac{\sigma_j}{\sigma_i} \left(  \int_\R  \phi_i'(x)\phi_j'(x) \nu_N(dx)\right)^2 
\end{split} \eeqs
Then, by Lemma~\ref{le:linstatbound}, since the functions $x\mapsto  \phi_i'(x)\phi_j'(x) $ are uniformly bounded and Lipschitz continuous on $\R$ with norm at most
$$
 {\| \phi'_j\|}_{\infty,\R}  \, {\| \phi''_i\|}_{\infty,\R} + {\| \phi'_i\|}_{\infty,\R} \,
   {\| \phi''_j\|}_{\infty,\R}  \,  \ll \, \sigma_i^{\eta}\sigma_j^{2\eta} ,
$$
using the estimate \eqref{eq:phinderb} when $i\le j$, we obtain 
$$
B^2 \, \ll \, N^{-2(1-\epsilon)} \sum_{1 \le i \le j \le d} \sigma_i^{2\eta-1}\sigma_j^{4\eta+1} + d^2 e^{-N^c}
$$
which yields the estimate  \eqref{eq:B_estimate} since $\sigma_j \asymp j$ and the second term is asymptotically negligible.
\end{proof}

\subsection{Proof of Proposition~\ref{prop:clt}}
By Theorem~\ref{th:basis}, $\sigma_n = \frac{1}{2\S(\phi_n)}$, so that according to 
\eqref{eq:FVdef}, we have for all $n \geq 1$
$$
\mathcal{X}_{N}(\phi_n)  \, = \, F_n(\lambda) . 
$$
Then, by Proposition~\ref{prop.main} and using the bounds of Lemmas~\ref{le:Abound}
and~\ref{le:Bbound}, we obtain for any {$\eta >  \frac {4(\kappa+1)}{2\kappa-1}$}, 
$$
\mathrm{W}_2\left(\big(\mathcal{X}_{N}(\phi_n) \big)_{n=1}^d, \gamma_d\right) \, \le \,  A+B  
  \, \ll \,  d^{4\eta } N^{-1+2\epsilon} .  
$$
Here, we used that the main error term is given by $A$ when the parameter $d$ is large.   Optimizing over the parameter $\eta$ thus completes the proof. 
\qed

\section{The CLT for general test functions}\label{sec:clt}

In this section, we fix a test function $f\in \Co_c^{{\freg}}(\R)$ with $\kappa \ge 5$. { Without loss of generality, we may also assume that $\widehat{f}_0 = \int f(x) \varrho(dx) = 0$ (otherwise we might consider the function $\widetilde{f}= f- \widehat{f}_0$ instead and note that by formulae \eqref{mean} and \eqref{eq:linstat}, we have both $\m(\widetilde{f}) = \m(f)$ and $\mathcal{X}_N(\widetilde{f})=\mathcal{X}_N(f)$)} 
and we define for any $d \geq 1$
\beq \label{g_d}
g_d(x) = \sum_{n = 1}^d \widehat{f}_n \phi_n(x) , 
\eeq
where $\widehat{f}_n=\langle f, \phi_n \rangle_{\mu_V}$.
Note that if $\kappa \ge 5$, by Theorem~\ref{th:basis}, the sum \eqref{g_d} converges as $d\to\infty$ and this defines a function $g_\infty \in \Co^1_c(\R)$ which satisfies $g_\infty(x) =f(x)$ for all $x\in\J$.
The proof  of Theorem~\ref{th:main} is divided into two steps. First, we make use of the multi-dimensional CLT -- Proposition \ref{prop:clt} -- to prove a CLT for the function $g_d$  in the regime as $d=d(N)\to\infty$. Then, we establish that along a suitable sequence  $d(N)$  the random variable $\mathcal{X}_N(f)$ given by \eqref{eq:linstat}, is well approximated with respect the Kantorovich distance by  
\beq \label{eq:linstat_2}
X_N^d =  \sqrt{\frac{\beta}{2\S(f)}}   \left(\int_\R g_d(x) \nu_N(dx)
    - \Big (\frac{1}{2}-\frac{1}{\beta}\Big)\m(f) \right)  . 
\eeq

\subsection{Consequence of the multi-dimensional CLT}
In this section, we shall use Proposition~\ref{prop:clt} to establish that, if $N$ is sufficiently large (compared to the parameter $d$), then the law of the random variable $X_N^d$ is close to the Gaussian measure  $\gamma_1$.  In particular, in order to control the error term with respect to the Kantorovich metric, we need to express the correction term to the mean $\m(f)$ and the asymptotic variance $\S(f)$ of a linear statistics in terms of the Fourier coefficients of the test function $f$. 
This is the goal of the next lemma.
  
\begin{lemma} \label{thm:Fourier_sum}
If $f\in \Co^{\kappa+4}(\bar\J)$ and $\kappa \ge 5$, using the notation of Proposition~\ref{thm:Fourier_coeff}, we have
\beq \label{Sigma_3}
\S(f) = \sum_{n=1}^\infty \frac{\widehat{f}_n^2}{2\sigma_n} \, , 
\eeq
and 
\beq \label{mean_5}
\m(f) =  \sum_{n=1}^\infty  \widehat{f}_n\m(\phi_n) .  
\eeq
\end{lemma}

\begin{proof}
On the one-hand, by \eqref{Sigma_1} and  \eqref{R_1},  we have
 $$
\S(f) \, = \, \frac{1}{4} \int_\J f'(x)  \, \U_x(f)  \scl(dx) 
   \, = \,  \frac{1}{4} \,  \big\langle f, \Ro^S(f)  \big\rangle_{\mu_V}  . 
$$
On the other-hand,  by linearity and \eqref{eigeneq_5},  we have
 $$
 \Ro^S(f)  \, =  \, \sum_{n\in\N} \widehat{f}_n \Ro^S(\phi_n) 
 \, = \,  2 \sum_{n\in\N} \frac{\widehat{f}_n}{\sigma_n} \, \phi_n
 $$
where the sum converges uniformly.  Since $(\phi_n)_{n\geq 1}$ is an orthonormal basis 
of $\Hi_{\mu_V}$, this yields \eqref{Sigma_3}.    
Formula \eqref{mean_5} follows by linearity of the operator $\m$, \eqref{mean}. 
Moreover, note that by  \eqref{mean_2},
\beqs 
\m(\phi_n) 
 \, = \,  \frac{1}{\sigma_n} \int_\J \phi_n''(x) \mu_V(dx)  
 \, = \,   -\frac{2}{\pi\sigma_n} \int_{\J} \phi_n'(x) \frac{d}{dx}\left( S(x)\sqrt{1-x^2}\right) dx , 
\eeqs
so that  for any {$\eta>\frac {4(\kappa+1)}{2\kappa-1}$}, 
\beq \label{mean_6}
\big| \m(\phi_n) \big|  \,\ll \,  \sigma_n^{-1} {\|\phi_n' \|}_{\infty,\J}   
 \, \ll \,  \sigma_n^{\eta-1} .
\eeq
Using the estimate of Theorem~\ref{th:basis}, this proves that the series \eqref{mean_5} converges absolutely as long as $\eta<3$ and $\kappa\ge 5$.
\end{proof}

\begin{lemma}\label{thm:cvg}
Let $X_N^d$ be as in \eqref{eq:linstat_2} and assume
that $d^{\frac{\kappa+1}{2}} \le N$. Then, for any $\epsilon>0$,
$$
\mathrm{W}_2^2\left( X_N^d, \gamma_1\right) \, \ll \, \big( d^{8\eta_*} N^{-2}  
+ d^{2\eta_*-\kappa-3} \big) N^\epsilon ,  
$$
where $\eta_* :=  \frac {4(\kappa+1)}{2\kappa-1}$. 
\end{lemma}

\begin{proof} Using the notation \eqref{eq:FVdef} and by linearity of $\m$, we may rewrite
$$
X_N^d  = \frac{1}{\sqrt{2\S(f)}}  \bigg ( \sum_{n=1}^d  \frac{ \widehat{f}_nF_n}{\sqrt{\sigma_n}} 
     +    \Big(\frac{\sqrt{\beta}}{2}-\frac{1}{\sqrt{\beta}}\Big )\m(g_d-f)  \bigg)  . 
$$
Let $Z\sim\gamma_d$. It follows from the previous formula and the triangle inequality that 
\beq \label{bound_5}
\begin{split}
\mathrm{W}_2^2\left( X_N^d, \gamma_1\right) 
 \, &\le \,  \frac{2}{\S(f)}  \bigg\{
\mathrm{W}_2^2\bigg (  \sum_{n=1}^d  \frac{ \widehat{f}_nF_n}{\sqrt{2\sigma_n}} ,  \sum_{n=1}^d  \frac{ \widehat{f}_nZ_n}{\sqrt{2\sigma_n}} \bigg) \\
&\quad + \mathrm{W}_2^2\bigg( \sum_{n=1}^d  \frac{ \widehat{f}_nZ_n}{\sqrt{2\sigma_n}} 
   + \Big(\frac{\sqrt{\beta}}{2}-\frac{1}{\sqrt{\beta}}\Big )\m(g_d-f)  , \sqrt{ \S(f)} \gamma_1\bigg )   \bigg\}
\end{split}
\eeq
Our task is to prove that both term on the RHS of \eqref{bound_5} converge to 0 as $N\to\infty$. 
 The second term corresponds to the distance between two Gaussian measures on $\R$ and equals 
\beqs \begin{split}
\mathrm{W}_2^2\bigg ( \sum_{n=1}^d  \frac{ \widehat{f}_nZ_n}{\sqrt{2\sigma_n}} &+\Big(\frac{\sqrt{\beta}}{2}-\frac{1}{\sqrt{\beta}}\Big )  \m(g_d-f)  ,  \sqrt{\S(f)} \gamma_1 \bigg )  \\
& \, = \,   \Big(\frac{\sqrt{\beta}}{2}-\frac{1}{\sqrt{\beta}}\Big )^2\m(g_d-f)^2 + 
  \Bigg( \sqrt{ \sum_{n=1}^d  \frac{ \widehat{f}_n^2}{2\sigma_n} } - \sqrt{\S(f)}  \Bigg)^2. 
\end{split} \eeqs
By Lemma~\ref{thm:Fourier_sum},  this implies that 
\begin{align*}
\mathrm{W}_2^2 \bigg ( \sum_{n=1}^d  \frac{ \widehat{f}_nZ_n}{\sqrt{2\sigma_n}} + \Big(\frac{\sqrt{\beta}}{2}- & \frac{1}{\sqrt{\beta}}\Big )\m(g_d-f)  , \sqrt{\S(f)} \gamma_1 \bigg ) \\
&  \, \ll \,   \bigg( \sum_{n>d}  \widehat{f}_n \m(\phi_n) \bigg)^2  +  \sum_{n>d}  \frac{ \widehat{f}_n^2}{2\sigma_n} \, . 
\end{align*}
Using the estimate \eqref{mean_6}, we see that both sums converge if  $\kappa\ge 5$ and we obtain for any $\epsilon>0$, 
\beq \begin {split} \label{bound_3}
\mathrm{W}_2^2\bigg( \sum_{n=1}^d  \frac{ \widehat{f}_nZ_n}{\sqrt{2\sigma_n}} + 
  \Big(\frac{\sqrt{\beta}}{2}-\frac{1}{\sqrt{\beta}}\Big )  \m(g_d-f)  ,  & \sqrt{\S(f)} \gamma_1 \bigg) \\
 & \,\ll \,  d^{2\eta-\kappa -3}  \\
 & \,  \ll \,  d^{2\eta_*-\kappa -3}  N^\epsilon .
\end{split} \eeq

On the other-hand, by definition of the Kantorovich distance and the Cauchy-Schwarz inequality, the first term on the RHS of in \eqref{bound_5}  satisfies the bound
\begin{equation*}
\mathrm{W}_2^2\bigg (  \sum_{n=1}^d  \frac{ \widehat{f}_nF_n}{\sqrt{2\sigma_n}} ,  \sum_{n=1}^d  \frac{ \widehat{f}_nZ_n}{\sqrt{2\sigma_n}} \bigg )
 \, \le \, \bigg( \sum_{n=1}^d  \frac{ \widehat{f}_n^2}{2\sigma_n}  \bigg)\
\mathrm{W}_2^2\big( (F_n)_{n=1}^d, \gamma_d\big).
\end{equation*}
By definition, $F_n = \mathcal{X}_{N}(\phi_n)$, so that by Proposition~\ref{prop:clt}, we obtain
\beq \label{bound_4}
\mathrm{W}_2^2\bigg (  \sum_{n=1}^d  \frac{ \widehat{f}_nF_n}{\sqrt{2\sigma_n}} ,  \sum_{n=1}^d  \frac{ \widehat{f}_nZ_n}{\sqrt{2\sigma_n}} \bigg )
\, \ll \, \S(f) d^{8\eta_*} N^{-2+2\epsilon}  .
\eeq
Combining the estimates \eqref{bound_5}, \eqref{bound_3}  and \eqref{bound_4} completes the proof of the lemma.
\end{proof}

\subsection{Truncation}
The goal of this section is to prove that the linear statistic $\mathcal{X}_{N}(f)$ is asymptotically close to the random variable $X_N^d$ with respect to the Kantorovich distance, as $d\to\infty$ and $N\to\infty$ simultaneously in a suitable way.   
At first, we compare the laws of  the linear statistics $\mathcal{X}_{N}(f)$ and $X_N^\infty$ by using rigidity of the random configurations under the Gibbs measure $\PP_{V,\beta}^N$. Then, we compare the laws of $X_N^d$  and $X_N^\infty$ by using the decay of the Fourier coefficients of the test function $f$; cf.~Theorem~\ref{th:basis}.

\begin{lemma} \label{thm:edge_truncation}
According to the notation \eqref{eq:linstat} and \eqref{eq:linstat_2}, we have for any $\epsilon>0$,  
$$
\mathrm{W}_2^2\big( X_N^\infty , \mathcal{X}_N(f)  \big) \, \ll_f \,  N^{- \frac{4}{3}+\epsilon} .
$$
\end{lemma}

\begin{proof}
By definition,
$$
\mathcal{X}_N(f)- X_N^\infty   \, = \, \sqrt{\frac{\beta}{2\S(f)}}  \int_\R \big(f-g_\infty) d\nu_N 
$$
so that
$$
\mathrm{W}_2^2\big( X_N^\infty , \mathcal{X}_N(f)  \big) \, \le \,
 \frac{\beta}{2\S(f)} \,  \E\left(  \int_\R \big(f-g_\infty) d\nu_N\right)^2 .
$$
Since both $f, g_\infty \in \Co^1(\R)$, up to an exponentially small error term that we shall neglect,
we see that for any $\epsilon>0$,
\beq\label{bound_6}
\mathrm{W}_2^2\big( X_N^\infty , \mathcal{X}_N(f)  \big) \,\ll\,
 \E\left( \ind_{\mathcal{B}_\epsilon} \int_\R \big(f-g_\infty\big) d\nu_N\right)^2 .
\eeq
We claim that it follows from the square-root vanishing of the equilibrium measure at the edges that 
for all $j \in\lbrace 1,\ldots ,N\rbrace$, 
\beq \label{gamma_estimate}
\rho_j - \rho_0 \, \gg  \,  \Big (\frac jN \Big)^{\frac 23}
\qquad\text{and}\qquad
\rho_N - \rho_{N-j} \,  \gg  \,  \Big (\frac jN \Big)^{\frac 23} .
 \eeq
This shows that  for any $\lambda \in\mathcal{B}_\epsilon$, $\lambda_j\in\J$ for all
$\widehat{j} \ge N^{3\epsilon}$.  Since $g_\infty(x)=f(x)$ for all $x\in\J$,
this implies that for any $\lambda\in\mathcal{B}_\epsilon$, 
$$
\left| \int_\R \big(f-g_\infty) d\nu_N\right|  \,\le \,  \sum_{\widehat{j} \le N^{3\epsilon}}  
   \int_{\rho_{j-1}}^{\rho_j}\left| \big(f-g_\infty\big)(\lambda_j) -  \big(f-g_\infty\big)(x)\right|  \mu_V(dx) .
$$
Like in the proof of Lemma~\ref{le:linstatbound} (see in particular the estimate \eqref{rigidity_1}),
we obtain
\beqs \begin{split}
\left| \int_\R \big(f-g_\infty\big) d\nu_N\right| 
& \, \ll \,  \big\|f'-g_\infty'\big\|_{L^\infty(\R)}
\sum_{\widehat{j} \leq N^{3\epsilon}} \Big\{ \widehat{j}^{-\frac 13}N^{- \frac 23+\epsilon}
    +(\rho_j - \rho_{j-1})\Big\} \\
& \, \ll_f \,     N^{- \frac 23 +4\epsilon} + (\rho_{\lfloor{N^{3\epsilon}}\rfloor}-\rho_0)
+  (\rho_N - \rho_{N-\lceil{N^{3\epsilon}}\rceil} ) . 
\end{split} \eeqs
By \eqref{gamma_estimate}, this error term is of order $N^{-\frac 23 +5\epsilon}$ and by \eqref{bound_6}, we conclude that 
$$
\mathrm{W}_2^2\big( X_N^\infty , \mathcal{X}_N(f)  \big) \,\ll_f \,  N^{- \frac 43+10\epsilon} .
$$
Since the parameter $\epsilon>0$ is arbitrary, this completes the proof.
\end{proof}

\begin{lemma} \label{thm:truncation}
If $d^{\frac{\kappa+1}{2}} \le N$, then we have for any $\epsilon>0$,
$$
\mathrm{W}_2^2\big( X_N^d , X_N^\infty \big)
 \, \ll \,  d^{2\eta_*- \kappa-1} N^{\epsilon}  
$$
where as above $\eta_* := \frac {4(\kappa+1)}{2\kappa-1}$.   
\end{lemma}

\begin{proof} By Theorem~\ref{th:basis}, since  $\sum_{n\ge 1} |\widehat{f}_n| < \infty$ and 
$\| \phi_n \|_{\infty,\R} \ll 1$ where the implied constant is independent of $n \geq 1$ we have
$$
X_N^\infty - X_N^d  \, = \,  \sqrt{\frac{\beta}{2\S(f)}} \sum_{n>d} \widehat{f}_n \int_\R  \phi_n(x)  \nu_N(dx)
$$
so that
\beq \label{bound_1}
\mathrm{W}_2^2\big( X_N^d ; X_N^\infty\big)
 \, \le \,  \frac{\beta}{2\S(f)}  \E\left(\sum_{n>d} \widehat{f}_n \int  \phi_n(x)  \nu_N(dx)\right)^2 .
\eeq
Note that by Lemma~\ref{le:linstatbound} and the estimate \eqref{eq:phinderb}, conditionally
on the event $\mathcal{B}_\epsilon$, we have for all $n \geq 1$
$$
 \left| \int \phi_n(x)  \nu_N(dx) \right|  \, \ll \,  N^\epsilon  \sigma_n^\eta . 
$$
This bound turns out to be useful when $n \ll N^{\frac 1\eta}$, otherwise it is better to use the trivial bound
$$
 \left| \int  \phi_n(x)  \nu_N(dx) \right| \, \ll \,  N .
$$
 From these two estimates, we deduce that 
$$
 \E\left(\sum_{n>d} \widehat{f}_n \int  \phi_n  d\nu_N\right)^2  \, \ll \,
   \bigg(N^\epsilon \sum_{n=d+1}^{N^{\frac 1\eta}}  \widehat{f}_n \sigma_n^\eta 
      +  N  \sum_{n>N^{\frac 1\eta}} \widehat{f}_n \bigg)^2 , 
$$
up to an error term coming from the complement of the event $\mathcal{B}_\epsilon$ which is exponential in $N$ and that we have neglected. 
Now, using the fact that
$|\widehat{f}_n| \ll  \sigma_n^{- \frac {\kappa+3}{2}}$ and $\sigma_n \asymp n$, we obtain:
$$
 \E\left(\sum_{n>d} \widehat{f}_n \int  \phi_n  d\nu_N\right)^2  \, \ll \,  
N^{2\epsilon} d^{2\eta - \kappa-1}  + N^{2- \frac{\kappa+1}{\eta}}
$$
Note that if $d\to\infty$ as $N\to\infty$ and  $\eta<\frac{\kappa+1}{2}$, both error terms converge to zero and the first term  is the largest in the regime $d^\eta \le N$. Then, the claim follows directly from the estimate \eqref{bound_1} and the fact that  $\epsilon>0$ and
$\eta_* =  \frac {4(\kappa+1)}{2\kappa-1} <  \eta < \frac{\kappa+1}{2}$ are arbitrary (the previous inequalities are possible only if $\kappa\ge 5$). 
\end{proof}

\subsection{Proof of Theorem~\ref{th:main}} \label{sect:proof_main}
By the triangle inequality, we have for any $d\in\N$,
$$
\mathrm{W}_2\big (\mathcal{X}_{N}(f), \gamma_1 \big )
\, \le \,  \mathrm{W}_2\big(\mathcal{X}_{N}(f), X_N^\infty \big) 
+ \mathrm{W}_2\big(X_{N}^d(f), X_N^\infty\big)  +  \mathrm{W}_2\big(X_{N}^d(f),\gamma_1\big).
$$
Note that by choosing
$$
d  \, = \,  \lfloor N^{\frac{2}{6\eta_* +\kappa+1}} \rfloor , 
$$
the error terms in Lemmas~\ref{thm:cvg} and~\ref{thm:truncation} match and are of order 
$N^{-2\tilde\theta+\epsilon}$ for any $\epsilon>0$ where
$$
\tilde\theta(\kappa) \, = \, \frac{\kappa+1-2\eta_*}{6\eta_*+\kappa+1}
$$  
and  $ \eta_*  \, = \,    \frac {4(\kappa+1)}{2\kappa-1}  $.
In particular, we easily check that for any $\kappa\ge 5$, 
$ \tilde\theta(\kappa) \, = \,  \frac{2\kappa-9}{2\kappa+11} > 0 $.
Compared with the error term in Lemma~\ref{thm:edge_truncation},  coming from the fluctuations of the particles near the edge of the spectrum,  we conclude that 
$$
\mathrm{W}_2\left(\mathcal{X}_{N}(f), \gamma_1 \right)
 \, \ll \,  N^{-\min\{ \tilde\theta, \frac{2}{3}\} +\epsilon}, 
$$
which completes the proof. Finally, observe that if $\operatorname{supp}(f) \subset \J$, then $g_\infty = f$ so that ${\mathrm{W}_2\big(\mathcal{X}_{N}(f), X_N^\infty \big)  =0}$ 
for all $N\geq 1$. This means that in this case, the rate of convergence is given by
$$
\mathrm{W}_2\big (\mathcal{X}_{N}(f), \gamma_1 \big)
 \, \ll \,  N^{- \tilde\theta +\epsilon} .  
$$
Since $\lim_{\kappa\to\infty}\tilde\theta(\kappa) =1$, when $f$ is smooth, 
this establishes the second claim of Remark~\ref{rk:rate}.
The proof of Theorem~\ref{th:main} is therefore complete.
\qed

\appendix
\section{Properties of the Hilbert transform, Chebyshev polynomials, and the equilibrium measure}\label{app:hilbert}

In this appendix we recall some basic properties of the standard Hilbert transform on $\R$ as well as some basic facts about expanding functions on $\J$ in terms of Chebyshev polynomials. We also discuss some properties related to the equilibrium measure $\mu_V$. We normalize the Hilbert transform of a function $f\in L^p(\R)$ for some $p>1$ in the following way:
\beq\label{eq:Hilbdef}
(\mathcal{H}_x f) \, = \, \pv\int_{\R}\frac{f(y)}{x-y} \, dy , 
\eeq
where $\pv$ means that we consider the integral as a Cauchy principal value integral. A basic fact about the Hilbert transform is that it is a bounded operator on $L^p(\R)$: for finite $p>1$, there exists a finite constant $C_p>0$ such that for any $f\in L^p(\R)$, $\mathcal{H}f\in L^p(\R)$ and 
\beq\label{eq:Hilbboun}
{\Vert\mathcal{H}f\Vert}_{L^p(\R)} \, \leq \, C_p {\Vert f\Vert}_{L^p(\R)}.
\eeq
See e.g.~\cite[Theorem 4.1.7]{Grafakos} or \cite[Theorem 101]{Titchmarsh} for a proof of this fact. With some abuse of notation, if $\mu$ is a measure whose Radon-Nikodym derivative with respect to the Lebesgue measure is in $L^p(\R)$ for some $p>1$, we write $\mathcal{H}(\mu)$ for the Hilbert transform of $\frac{d\mu}{dx}$.

Another basic property of the Hilbert transform that we will make use of is its anti self-adjointness: if $f\in L^p(\R)$ and $g\in L^q(\R)$, where $\frac 1p + \frac 1q = 1$, then 
\beq\label{eq:antiself}
\int_\R f(x)(\mathcal{H}_x g)dx \, = \, -\int_\R (\mathcal{H}_x f) g(x)dx.
\eeq 
For a proof, see e.g.~\cite[Theorem 102]{Titchmarsh}.

One of the main objects we study in Section~\ref{sec:basis} is the (weighted) finite Hilbert transform 
$$
\Ucal_x \phi \, = \, -\mathcal{H}_x\left(\phi\varrho\right).
$$
Since one can check that $\mathcal{H}_x \varrho=(x^2-1)^{-\frac 12}\mathbf{1}\lbrace |x|>1\rbrace$, one can actually write for $x\in \J$
\beq \label{U_0}
  \U_x(f ) \, = \,   \int_\J \frac{f(t)-f(x)}{t-x} \, \varrho(dt)  . 
\eeq

We will need to consider derivatives of $\U_x f$ and other similar functions, and  for this, we record the following simple result.

\begin{lemma} \label{lem:derivative}
 Let $\mathrm{I} \subseteq \R$  be a closed interval, $\mu$ be a probability measure with compact support, $\operatorname{supp} \mu \subseteq \mathrm{I}$,  and $\phi\in\Co^{k,\alpha}(\mathrm{I})$ for some $0<\alpha \le 1$. Then for almost every $x\in\R$, 
\beq \label{diff_1}
\left(\frac{d}{dx}\right)^k  \int_\mathrm{I} F_0(x,t) \mu(dt) \, = \,  k! \int_\mathrm{I} F_k(x,t) \mu(dt) 
\eeq
where $F_k(x,t) = \frac{\phi(t)- \Lambda_k[\phi](x,t) }{(t-x)^{k+1}} $; see~\eqref{F_function}. In addition, if $\alpha=1$, \eqref{diff_1} holds for all $x\in\mathrm{I}$ and the RHS is continuous. 
\end{lemma}

\begin{proof}
There is nothing to prove when $k=0$ and since $\mu$ has compact support the integral exists.  By induction, it suffices to show that  for almost every $x\in\mathrm{I}$, 
\beq \label{diff_2}
\frac{d}{dx}  \int_\mathrm{I} F_{k-1}(x,t) \mu(dt)  \, = \,  k \int_\mathrm{I} F_k(x,t) \mu(dt). 
\eeq
Plainly, the function $x \mapsto F_{k-1}(x,t)$ is differentiable for any $t\in\mathrm{I}$ and it is easy to check that its derivative $F_{k-1}'(x,t) = k F_k(x,t)$. 
Note that for any $u,v\in\mathrm{I}$ with $u<v$,  the function $(t,x) \mapsto |x-t|^{\alpha-1}$ is integrable on $\R \times [u,v]$ with respect to $\mu \times dx$, so that by \eqref{estimate_1} and Fubini's theorem,
\beqs \begin {split}
\int_u^v dx \int_\mathrm{I} F_k(x,t) \mu(dt) 
  & \, = \,  \frac{1}{k} \int_\mathrm{I}\mu(dt) \int_u^v F_{k-1}'(x,t) dx \\
& \, = \,  \frac{1}{k}  \left(  \int_\mathrm{I} F_{k-1}(v,t) \mu(dt)  -  \int_\mathrm{I} F_{k-1}(u,t) \mu(dt)\right)  . 
\end {split} \eeqs
By the Lebesgue differentiation theorem, we conclude that \eqref{diff_2} holds.
Finally, note that when $\alpha=1$, $\sup_{x,t \in \mathrm{I}}\big| F_k(x,t) \big| \ll  1$,
so that the function on the RHS of \eqref{diff_2} is continuous on $\R$.
\end{proof}

While it might be possible to prove many of the properties of $\Ucal$ and $\Rcal^S$ 
(see Section~\ref{sect:eigenfunctions} for the definition) through properties of $\mathcal{H}$ directly, we found it more convenient to expand elements of $\Hi_{\mu_V}$ in terms of Chebyshev polynomials. We recall now the definitions of these. For any $k\in\N$, 
the $k$-th Chebyshev polynomial of the first kind, $T_k(x)$, and the $k$th Chebyshev polynomial of the second kind, $U_k(x)$, are the unique polynomials for which 
\beq \label{eq:chebydef}
T_k(\cos \theta) \, = \, \cos (k\theta) \quad \mathrm{and}
   \quad U_k(\cos\theta) \, = \, \frac{\sin((k+1)\theta)}{\sin \theta} \quad \mathrm{for} \quad \theta\in[0,2\pi].
\eeq
Note that in particular that $T_0'=0$ and for $k\geq 1$
\beq \label{eq:chebyder}
T_k'(x) \,  = \, kU_{k-1}(x).
\eeq
These satisfy the basic orthogonality relations (easily checked using orthogonality of Fourier modes):
\beq\label{eq:chebyortho}
\int_{\J} T_k(x)T_l(x)\varrho(dx) \, = \, \begin{cases}
\delta_{k,l}, & k=0\\
\frac{1}{2}\delta_{k,l}, & k\neq 0
\end{cases}\quad \mathrm{and} \quad \int_{\J} U_k(x)U_l(x)\scl(dx)=\delta_{k,l},
\eeq
where we used the notation of \eqref{eq:sclvarrho}.
We consider now expanding functions in $L^p(\J,\varrho)$ for $p>1$ in terms of series of Chebyshev polynomials of the first kind. We define for $f\in L^p(\J,\varrho)$, 
\beq\label{eq:chebycoef}
f_k \, = \, \begin{cases}
\int_\J f(x)\varrho(dx), & k=0 ,\\
2\int_\J f(x)T_k(x)\varrho(dx), & k>0.
\end{cases}
\eeq
The following result is a direct consequence of the standard $L^p$-theory of Fourier series
-- see e.g.~\cite[Section 3.5]{Grafakos} for the basic results about $L^p$ convergence of Fourier series.

\begin{lemma}\label{le:cLp}
Let $f\in L^p(\J,\varrho)$ for some $p>0$. Then as $N\to\infty$, $\sum_{k=0}^N f_k T_k\to f$, where the convergence is in $L^p(\J,\varrho)$.
\end{lemma}

We will also need to know how $T_k$ and $U_k$ behave under suitable Hilbert transforms.

\begin{lemma}\label{le:Ucheb}
For $k\geq 1$ and $x\in \J = (-1,1)$, $\Ucal_x T_k=U_{k-1}(x)$ and 
$ \mathcal{H}_x [U_{k-1}\scl] = 2 T_k $.
\end{lemma}

\noindent For a proof, see e.g. the discussion around \cite[Section 4.3: equation (22)]{Tricomi}.
We also recall a closely related result, namely Tricomi's formula for inverting the finite Hilbert
transform (i.e. the Hilbert transform of functions on $\J$). As it is in a slightly different form than the standard one (see \cite[Section 4.3]{Tricomi}), we  provide a proof.
\begin{lemma}[Tricomi's formula] \label{lem:Tricomi}
For any $\phi \in \Hi$, we have 
$\phi =\frac{1}{2} \mathcal{H}\big( \U(\phi)\scl)$. 
\end{lemma}

\begin{proof}
Recall that the Hilbert transform of the semicircle law is given by
\beq  \label{H_sc}
\mathcal{H}_x(\scl) \, = \,  2\big ( x - \ind_{|x|>1}\sqrt{x^2-1} \, \big) . 
\eeq
For any function $f\in L^p(\R)$ and $g\in L^q(\R)$ with $\frac 1p + \frac 1q <1$,
the Hilbert transform  satisfies the following {\it convolution property}: 
$$ 
\mathcal{H}\big(f \mathcal{H}(g)  + \mathcal{H}(f)g \big) \, = \, \mathcal{H}(f)\mathcal{H}(g)- \pi^2fg 
$$
see \cite[formula (14) in Section 4.2]{Tricomi}.
We can apply this formula with $f= \frac{d\scl}{dx}$ and 
$g = \phi \frac{d\varrho}{dx}$ when $\phi\in \Hi$. Indeed, since $\phi\in \Co(\bar{\J})$ (c.f.~the estimate \eqref{continuity_1}), the functions $f\in L^p(\R)$ for any $p>1$ and  $g\in L^q(\R)$ for any  $1<q<2$. Moreover, note that both $f$ and $g$ have support on $\bar\J$ and  that $\pi^2 fg = 2\phi$. Now, since $\mathcal{H}(g) = - \mathcal{U}(\phi)$, by formula \eqref{H_sc}, this implies that for almost every $x\in \J$,
\beq \label{H_U}
- \mathcal{H}_x\left(f(t)\mathcal{U}_t(\phi) - 2t g(t) \right) 
 \, = \,  2x  \mathcal{H}_x(g) - 2\phi(x)   .
\eeq
Since $g=\phi \frac{d\varrho}{dx}$ where $\phi\in\Hi$, we obtain
$$
 \mathcal{H}_x\big( t g(t) \big)  = \pv \int_\J \frac{tg(t)}{x-t} \, dt 
  \, = \,   - \int_\J \phi(t) \varrho(dt) +  x \mathcal{H}_x(g) 
$$
and the first term of the RHS vanishes by definition of the Sobolev space $\Hi$. We conclude  that 
$  \mathcal{H}_x\big( t g(t) \big)   = x\mathcal{H}_x(g)  $ 
and deduce from \eqref{H_U} that for all $x\in \J$,
$$
\phi(x) \,= \, \frac{1}{2} \mathcal{H}_x\big( \U(\phi)\scl)  \, = \,  \pv \int_\J \frac{\U_t(\phi)}{x-t} \frac{\sqrt{1-t^2}}{\pi} \, dt   .
$$
\end{proof}

We will also make use of the following result, whose proof is likely to exist somewhere in the literature, 
but we give one for completeness.

\begin{lemma} \label{lemma:mean}
For any function $f \in \Co^2(\bar\J)$ for which $\int_\J f(x)\varrho(dx)=0$,   we have
\beq \label{mean_3}
\int_\J x \, \mathcal{U}_x(f) \varrho(dx) \, = \, \frac{f(1)+f(-1)}{2} \, .
\eeq
\end{lemma}

\begin{proof}
We begin by showing that \eqref{mean_3} holds when $f$ is a Chebyshev polynomial 
of the first kind. By Lemma~\ref{le:Ucheb}, it suffices to show that for any $k\ge 1$,
\beq \label{mean_4}
\int_\J x U_{k-1}(x) \varrho(dx) \, = \, \frac{T_k(1)+T_k(-1)}{2} \,  .
\eeq
Using the identity, $ x U_{k-1}(x) = U_k(x) - T_k(x)$ and
\eqref{eq:chebyortho}, we obtain for any  $k\ge 1$,
$$
\int_\J x U_{k-1}(x)\varrho(dx) \, = \,  \int_\J U_k(x)\varrho(dx).
$$
Another standard fact about the Chebyshev polynomials of the second kind is that  
$$
U_{2k}(x) \, = \, 2\sum_{j=0}^k T_{2j}(x)-1 
   \quad \mathrm{and} \quad U_{2k+1}(x) \, = \, 2\sum_{j=0}^k T_{2j+1}(x),
$$
which combined with \eqref{eq:chebyortho} implies that $\int_\J U_k(x)\varrho(dx)=\frac{1}{2}(1+(-1)^k)=\frac{1}{2}(T_k(1)+T_k(-1))$, concluding the proof of \eqref{mean_4} -- the claim in case $f$ is a Chebyshev polynomial of the first kind.

Now, we shall prove that formula \eqref{mean_3} holds for arbitrary functions $f \in \Co^2(\bar\J) $ by approximating $\mathcal{U}(f)$ by $\mathcal{U}^N(f) = \sum_{k=1}^N f_k U_{k-1}$. 
By definition,
$$
f_k \,  = \,  \frac{2}{\pi} \int_0^\pi f(\cos\theta) \cos(k\theta) d\theta 
   \,  = \, - \frac{2}{\pi k^2}\int_0^\pi f(\cos\theta) \, \frac{d^2\cos(k\theta)}{d\theta^2} \, d\theta   
$$
and, an integration by parts shows that $|f_k| \ll k^{-2}$ where the implied constant  depends only on the function $f \in \Co^2(\bar\J)$. In particular, this estimate shows that the series $f= \sum_{k\ge 1} f_k T_k$  converges uniformly on $\bar{\J}$. Thus, we obtain
\beq  \begin {split} \label{U_bound_3}
\lim_{N\to\infty}\int_\J x \, \mathcal{U}^N_x(f)\varrho(dx)
  & \, = \, \lim_{N\to\infty}\sum_{k=1}^N f_k \, \frac{T_k(1)+T_k(-1)}{2} \\
  & \, = \, \frac{f(1)+f(-1)}{2} \, . \\
\end {split} \eeq
On the other-hand, using the results of Section~\ref{sect:preparation} and Section \ref{sect:eigenfunctions}, we immediately  see that formula \eqref{mean_3} follows from \eqref{U_bound_3} and the dominated convergence theorem. 
The lemma is established.
\end{proof}

We conclude this appendix with an estimate concerning the equilibrium measure (or its Hilbert transform) we need in Section~\ref{sec:approxev}. Recall that we denote by $S$ the density of the equilibrium measure $\mu_V$ with respect to the semicircle law $\scl$. 
Here, we assume that $S$ exists and $S(x)>0$ for all $x\in \bar\J$. Then, it is proved e.g.~\cite[Theorem 11.2.4]{PS11} that, if $V$ is sufficiently smooth, then $S= \frac{1}{2}\U(V')$, so that we may extend $S$ to a continuous function which  is given by
$$
 S(x) \, = \,  \frac{1}{2} \int_\J \frac{V'(t)-V'(x)}{t-x} \, \varrho(dt) ,
$$
Moreover, if the potential $V\in \Co^{\kappa+3}(\R)$, we obtain by Lemma~\ref{lem:derivative}
that $S\in  \Co^{\kappa+1}(\R)$. 
Observe that by Tricomi's formula, for all $x\in\J$,
$$
\mathcal{H}_x(S\scl) = V'(x) .
$$
This establishes the first of the variational conditions  \eqref{V_condition} and the next proposition shows that the second condition holds as well. 

\begin{proposition} \label{thm:s}
The function $\mathcal{H}(\mu_V)\in\Co^\kappa(\R)$ and 
\beq \label{s_1}
\mathcal{H}_x(\mu_V) \, = \,  V'(x) - 2S(x)\Re\big(\sqrt{x^2-1} \, \big) +  \underset{x\to1}{o}(x-1)^{\kappa} . 
\eeq
 \end{proposition}
    
\begin{proof}
We have just seen that $\mathcal{H}(\mu_V)=V'$ on $\J$ so that \eqref{s_1} holds (without the error term) for all $x\in\J$.
Define for all $x\in\R$,
\begin{equation*}
\Phi(x) \, = \,   \int_\J \frac{S(t)-S(x)}{t-x} \, d\scl(t)  - 2xS(x) . 
\end{equation*}
 Since $d\mu_V = S d\scl$, by \eqref{H_sc},  we obtain for all $x\in\R$,
$$
\mathcal{H}_x(\mu_V) \,  = \, -\Phi(x)  - 2\Re\big(\sqrt{x^2-1} \, \big) S(x) .
$$     
Since the density $S\in\Co^{\kappa+1}(\R)$, by Lemma~\ref{lem:derivative} the function $\Phi\in\Co^{\kappa}(\R)$ and $\Phi = -V'$ on $\J$. By continuity, this implies that
$$
\Phi(x) \, = \,  -V'(x) + \underset{x\to1}{o}(x-1)^{\kappa} .
$$
This completes the proof of \eqref{s_1}. 
\end{proof}


\begin{thebibliography}{99}
\bibitem {AGZ10}
G. Anderson, A. Guionnet, O. Zeitouni.
An introduction to random matrices.  Cambridge Studies in Advanced Mathematics, 118. \textit{Cambridge University Press} 2010. 

\bibitem {BGL14} 
D. Bakry, I. Gentil, M. Ledoux. Analysis and geometry of Markov diffusion operators. Grundlehren der Mathematischen Wissenschaften, 348. \textit{Springer} 2014. 

\bibitem{BFG13}
F. Bekerman, A. Figalli, A. Guionnet. Transport maps for $\beta$-matrix models and universality. \textit{Comm. Math. Phys.}, 338(2) (2013), 589--619.


\bibitem{BLS17}
F. Bekerman, T. Lebl\'e, S. Serfaty. CLT for fluctuations of $\beta$-ensembles with general potential.
\textit{Preprint} arXiv:1706.09663 (2017)

\bibitem{BG13}
G. Borot, A. Guionnet. Asymptotic expansion of $\beta$ matrix models in the one-cut regime. 
\textit{Comm. Math. Phys.} 317 (2013), 447--483.

\bibitem{BG2} G. Borot, A. Guionnet. Asymptotic expansion of beta matrix models in the multi-cut regime. \textit{Preprint} arXiv:1303.1045 (2013).

\bibitem{BEY14} P. Bourgade, L. Erd\H{o}s, H.-T. Yau.  Edge universality of beta ensembles. 
\textit{Comm. Math. Phys.} 332 (2014), 261--353.

\bibitem{BD17} 
J. Breuer, M. Duits. Central Limit Theorems for biorthogonal ensembles and asymptotics of
recurrence coefficients. \textit {J. Amer. Math. Soc.} 30 (2017), 27--66.

\bibitem {CD01}
T. Cabanal-Duvillard. Fluctuations de la loi empirique de grandes matrices aléatoires. \textit{Ann. Inst. H. Poincaré Probab. Statist.} 37 (2001), 373--402.

\bibitem {C09}
S. Chatterjee. Fluctuations of eigenvalues and second order Poincaré inequalities. 
\textit{Probab. Theory Related Fields} 143 (2009), 1--40. 

\bibitem {DS1}
C. D\"obler, M. Stolz. Stein's method and the multivariate CLT for traces of powers on the classical compact groups. \textit{Electron. J. Probab.} 16 (2011), 2375--2405. 

\bibitem {DS2}
C. D\"obler, M. Stolz. A quantitative central limit theorem for linear statistics of random matrix eigenvalues. \textit{J. Theoret. Probab.} 27 (2014), 945--953.

\bibitem {F}
J. Fulman. Stein's method, heat kernel, and traces of powers of elements of compact Lie groups. 
\textit{Electron. J. Probab.} 17 (2012), no. 66, 16 pp. 

\bibitem{Grafakos}  
L. Grafakos. Classical Fourier analysis. Second edition. Graduate Texts in Mathematics, 249
\textit{Springer} 2008. 

\bibitem {J1}
K. Johansson. On random matrices from the compact classical groups. 
\textit {Ann. of Math.}~145, 519--545 (1997).

\bibitem {J2}
K. Johansson. On fluctuations of eigenvalues of random Hermitian matrices.
\textit {Duke Math. J.}~91, 151--204 (1998).

\bibitem{KS}
T. Kriecherbauer, M. Shcherbina. Fluctuations of eigenvalues of matrix models and their applications.
\textit{Preprint} arXiv:1003.6121 (2010)

\bibitem{L} G. Lambert.
CLT for biorthogonal ensembles and related combinatorial identities.
\textit{Adv. Math.}~329, 590--648 (2018)

\bibitem{LS17}
T. Lebl\'e, S. Serfaty. 
Large Deviation Principle for Empirical Fields of Log and Riesz Gases.
\textit{Invent. Math.} 210(3), 645--757 (2017)

\bibitem {L12}
M. Ledoux. Chaos of a Markov operator and the fourth moment condition 
\textit {Ann. Probab.}~40, 2439--2459 (2012).

\bibitem {LNP15}
M. Ledoux, I. Nourdin, G. Peccati.
Stein's method, logarithmic Sobolev and transport inequalities.
\textit {Geom. and Funct. Anal.}~25, 256--306 (2015).

\bibitem {LP13}
M. Ledoux, I. Popescu. The one dimensional free Poincaré inequality 
\textit {Trans. Amer. Math. Soc.}~365, 4811--4849 (2013).

\bibitem{MH} 
J.C. Mason, D.C. Handscomb. Chebyshev polynomials. \textit{Chapman \& Hall} 2003. 

\bibitem{Meckes} 
E. Meckes. On Stein's method for multivariate normal approximation. \textit{High dimensional probability V: the Luminy volume,} 153–178, Inst. Math. Stat. (IMS) Collect., 5, \textit{Inst. Math. Statist., Beachwood, OH,} 2009. 

\bibitem {M91}
M. L. Mehta. Random matrices. Third edition. 
Pure and Applied Mathematics, 142. \textit{Elsevier/Academic Press} 2004. 

\bibitem {NP12}
I.~Nourdin, G.~Peccati. Normal approximations with Malliavin calculus. From Stein's method to universality. Cambridge Tracts in Mathematics, 192. \textit{Cambridge University Press} 2012.  

\bibitem {OV00}
F. Otto, C. Villani. Generalization of an inequality by Talagrand,
and links with the logarithmic Sobolev inequality. \textit{J. Funct. Anal.}~173, 361--400 (2000).

\bibitem{Pastur06}
L. Pastur. Limiting laws of linear eigenvalue statistics for Hermitian matrix models. \textit{J. Math. Phys.} 47, 
no. 10, 22 pp. (2006).

\bibitem{PS11}
L. Pastur, M. Shcherbina. Eigenvalue distribution of large random matrices.
Mathematical Surveys and Monographs~117. \textit{American Mathematical Society} 2011.

\bibitem{Reed_Simon_1} 
M. Reed, B. Simon. Methods of modern mathematical physics. I. Functional analysis. Second edition. \textit{Academic Press} 1980. 

\bibitem{Shcherbina13}
M. Shcherbina. Fluctuations of linear eigenvalue statistics of $\beta$ matrix models in the multi-cut regime. \textit{J. Stat. Phys.} 151 (2013), 1004--1034. 

\bibitem{Shcherbina14}
M. Shcherbina. 
Change of variables as a method to study general $\beta$-models: bulk universality. \textit{Journal of Mathematical Physics}, 55(4):043504, (2014).

\bibitem {S1}
C. Stein. The accuracy of the normal approximation to the distribution of the traces of powers of random orthogonal matrices. \textit {Department of Statistics, Stanford University, Technical Report 470} (1995).


\bibitem{Titchmarsh} E.C. Titchmarsh. Introduction to the theory of Fourier integrals. Third edition. \textit{Chelsea Publishing Co.} 1986. 

\bibitem{Tricomi} F.G. Tricomi. Integral equations. Reprint of the 1957 original. \textit{Dover Publications} 1985. 

\bibitem {V09}
C. Villani. Optimal transport. Old and new. Grundlehren der Mathematischen Wissenschaften, 338. 
\textit{Springer} 2009.

\bibitem {W16}
C. Webb. Linear statistics of the circular $\beta$-ensemble, Stein’s method,
and circular Dyson Brownian motion. \textit {Electron. J. Probab.}~21, no. 25, 16 pp. (2016). 



\end{thebibliography}
\end{document}